\documentclass[oneside]{amsart}

\newcommand{\be}{\begin{equation}}
\newcommand{\ee}{\end{equation}}
\newcommand{\bea}{\begin{equation*}\begin{aligned}}
		\newcommand{\eea}{\end{aligned}\end{equation*}}

\newcommand{\R}{\mathbb{R}}
\newcommand{\wh}{\widehat}
\newcommand{\mbb}{\mathbb}
\newcommand{\norm}[1]{\big\| #1\big\| }
\newcommand{\covsa}{\wh{\Sigma}} 

\DeclareMathOperator{\spec}{spec}
\DeclareMathOperator{\vect}{vec}

\newcommand{\eps}{\varepsilon}

\newcommand{\X}{\mathbb{X}}
\newcommand{\Wass}{\mathds{W}}

\newcommand{\half}{\frac{1}{2}}

\newcommand{\dualvar}{\gamma}

\newcommand{\m}{\mu}
\newcommand{\msa}{\wh{\m}}

\newcommand{\PP}{\mathbb{P}}
\newcommand{\QQ}{\mathbb{Q}}

\newcommand{\EE}{\mathds{E}}
\newcommand{\Ec}{\mathcal{E}}
\newcommand{\Pnom}{\wh \PP}

\newcommand{\Ac}{\mathcal{A}} 
\newcommand{\M}{\mathcal{M}}

\newcommand{\opt}{^\star}
\newcommand{\prior}{\widehat}

\newcommand{\Max}{\max\limits_}
\newcommand{\Min}{\min\limits_}
\newcommand{\Sup}{\sup\limits_}
\newcommand{\Inf}{\inf\limits_}
\newcommand{\Tr}[1]{\Trace \big[ #1 \big]}
\newcommand{\inner}[2]{\langle #1, #2 \rangle}

\newcommand{\dualx}{\dualvar_x}
\newcommand{\dualw}{\dualvar_w}

\newcommand{\cov}{\Sigma}

\newcommand{\direc}{D}
\newcommand{\lin}{L}
\newcommand{\scat}{S}
\newcommand{\cent}{\mu}
\newcommand{\slope}{A}
\newcommand{\intercept}{b}

\newcommand{\PSD}{\mathbb{S}_{+}}
\newcommand{\PD}{\mathbb{S}_{++}}

\newcommand{\eg}{{\text{\emph{e.g.}}}}

\newcommand{\risk}{R}
\newcommand{\avrisk}{\mathcal{R}}

\DeclareMathOperator{\Trace}{Tr}

\DeclareMathOperator{\st}{\,s.t.\,}

\newcommand{\Gelbrich}{\mathds{G}}
\newcommand{\G}{{\mbb G}}
\newcommand{\F}{{\mc F}}
\newcommand{\set}{\mathbb}

\DeclareFontFamily{U}{mathc}{}
\DeclareFontShape{U}{mathc}{m}{it}
{<->s*[1.03] mathc10}{}
\DeclareMathAlphabet{\mc}{U}{mathc}{m}{it}

\newcommand{\DS}{\displaystyle}
\newcommand{\RR}{\mbb R}

\input{style_PM}

\graphicspath{{graph/}}

\title[Bridging Bayesian and Minimax Mean Square Error Estimation]{Bridging Bayesian and Minimax Mean Square Error Estimation \\ via Wasserstein Distributionally Robust Optimization}

\author{Viet~Anh~Nguyen, Soroosh~Shafieezadeh-Abadeh,\\ Daniel~Kuhn, and Peyman~Mohajerin~Esfahani}
\thanks{The authors are with the Department of Management Science and Engineering, Stanford University (\texttt{viet-anh.nguyen@stanford.edu}), the Risk Analytics and Optimization Chair, EPFL, Switzerland (\texttt{soroosh.shafiee, daniel.kuhn@epfl.ch}) and the Delft Center for Systems and Control, Delft University of Technology, The Netherlands ({\tt P.MohajerinEsfahani@tudelft.nl}).}

\date{\today}

\begin{document}

\begin{abstract}
	We introduce a distributionally robust minimium mean square error estimation model with a Wasserstein ambiguity set to recover an unknown signal from a noisy observation. The proposed model can be viewed as a zero-sum game between a statistician choosing an estimator---that is, a measurable function of the observation---and a fictitious adversary choosing a prior---that is, a pair of signal and noise distributions ranging over independent Wasserstein balls---with the goal to minimize and maximize the expected squared estimation error, respectively. We show that if the Wasserstein balls are centered at normal distributions, then the zero-sum game admits a Nash equilibrium, where the players' optimal strategies are given by an {\em affine} estimator and a {\em normal} prior, respectively. We further prove that this Nash equilibrium can be computed by solving a tractable convex program.
	Finally, we develop a Frank-Wolfe algorithm that can solve this convex program orders of magnitude faster than state-of-the-art general purpose solvers. We show that this algorithm enjoys a linear convergence rate and that its direction-finding subproblems can be solved in quasi-closed form.
\end{abstract}

\maketitle
\section{Introduction}
\label{sect:intro}

Consider the problem of estimating an unknown parameter $x\in\R^n$ based on a linear measurement $y\in\R^m$ corrupted by additive noise $w \in \R^m$. This setup is formalized through the linear measurement model 
\be 
\label{eq:linear:observation}
y = H x + w,
\ee
where the observation matrix $H \in \R^{m \times n}$ is assumed to be known. We further assume that the distribution~$\PP_w$ of~$w$ has finite second moments and that~$w$ is independent of~$x$. Thus, the conditional distribution~$\PP_{y|x}$ of~$y$ given $x$ is obtained by shifting~$\PP_w$ by~$Hx$. 
We emphasize that none of the subsequent results rely on a particular ordering of the dimension~$n$ of the parameter~$x$ and the dimension~$m$ of the measurement~$y$.
The linear measurement model~\eqref{eq:linear:observation} is fundamental for numerous applications in engineering ({\em e.g.}, linear systems theory \cite{ref:golnaraghi2017automatic, ref:ogata2009modern}), econometrics ({\em e.g.}, linear regression \cite{ref:stock2015introduction, ref:wooldridge2010econometric}, time series analysis \cite{ref:chatfield2016analysis, ref:hamilton1994time}), machine learning and signal processing ({\em e.g.}, Kalman filtering \cite{ref:kay1993fundamentals, ref:murphy2012machine, ref:oppenheim2015signals}) or information theory ({\em e.g.}, multiple-input multiple-output systems \cite{ref:cover2006elements, ref:mackay2003information}) etc.
In addition, model~\eqref{eq:linear:observation} also emerges naturally in many applications in operations research such as traffic management and control~\cite{ref:lint2012applications}, inventory control~\cite{ref:aviv2003time}, advertising and promotion budgeting~\cite{ref:sriram2007optimal} or resource management~\cite{ref:rubel2017robust}.

An estimator of $x$ given $y$ is a measurable function $\psi: \R^m \rightarrow \R^n$ that grows at most linearly. Thus, there exists $C>0$ such that $|\psi(y)|\le C(1+\|y\|)$ for all $y\in\R^m$. The function value $\psi(y)$ is the prediction of $x$ based on the measurement $y$ under the estimator $\psi$. In the following we denote the family of all estimators by~$\F$. The quality of an estimator is measured by a risk function $\risk:\F\times\R^n\rightarrow \R$, which quantifies the mismatch between the parameter $x$ and its prediction $\psi(y)$. A popular risk function is the mean square error~(MSE) 
\begin{align*} 
	\risk (\psi, x) =  \EE_{\PP_{y|x}} \left[ \| x - \psi(y)\|^2 \right],
\end{align*}
which defines the estimation error as the expected squared Euclidean distance between $\psi(y)$ and $x$. If $x$ was known, then $\risk (\psi, x)$ could be minimized directly, and the constant estimator $\psi\opt(y)\equiv x$ would be optimal. In practice, however, $x$ is unobservable. Otherwise there would be no need to solve an estimation problem in the first place. With $x$ unknown, it is impossible to minimize the MSE directly. The statistics literature proposes two complementary workarounds for this problem: the Bayesian approach and the minimax approach. 

The Bayesian statistician treats $x$ as a random vector governed by a {\em prior} distribution~$\PP_x$ that captures her beliefs about $x$ before seeing~$y$ \cite[\S~11.4]{ref:kay1993fundamentals} and solves the minimum MSE (MMSE) estimation problem
\begin{align}\label{eq:Bayes} 
	\mathop{\text{minimize}}_{\psi\in\F} ~ \EE_{\PP_x} \left[ \risk(\psi, x) \right] .
\end{align}
If the distribution $\PP_x$ of~$x$ has finite second moments, then~\eqref{eq:Bayes} is solvable. In this case, the optimal estimator, which is usually termed the Bayesian MMSE estimator, is of the form $\psi\opt_{\mc B}(y)=\EE_{\PP_{x|y}} [x]$, where the conditional distribution~$\PP_{x|y}$ of~$x$ given $y$ is obtained from $\PP_x$ and $\PP_{y|x}$ via Bayes' theorem. However, the Bayesian MMSE estimator suffers from two conceptual shortcomings. First, $\psi\opt_{\mc B}$ is highly sensitive to the prior distribution~$\PP_x$, which is troubling if the statistician has little confidence in her beliefs. Second, computing $\psi\opt_{\mc B}$ requires precise knowledge of the noise distribution $\PP_w$, which is typically unobservable and thus uncertain at least to some extent. Moreover, $\psi\opt_{\mc B}$ may generically have a complicated functional form, and evaluating $\psi\opt_{\mc B}(y)$ to high precision for a particular measurement $y$ ({\em e.g.}, via Monte Carlo simulation) may be computationally challenging if the dimension of $x$ is high.

These shortcomings are mitigated if we restrict the space $\F$ of all measurable estimators in~\eqref{eq:Bayes} to the~space
\begin{align} \label{eq:affine}
	\Ac = \big\{ \psi \in\F\;:\; \exists\slope \in \R^{n \times m}, \, \intercept \in \R^n \text{ with } \psi(y) = \slope y + \intercept ~\forall y\in\R^m  \big\}
\end{align}
of all {\em affine} estimators. In this case the distributions $\PP_x$ and $\PP_w$ need not be fully known. Instead, in order to evaluate the optimal affine estimator $\psi_{\mc A}\opt(y)= A\opt y+b\opt$, it is sufficient to know the mean vectors $\mu_x$ and $\mu_w$ as well as the covariance matrices $\cov_x$ and $\cov_w$ of the distributions $\PP_x$ and $\PP_w$, respectively. If $H \cov_x H^\top + \cov_w \succ 0$, which is the case if the noise covariance matrix has full rank, then the coefficients of the best affine estimator can be computed in closed form. Using~\eqref{eq:linear:observation} together with the independence of $x$ and $w$ one can show that 
\begin{align}
	\label{eq:optimal:affine}
	\slope\opt = \cov_x H^\top (H \cov_x H^\top + \cov_w)^{-1} \quad\text{and}\quad 
	\intercept\opt = \mu_x - \slope\opt (H \mu_x + \mu_w).
\end{align}
If the random vector $(x,y)$ follows a normal distribution, then the best affine estimator is also optimal among all measurable estimators. In general, however, we do not know how much optimality is sacrificed by restricting attention to affine estimators. Moreover, the uncertainty about $\PP_x$ and $\PP_w$ transpires through to their first- and second-order moments. As the coefficients~\eqref{eq:optimal:affine} tend to be highly sensitive to these moments, their uncertainty remains worrying.

The minimax approach models the statistician's prior knowledge concerning $x$ via a convex closed uncertainty set $\X\subseteq \R^n$ as commonly used in robust optimization. The minimax MSE estimation problem is then formulated as a zero-sum game between the statistician, who selects the estimator $\psi\in\F$ with the goal to minimize the MSE, and nature, who chooses the parameter value $x\in\X$ with the goal to maximize the MSE.\begin{subequations}
	\begin{align}
		\label{eq:minimax}
		\mathop{\text{minimize}}_{\psi\in\F} ~ \max_{x \in \X} ~ \risk(\psi, x)
	\end{align}
	By construction, any minimizer $\psi_{\M}\opt$ of~\eqref{eq:minimax} incurs the smallest possible estimation error under the worst parameter realization within the uncertainty set $\X$. For this reason $\psi_{\M}\opt$ is called a minimax estimator. Note that the MSE~$\risk(\psi, x)$ generically displays a complicated non-concave dependence on~$x$ for any fixed~$\psi$, which implies that nature's inner maximization problem in~\eqref{eq:minimax} is usually non-convex. Thus, we should not expect the zero-sum game~\eqref{eq:minimax} between the statistician and nature to admit a Nash equilibrium. However, the inner maximization problem can be convexified by allowing nature to play mixed (randomized) strategies, that is, by reformulating~\eqref{eq:minimax} as the (equivalent) convex-concave saddle point problem
	\begin{align}
		\label{eq:minimax-saddle}
		\mathop{\text{minimize}}_{\psi\in\F} \max_{\QQ_x \in \M(\X)} \EE_{\QQ_x} \left[ \risk(\psi, x) \right] ,
	\end{align}
\end{subequations}
where $\M(\X)$ stands for the family of all distributions supported on $\X$ with finite second-order moments. As $\EE_{\QQ_x} [ \risk(\psi, x)]$ is convex in $\psi$ for any fixed $\QQ_x$ and concave (linear) in $\QQ_x$ for any fixed $\psi$, while $\F$ and $\M(\X)$ are both convex sets, the zero-sum game~\eqref{eq:minimax-saddle} admits a Nash equilibrium $(\psi_{\M}\opt, \QQ_x\opt)$ under mild technical conditions. Note that $\psi_{\M}\opt$ is again a minimax estimator. Moreover, $\psi_{\M}\opt$ is the statistician's best response to nature's choice $\QQ_x\opt$ and vice versa. Using the terminology introduced above, this means that $\psi_{\M}\opt$ is the Bayesian MMSE estimator corresponding to the prior~$\QQ_x\opt$. For this reason, $\QQ_x\opt$ is usually referred to as the {\em least favorable prior}. Even though the minimax approach exonerates the statistician from narrowing down her beliefs to a single prior distribution $\QQ_x$, it still requires precise information about $\PP_w$, which may not be available in practice. On the other hand, as it robustifies the estimator against {\em any} distribution on $\X$, the minimax approach is often regarded as overly pessimistic. Moreover, as in the case of the Bayesian MMSE estimation problem, $\psi\opt_{\M}$ may generically have a complicated functional form, and evaluating $\psi\opt_{\M}(y)$ to high precision may be computationally challenging if the dimension of $x$ is high. A simple remedy to mitigate these computational challenges would be to restrict $\F$ to the family $\mc A$ of affine estimators. The loss of optimality incurred by this approximation for different choices of $\X$ is discussed in~\cite[\S~4]{ref:juditsky2018lectures} and the references therein.

In this paper we bridge the Bayesian and the minimax approaches by leveraging tools from distributionally robust optimization. Specifically, we study distributionally robust estimation problems of the form
\begin{align}
	\label{eq:dro_estimator}
	\mathop{\text{minimize}}_{\psi\in\F} \max_{\QQ_x \in \mc Q_x} \EE_{\QQ_x} \left[ \risk(\psi, x) \right] ,
\end{align}
where $\mc Q_x\subseteq \M(\R^n)$ is an {\em ambiguity set} of multiple (possibly infinitely many) plausible prior distributions of~$x$. Note that if the ambiguity set collapses to the singleton $\mc Q_x = \{\PP_{x}\}$ for some $\PP_x\in\M(\R^n)$, then the distributionally robust estimation problem~\eqref{eq:dro_estimator} reduces to the Bayesian MMSE estimation problem~\eqref{eq:Bayes}. Similarly, under the ambiguity set $\mc Q_x=\M(\X)$ for some convex closed uncertainty set $\X\subseteq \R^n$, problem~\eqref{eq:dro_estimator} reduces to the minimax mean square error estimation problem~\eqref{eq:minimax-saddle}. By providing considerable freedom in tailoring the ambiguity set $\mc Q_x$, the distributionally robust approach thus allows the statistician to reconcile the specificity of the Bayesian approach with the conservativeness of the minimax approach.

The estimation model~\eqref{eq:dro_estimator} still relies on the premise that the noise distribution~$\PP_w$ is precisely known, and this assumption is not tenable in practice. However, nothing prevents us from further robustifying~\eqref{eq:dro_estimator} against uncertainty in $\PP_w$. To this end, we define $\M(\R^{n+m})$ as the family of all joint distributions of $x$ and $w$ with finite second-order moments. Moreover, we define the {\em average risk} $\avrisk:\F \times \M(\R^{n+m}) \rightarrow \R$ through 
\[
\avrisk(\psi, \PP) = \EE_{\PP} [ \| x - \psi(H x + w) \|^2 ].
\]
If $\PP=\PP_x\times\PP_w$ for some marginal distributions $\PP_x\in\M(\R^n)$ and $\PP_w\in\M(\R^m)$, which implies that $x$ and $w$ are independent under $\PP$, and if $\PP_{y|x}$ is defined as $\PP_w$ shifted by $Hx$, then $\avrisk(\psi, \PP)=  \EE_{\PP_x} [ \risk(\psi, x) ]$. Thus, the average risk $\avrisk(\psi, \PP)$ corresponds indeed to the risk $\risk(\psi, x)$ averaged under the marginal distribution~$\PP_x$. In the remainder of this paper we will study generalized distributionally robust estimation problems of the form
\begin{align}
	\label{eq:dro}
	\mathop{\text{minimize}}_{\psi\in\F} \Sup{\QQ \in \mbb B(\Pnom)} \avrisk(\psi, \QQ) ,
\end{align}
where the ambiguity set $\mbb B(\Pnom)\subseteq \M(\R^{n+m})$ captures distributional uncertainty in both $\PP_x$ and $\PP_w$. Specifically, we will model $\mbb B(\Pnom)$ as a set of factorizable distributions $\QQ=\QQ_x\times \QQ_w$ close to a nominal distribution $\Pnom = \Pnom_x \times \Pnom_w$ in the sense that $\QQ_x$ and $\QQ_w$ are close to $\Pnom_x$ and $\Pnom_w$ in Wasserstein distance, respectively.

\begin{definition}[Wasserstein distance]
	\label{definition:wasserstein}	
	For any $d\in\mathbb N$, the type-2 Wasserstein distance between two distributions $\QQ_1,\QQ_2\in\M(\R^d)$ is defined as
	\be
	\notag
	\Wass(\QQ_1, \QQ_2) = \Inf{\pi\in\Pi(\QQ_1,\QQ_2)}  \left( \int_{\R^d \times \R^d} \norm{\xi_1 - \xi_2}^2\, \pi({\rm d}\xi_1, {\rm d} \xi_2) \right)^{\frac{1}{2}},
	\ee
	where $\Pi(\QQ_1,\QQ_2)$ denotes the set of all joint distributions or couplings $\pi\in \M(\R^d\times \R^d)$ of the random vectors $\xi_1\in\R^d$ and $\xi_2\in\R^d$ with marginal distributions $\QQ_1$ and  $\QQ_2$, respectively.
\end{definition}

The dependence of the Wasserstein distance on $d$ is notationally suppressed to avoid clutter. Note that $\Wass(\QQ_1, \QQ_2)^2$ is naturally interpreted as the optimal value of a transportation problem that determines the minimum cost of moving the distribution $\QQ_1$ to $\QQ_2$, where the cost of moving a unit probability mass from $\xi_1$ to $\xi_2$ is given by the squared Euclidean distance $\| \xi_1 - \xi_2\|^2$. For this reason, the optimization variable $\pi$ is sometimes referred to as a transportation plan and the Wasserstein distance as the earth mover's distance. 

Formally, we define the {\em Wasserstein ambiguity set} as
\begin{align}\label{eq:Ambi}
	\begin{aligned}
		\mbb B(\Pnom) = \left\{ \QQ_x \times \QQ_w:
		\def\arraystretch{1.2}
		\begin{array}{l@{\;}ll@{\;}l}
			\QQ_x &\in \M(\R^n),& \, \Wass(\QQ_x, \Pnom_x) &\leq \rho_x \\
			\QQ_w &\in \M(\R^m),& \, \Wass(\QQ_w, \Pnom_w) &\leq \rho_w
		\end{array}
		\right\},
	\end{aligned}
\end{align}
where $\Pnom_x$ and $\Pnom_w$ represent prescribed nominal distributions that could be constructed via statistical analysis or expert judgement, while the Wasserstein radii $\rho_x\ge 0$ and $\rho_w\ge 0$ constitute hyperparameters that quantify the statistician's uncertainty about the nominal distributions of $x$ and $w$. We emphasize that the distributionally robust estimation model~\eqref{eq:dro} generalizes all preceding models. Indeed, if $\rho_w=0$, then~\eqref{eq:dro} reduces to the first distributionally robust model~\eqref{eq:dro_estimator}, which in turn encompasses both the MMSE estimation problem~\eqref{eq:Bayes} (for $\rho_x=0$) and the minimax estimation problem~\eqref{eq:minimax-saddle} (for $\rho_x=\infty$) as special cases. 

The distributionally robust estimation model~\eqref{eq:dro} is conceptually attractive because the hyperparameters~$\rho_x$ and~$\rho_w$ allow the statistician to specify her level of trust in the nominal prior distribution~$\Pnom_x$ and the nominal noise distribution~$\Pnom_w$. In the remainder of the paper we will show that~\eqref{eq:dro} is also computationally attractive. This is maybe surprising because mixtures of factorizable distributions are generally not factorizable, which implies that the Wasserstein ambiguity set $\mbb B(\Pnom)$ is non-convex.

We remark that one could also work with an alternative ambiguity set of the form
\begin{align}
	\label{eq:Ambi'}
	\mbb B'(\Pnom) = \left\{ \QQ_x \times \QQ_w: \QQ_x \in \M(\R^n),~ \QQ_w \in \M(\R^m), ~\Wass(\QQ_x\times \QQ_w, \Pnom_x\times \Pnom_w) \leq \rho \right\},
\end{align}
which involves only a single hyperparameter $\rho\ge 0$ and is therefore less expressive but maybe easier to calibrate than~$\mbb B(\Pnom)$. The following lemma is instrumental to understanding the relation between~$\mbb B(\Pnom)$ and~$\mbb B'(\Pnom)$. The proof of this result is relegated to the appendix.

\begin{lemma}[Pythagoras' theorem for Wasserstein distances]
	\label{lem:wasserstein-decomposition}
	For any $\QQ_x^1,\QQ_x^2\in\M(\R^n)$ and $\QQ_w^1,\QQ_w^2\in\M(\R^m)$ we have $\Wass(\QQ_x^1\times \QQ_w^1, \QQ_x^2\times \QQ_w^2)^2 = \Wass(\QQ_x^1, \QQ_x^2)^2 + \Wass(\QQ_w^1, \QQ_w^2)^2$.
\end{lemma}

If we denote the ambiguity sets \eqref{eq:Ambi} and \eqref{eq:Ambi'} temporarily by $\mbb B_{\rho_x,\rho_w} (\Pnom)$ and $\mbb B'_\rho(\Pnom)$ in order to make their dependence on the hyperparameters explicit, then Lemma~\ref{lem:wasserstein-decomposition} implies that
\begin{align*}
	\mbb B'_\rho(\Pnom) =\bigcup_{\rho_x^2+\rho_w^2\le \rho^2} \mbb B_{\rho_x,\rho_w} (\Pnom).
\end{align*}
This relation suggests that $\mbb B'_\rho(\Pnom)$ could be substantially larger than $\mbb B_{\rho_x,\rho_w} (\Pnom)$ for any fixed $\rho,\rho_x,\rho_w\ge 0$ with $\rho_x^2+\rho_w^2= \rho^2$ and thus lead to substantially more conservative estimators.

In the following we summarize the key contributions of this paper.
\begin{enumerate}
	\item \label{contribution1} We construct a safe approximation for the distributionally robust MMSE estimation problem~\eqref{eq:dro} by restricting attention to {\em affine} estimators and by maximizing the average risk over an {\em outer} approximation of the Wasserstein ambiguity set, which is described through first- and second-order moment conditions. We then prove that this safe approximation is equivalent to a tractable convex program.
	\item \label{contribution2} We also study a {\em dual} estimation problem, which is obtained by interchanging the minimization and maximization operations in the {\em primal} problem~\eqref{eq:dro} and thus lower bounds the optimal value of~\eqref{eq:dro}. We then construct a safe approximation for this dual problem by restricting the Wasserstein ambiguity set to contain only {\em normal} distributions. Assuming that the nominal distribution is normal, we prove that this safe approximation is again equivalent to a tractable convex program.
	\item By construction, the primal and dual estimation problems are upper and lower bounded by their respective safe approximations. We prove, however, that the optimal values of the safe approximations collapse if the nominal distribution is normal. This result has three important implications. 
	\begin{enumerate}
	    \item The primal and dual estimation problems and their safe approximations are all equivalent. This implies via contributions~(\ref{contribution1}) and~(\ref{contribution2}) that both original estimation problems are tractable.
	    \item The primal estimation problem is solved by an {\em affine} estimator, and the dual estimation problem is solved by a {\em normal} distribution. In other words, we have discovered a new class of adaptive distributionally robust optimization problems for which affine decision rules are optimal. 
	    \item The affine estimator and the normal distribution that solve the primal and dual estimation problems, respectively, form a {\em Nash equilibrium} for the zero-sum game between the statistician and nature. Thus, the optimal normal distribution constitutes a {\em least favorable prior}, and the optimal affine estimator represents the corresponding {\em Bayesian MMSE estimator}.
	\end{enumerate}
	We leverage these insights to prove that the optimal affine estimator can be constructed easily from the least favorable prior without the need to solve another optimization problem.
	\item We argue that our main results remain valid if the nominal distribution is any elliptical distribution.
	\item We develop a tailor-made Frank-Wolfe algorithm that can solve the dual estimation problem orders of magnitude faster than state-of-the-art general purpose solvers. We show that this algorithm enjoys a linear convergence rate. Moreover, we prove that the direction-finding subproblems can be solved in quasi-closed form, which means that the algorithm offers a favorable iteration complexity.
\end{enumerate}

We highlight that the Wasserstein ambiguity set~\eqref{eq:Ambi} is non-convex as it contains only distributions under which the signal and the noise are independent. To our best knowledge, we describe the first distributionally robust optimization model with independence conditions that admits a tractable reformulation.
We also emphasize that some of our results hold only if the nominal distribution~$\Pnom$ is normal or elliptical. While this is restrictive, there is strong evidence that normal distributions are natural candidates for~$\Pnom$. One reason for this is that the normal distribution has maximum entropy among all distributions with prescribed first- and second-order moments~\cite[\S~12]{ref:cover2006elements}. Therefore, it has appeal as the least prejudiced baseline model. Similarly, if the parameter~$x$ in~\eqref{eq:linear:observation} is normally distributed, then a normal distribution minimizes the mutual information between~$x$ and the observation~$y$ among all noise distributions with bounded variance~\cite[Lemma~II.2]{ref:diggavi2001worst}. In this sense, normally distributed noise renders the observations least informative. Conversely, if the noise in~\eqref{eq:linear:observation} is normally distributed, then a normal distribution maximizes the MMSE across all distributions of~$x$ with bounded variance~\cite[Proposition~15]{ref:guo2011estimation}. In this sense, normally distributed parameters are the hardest to estimate. Using normal nominal distributions thus amounts to adopting a worst-case perspective.

In the following we briefly survey the landscape of existing MMSE estimation models that have a robustness flavor. Several authors have addressed the {\em minimax} MMSE estimation problem~\eqref{eq:minimax} from the perspective of classical robust optimization \cite{ref:beck2007regularization, ref:beck2007mean, ref:eldar2008minimax, ref:eldar2004linear, ref:eldar2004robust, ref:juditsky2018nearoptimality}. To guarantee computational tractability, in all of these papers the estimators are restricted to be affine functions of the measurements. In this case, the minimax MMSE estimation problem can be reformulated as a tractable semidefinite program (SDP) if the uncertainty set $\set{X}$ is an ellipsoid~\cite{ref:eldar2004linear, ref:eldar2004robust} or an intersection of two ellipsoids~\cite{ref:beck2007regularization}. Similar SDP reformulations are available if the observation matrix $H$ is also subject to uncertainty and ranges over a spectral norm ball~\cite{ref:eldar2004robust} or displays a block circulant structure, with each block ranging over a Frobenius norm ball \cite{ref:beck2007mean}. If the uncertainty set is described by an intersection of several ellipsoids, then the minimax MMSE estimation problem admits an (inexact) SDP relaxation~\cite{ref:eldar2008minimax}. Even though the restriction to affine estimators may incur a loss of optimality, affine estimators are known to be near-optimal in all of the above minimax estimation models~\cite{ref:juditsky2018nearoptimality}.

Another stream of literature investigates the {\em distributionally robust} estimation model~\eqref{eq:dro_estimator} with an ambiguous signal distribution and a crisp noise distribution. When focusing on affine estimators only, this model can be reformulated as a tractable SDP if the uncertainty in the signal distribution is characterized through spectral constraints on its covariance matrix~\cite{ref:eldar2004competitive}. This tractability result also extends to uncertain observation matrices. Similar SDP reformulations are available for the distributionally robust estimation model~\eqref{eq:dro} when both the signal and the noise distribution are ambiguous and their covariance matrices are subject to spectral constraints~\cite{ref:eldar2006robust}. Extensions to uncertain block circulant observation matrices are discussed in~\cite{ref:beck2006robust}. 

Some authors have studied less structured distributionally robust estimation problems where the signal~$x$ and the measurement~$y$ are governed by a distribution that may {\em not} obey the linear measurement model~\eqref{eq:linear:observation}. In this case, the zero-sum game between the statistician and nature admits a Nash equilibrium if nature may choose any distribution that has a bounded Kullback-Leibler divergence with respect to a {\em normal} nominal distribution~\cite{ref:levy2004robust}. Intriguingly, the (affine) Bayesian MMSE estimator for the nominal distribution is optimal in this model and thus enjoys strong robustness properties. On the downside, there is no hope to improve this estimator's performance by tuning the size of the Kullback-Leibler ambiguity set. The underlying distributionally robust estimation model also serves as a fundamental building block for a robust Kalman filter~\cite{ref:levy2012robust}. Extensions to general $\tau$-divergence ambiguity sets that contain only normal distributions are reported in~\cite{ref:zorzi2017robustness}. We emphasize that all papers surveyed so far merely derive SDP reformulations or SDP relaxations that can be addressed with general purpose solvers, but none of them develops a customized solution algorithm.

The present paper extends the distributionally robust MMSE estimation model introduced in~\cite{ref:shafieezadeh2018wasserstein}, which accommodates a simple Wasserstein ambiguity set for the distribution of the signal-measurement pairs and makes no structural assumptions about the measurement noise. Note, however, that the linear measurement model~\eqref{eq:linear:observation} abounds in the literature on control theory, signal processing and information theory, implying that there are numerous applications where the measurement noise is {\em known} to be additive and independent of the signal. As we will see in Section~\ref{sect:numerical}, ignoring this structural information may result in weak estimators that sacrifice predictive performance. In Sections~\ref{sect:approx}--\ref{sect:nash} we will further see that constructing an explicit Nash equilibrium is considerably more difficult in the presence of structural information. Finally, we describe here an accelerated Frank-Wolfe algorithm that improves the sublinear convergence rate established in~\cite{ref:shafieezadeh2018wasserstein} to a linear rate. We emphasize that, in contrast to the robust MMSE estimators derived in~\cite{ref:levy2004robust, ref:zorzi2017robustness} that are insensitive to the radii of the underlying divergence-based ambiguity sets, the estimators constructed here change with the Wasserstein radii~$\rho_x$ and~$\rho_w$. Thus, using a Wasserstein ambiguity set to robustify the nominal MMSE estimation problem has a regularizing effect and leads to a parametric family of estimators that can be tuned to attain maximum prediction accuracy. Similar connections between robustification and regularization have previously been discovered in the context of statistical learning~\cite{ref:shafieezadeh2019mass-transportation} and covariance estimation~\cite{ref:nguyen2018distributionally}.


The paper is structured as follows. Sections~\ref{sect:approx} and~\ref{sect:dual} develop conservative approximations for the primal and dual distributionally robust MMSE estimation problems, respectively, both of which are equivalent to tractable convex programs.  Section~\ref{sect:nash} shows that if the nominal distribution is normal, then both approximations are exact and can be used to find a Nash equilibrium for the zero-sum game between the statistician and nature. Extensions to non-normal nominal distributions are discussed in Section~\ref{sect:elliptical}. Section~\ref{sect:algorithm} develops an efficient Frank-Wolfe algorithm for the dual MMSE estimation problem, and Section~\ref{sect:numerical} reports on numerical results.

\textbf{Notation.} For any $A \in \R^{d \times d}$ we use $\Tr{A}$ to denote the trace and $\|A\|=\sqrt{\Tr{A^\top A}}$ to denote the Frobenius norm of $A$. By slight abuse of notation, the Euclidean norm of $v\in\R^d$ is also denoted by $\|v\|$. Moreover, $I_d$ stands for the identity matrix in $\R^{d \times d}$. For any $A,B\in\R^{d\times d}$, we use $\inner{A}{B} = \Tr{A^\top B}$ to denote the inner product and $A\otimes B$ to denote the Kronecker product of $A$ and $B$. The space of all symmetric matrices in $\R^{d\times d}$ is denoted by $\mathbb S^d$. We use $\PSD^d$ ($\PD^d$) to represent the cone of symmetric positive semidefinite (positive definite) matrices in $\mathbb S^d$. For any $A,B\in\mathbb S^d$, the relation $A\succeq B$ ($A\succ B$) means that $A-B\in\PSD^d$ ($A-B\in \PD^d$). The unique positive semidefinite square root of a matrix $A \in \PSD^d$ is denoted by $A^\half$. For any $A \in \mathbb S^d$, $\lambda_{\min}(A)$ and $\lambda_{\max}(A)$ denote the minimum and maximum eigenvalues of $A$, respectively.

\section{The Gelbrich MMSE Estimation Problem}
\label{sect:approx}

The distributionally robust estimation problem~\eqref{eq:dro} poses two fundamental challenges. First, checking feasibility of the inner maximization problem in~\eqref{eq:dro} requires computing the Wasserstein distances $\Wass(\Pnom_x,\QQ_x)$ and $\Wass(\Pnom_w,\QQ_w)$, which is \#P-hard even if~$\Pnom_x$ and~$\Pnom_w$ are simple two-point distributions, while~$\QQ_x$ and~$\QQ_w$ are uniform distributions on hypercubes \cite{ref:taskesen2019complexity}. Efficient algorithms for computing Wasserstein distances are available only if both involved distributions are discrete~\cite{ref:cuturi2013sinkhorn, ref:peyre2019computational, ref:solomon2015convolutional}, and analytical formulas are only known in exceptional cases ({\em e.g.}, if both distributions are Gaussian~\cite{ref:givens1984class} or belong to the same family of elliptical distributions~\cite{ref:gelbrich1990formula}). The second challenge is that the outer minimization problem in~\eqref{eq:dro} constitutes an infinite-dimensional functional optimization problem. In order to bypass these computational challenges, we first seek a conservative approximation for~\eqref{eq:dro} by relaxing the ambiguity set~$\mbb B(\Pnom)$ and restricting the feasible set~$\F$. We begin by constructing an outer approximation for the ambiguity set. To this end, we introduce a new distance measure on the space of mean vectors and covariance matrices.
\begin{definition}[Gelbrich distance]
	\label{def:induced_distance}
	For any $d\in\mathbb N$, the Gelbrich distance between two tuples of mean vectors and covariance matrices $(\m_1, \cov_1),(\m_2, \cov_2) \in \R^d \times \PSD^d$ is defined as
	\[
	\Gelbrich \big( (\m_1, \cov_1), (\m_2, \cov_2) \big) = \sqrt{\norm{\m_1 - \m_2}^2 + \Tr{\cov_1 + \cov_2 - 2 \left( \cov_2^\half \cov_1 \cov_2^\half \right)^\half}}.
	\]
	
\end{definition}

The dependence of the Gelbrich distance on $d$ is notationally suppressed in order to avoid clutter. One can show that $\Gelbrich$ constitutes a metric on $\R^d \times \PSD^d$, that is, $\Gelbrich$ is symmetric, non-negative, vanishes if and only if $(\m_1,\cov_1)=(\m_2,\cov_2)$ and satisfies the triangle inequality~\cite[pp.~239]{ref:givens1984class}. 

\begin{proposition}[{Commuting covariance matrices \cite[p.~239]{ref:givens1984class}}]
	\label{prop:commuting}
	If $\mu_1,\mu_2\in\R^d$ are identical and $\cov_1, \cov_2\in\PSD^d$ commute ($\cov_1 \cov_2= \cov_2 \cov_1$), then the Gelbrich distance simplifies to $\Gelbrich\big( (\m_1, \cov_1), (\m_2, \cov_2)\big) = \norm{\sqrt{\cov_1} - \sqrt{\cov_2}}$.
\end{proposition}

While the Gelbrich distance is non-convex, the squared Gelbrich distance is convex in all of its arguments.

\begin{proposition}[Convexity and continuity of the squared Gelbrich distance]
    \label{prop:gelbrich-convexity}
    The squared Gelbrich distance $\Gelbrich\big( (\mu_1, \cov_1), (\mu_2, \cov_2)\big)^2$ is jointly convex and continuous in $\mu_1 , \mu_2 \in \R^d$ and $\cov_1, \cov_2 \in \PSD^d$.
\end{proposition}
\begin{proof}[Proof of Proposition~\ref{prop:gelbrich-convexity}]
    By \cite[Proposition~2]{ref:malago2018wasserstein}, the squared Gelbrich distance $\Gelbrich\big( (\mu_1, \cov_1), (\mu_2, \cov_2)\big)^2$ coincides with the optimal value of the semidefinite program
    \begin{equation*}
        \begin{array}{cl}
             \min & \| \mu_1 - \mu_2 \|_2^2 + \Tr{\cov_1 + \cov_2 - 2 C} \\[1ex]
             \st & C \in \RR^{d \times d} \\[1ex]
             & \begin{bmatrix} \cov_1 & C \\ C^\top & \cov_2 \end{bmatrix} \succeq 0,
        \end{array}
    \end{equation*}
    see also \cite[Section~3]{ref:dowson1982frechet}. Less general results that hold when one of the matrices $\cov_1$ or $\cov_2$ is positive definite are reported in~\cite{ref:givens1984class, ref:knott1984optimal, ref:olkin1982distance}.
    The convexity of the squared Gelbrich distance then follows from~\cite[Proposition~3.3.1]{ref:bertsekas2009convex}, which guarantees that convexity is preserved under partial minimization. Moreover, the continuity of the squared Gelbrich distance follows from the continuity of the matrix square root established in Lemma~\ref{lemma:hoelder}.
\end{proof}

Our interest in the Gelbrich distance stems mainly from the next proposition, which lower bounds the Wasserstein distance between two distributions in terms of their first- and second-order moments. We will later see that this bound becomes tight when $\QQ_1$ and $\QQ_2$ are normal or---more generally---elliptical distributions of the same type.

\begin{proposition}[Moment bound on the Wasserstein distance {\cite[Theorem~2.1]{ref:gelbrich1990formula}}]
	\label{prop:wasserstein}
	For any distributions $\QQ_1$,~$\QQ_2 \in \M(\R^d)$ with mean vectors $\mu_1$, $\mu_2 \in \R^d$ and covariance matrices $\cov_1$, $\cov_2 \in \PSD^d$, respectively, we have
	\begin{align} \notag
		\Wass(\QQ_1, \QQ_2) \geq \Gelbrich \big( (\m_1, \cov_1), (\m_2, \cov_2) \big).
	\end{align}
\end{proposition}

Proposition~\ref{prop:wasserstein} prompts us to construct an outer approximation for the Wasserstein ambiguity set $\mbb B(\Pnom)$ by using the Gelbrich distance. Specifically, we define the {\em Gelbrich ambiguity set} centered at $\Pnom=\Pnom_x\times \Pnom_w$ as 
\begin{align*}
	\begin{aligned}
		\G(\Pnom) = \left\{ \QQ_x \times \QQ_w:
		\def\arraystretch{1.2}
		\begin{array}{l}
			\begin{array}{l@{\;}@{\;}l@{\;}l}
				\QQ_x \in \M(\R^n), & \m_x\,= \EE_{\QQ_x}[x], & \cov_x\,= \EE_{\QQ_x} [x x^\top]  -\mu_x \mu_x^\top \\
				\QQ_w \in \M(\R^m), & \m_w=\EE_{\QQ_w}[w], & \cov_w = \EE_{\QQ_w} [w w^\top] - \mu_w \mu_w^\top
			\end{array}\\
			\begin{array}{ll}
				\Gelbrich \big( (\m_x, \cov_x), (\msa_x, \covsa_x) \big)\leq \rho_x, & \Gelbrich \big( (\m_w, \cov_w), (\msa_w, \covsa_w) \big) \leq \rho_w
			\end{array}
		\end{array}
		\right\},
	\end{aligned}
\end{align*}
where $\msa_x$ and $\msa_w$ denote the mean vectors and $\covsa_x$ and $\covsa_w$ the covariance matrices of $\Pnom_x$ and $\Pnom_w$, respectively.

\begin{corollary}[Relation between Gelbrich and Wasserstein ambiguity sets]
	\label{cor:wasserstein-in-gelbrich}
	For any $\Pnom =\Pnom_x\times \Pnom_w$ with $\Pnom_x\in  \M(\R^n)$ and $\Pnom_w\in  \M(\R^m)$ we have $\mbb B(\Pnom) \subseteq \G(\Pnom)$.
\end{corollary}
\begin{proof}[Proof of Corollary~\ref{cor:wasserstein-in-gelbrich}]
	Select any $\QQ=\QQ_x\times\QQ_w\in\mbb B(\Pnom)$ and define $\mu_x$ and $\mu_w$ as the mean vectors and $\cov_x$ and $\cov_w$ as the covariance matrices of $\QQ_x$ and $\QQ_w$, respectively. By Proposition~\ref{prop:wasserstein} we then have
	\[
	\Gelbrich \big( (\m_x, \cov_x), (\msa_x, \covsa_x) \big) \le \Wass(\QQ_x, \Pnom_x)\le \rho_x\quad\text{and}\quad 
	\Gelbrich \big( (\m_w, \cov_w), (\msa_w, \covsa_w) \big) \le \Wass(\QQ_w, \Pnom_w)\le \rho_w,
	\]
	which in turn implies that $\QQ\in\G(\Pnom)$. We may thus conclude that $\mbb B(\Pnom) \subseteq \G(\Pnom)$.
\end{proof}
By restricting $\F$ to the set $\Ac$ of all affine estimators while relaxing $\mbb B(\Pnom)$ to the Gelbrich ambiguity set $\G(\Pnom)$, we obtain the following conservative approximation of the distributionally robust estimation problem~\eqref{eq:dro}.
\begin{align}
	\label{eq:dro:approx}
	\mathop{\text{minimize}}_{\psi\in\Ac} \Sup{\QQ \in \G(\Pnom)} \avrisk(\psi, \QQ) 
\end{align}
From now on we will call~\eqref{eq:dro} and~\eqref{eq:dro:approx} the Wasserstein and Gelbrich MMSE estimation problems, and we will refer to their minimizers as Wasserstein and Gelbrich MMSE estimators, respectively. As the average risk $\avrisk(\psi, \QQ)$ of a fixed affine estimator $\psi\in\Ac$ is convex and quadratic in the mean vector $\m$ and affine in the covariance matrix $\cov$ of the distribution~$\QQ$, the inner maximization problem in~\eqref{eq:dro:approx} is non-convex. Thus, one might suspect that the Gelbrich MMSE estimation problem is intractable. Below we will show, however, that~\eqref{eq:dro:approx} is equivalent to a finite convex program that can be solved in polynomial time. To this end, we first show that, under mild conditions, problem~\eqref{eq:dro:approx} is stable with respect to changes of its input parameters.

\begin{proposition}[Regularity of the Gelbrich MMSE estimation problem] 
	\label{prop:Gelbrich}
	The Gelbrich MMSE estimation problem~\eqref{eq:dro:approx} enjoys the following regularity properties.
	\begin{enumerate}[label = $(\roman*)$]
		\item {\bf Conservativeness:} \label{prop:Gelbrich:conservative}
		Problem~\eqref{eq:dro:approx} upper bounds the Wasserstein MMSE estimation problem~\eqref{eq:dro}. 
		
		\item {\bf Solvability:} \label{prop:Gelbrich:solvability}
		If $\rho_x>0$ and $\rho_w>0$, then the minimum of~\eqref{eq:dro:approx} is attained.  
		
		\item {\bf Stability:} \label{prop:Gelbrich:cont}
		If $\rho_x>0$ and $\rho_w>0$, then the minimum of~\eqref{eq:dro:approx} is continuous in $(\msa_x,\msa_w,\covsa_x,\covsa_w)$.
		
	\end{enumerate}
\end{proposition}

The proof of Proposition~\ref{prop:Gelbrich} is lengthy and technical and is therefore relegated to the appendix. We are now ready to prove the main result of this section.

\begin{theorem}[Gelbrich MMSE estimation problem]
	\label{thm:conservative}
	The Gelbrich MMSE estimation problem~\eqref{eq:dro:approx} is equivalent to the finite convex optimization problem 
	\be
	\label{eq:VA}
	\begin{array}{cl}
		\inf 
		& \dualx \big( \rho_x^2 - \Tr{\covsa_x} \big) + \dualx^2 \inner{[\dualx I_n - (I_n- \slope H)^\top (I_n - AH)]^{-1}}{\covsa_x} \\
		& \hspace{1cm} + \dualw \big( \rho_w^2 - \Tr{\covsa_w} \big) + \dualw^2 \inner{(\dualw I_m - \slope^\top \slope)^{-1}}{\covsa_w} \\[1ex]
		\st & A \in \R^{n \times m},\quad \dualx, \dualw \in \R_+ \\ 
		& \dualx I_n - (I_n-\slope H)^\top (I_n-\slope H) \succ 0, \quad \dualw I_m - \slope^\top \slope \succ 0.
	\end{array}
	\ee
	Moreover, if $\rho_x> 0$ and $\rho_w > 0$, then~\eqref{eq:VA} admits an optimal solution\footnote{We say that $\slope\opt$ solves~\eqref{eq:VA} if adding the constraint $\slope=\slope\opt$ does not change the infimum of~\eqref{eq:VA}. Note that the infimum of the resulting problem over $(\dualx,\dualw)$ may not be attained, {\em i.e.}, the existence of a solution $\slope\opt$ does not imply that~\eqref{eq:VA} is solvable.}~$\slope\opt$, and the infimum of~\eqref{eq:dro:approx} is attained by the affine estimator $\psi\opt(y) =  \slope\opt y +\intercept\opt$, where $b\opt = \msa_x - \slope\opt (H \msa_x + \msa_w)$.
\end{theorem}

{The \textit{strict} semidefinite inequalities in~\eqref{eq:VA} ensure that the inverse matrices in the objective function exist.}

\begin{proof}[Proof of Theorem~\ref{thm:conservative}]
	Throughout this proof we denote by $\psi_{A,b}\in\Ac$ the affine estimator $\psi_{A,b}(y)=Ay+b$ corresponding to the sensitivity matrix $A \in \R^{n \times m}$ and the vector $b\in\R^n$ of intercepts. In the following we fix some $A \in \R^{n \times m}$ and define $K = I_n-\slope H$ in order to simplify the notation. By the definitions of the average risk $\avrisk(\psi, \QQ)$ and the Gelbrich ambiguity set $\G(\prior \PP)$, we then have
	\be
	\label{eq:robust}
	\inf_{\substack{b}} \Sup{\QQ \in \G(\prior \PP)} \avrisk(\psi_{A,b}, \QQ) = \left\{
	\begin{array}{ccl}
		\Inf{\intercept}& \Sup{\substack{\mu_x, \mu_w \\ \cov_x, \cov_w\succeq 0}} & \inner{K^\top K}{\cov_x + \mu_x \mu_x^\top} + \inner{\slope^\top \slope}{\cov_w + \mu_w \mu_w^\top} + \intercept^\top \intercept \\[-2ex]
		&& \hspace{1cm} - 2 \mu_x^\top K^\top \slope \mu_w - 2 \intercept^\top ( K \mu_x - \slope \mu_w ) \\[1ex]
		&\st & \Gelbrich \big( (\m_x, \cov_x), (\msa_x, \covsa_x) \big)^2 \leq \rho_x^2 \\
		&& \Gelbrich \big( (\m_w, \cov_w), (\msa_w, \covsa_w) \big)^2 \leq \rho_w^2.
	\end{array}	\right.
	\ee
	The outer minimization problem in~\eqref{eq:robust} is convex because the objective function of the minimax problem is convex in $\intercept$ for any fixed $(\mu_x, \mu_w, \cov_x, \cov_w)$ and because convexity is preserved under maximization. Moreover, the inner maximization problem in~\eqref{eq:robust} is non-convex because its objective function is convex in~$(\mu_x, \mu_w)$. This observation prompts us to maximize over $(\m_x,\m_w)$ and $(\cov_x,\cov_w)$ sequentially and to reformulate~\eqref{eq:robust}~as
	\be
	\label{eq:robust2}
	\begin{array}{cccl}
		\Inf{\intercept}& \Sup{\substack{\mu_x, \mu_w\\ \| \mu_x - \msa_x \| \leq \rho_x \\\| \mu_w - \msa_w \| \leq \rho_w}}& \Sup{\substack{ \cov_x, \cov_w\succeq 0}} & \inner{K^\top K}{\cov_x + \mu_x \mu_x^\top} + \inner{\slope^\top \slope}{\cov_w + \mu_w \mu_w^\top} + \intercept^\top \intercept \\[-4ex]
		&&& \hspace{0.2cm} - 2 \mu_x^\top K^\top \slope \mu_w - 2 \intercept^\top ( K \mu_x - \slope \mu_w ) \\[1ex]
		&&\st & \Gelbrich \big( (\m_x, \cov_x), (\msa_x, \covsa_x) \big)^2 \leq \rho_x^2 \\
		&&& \Gelbrich \big( (\m_w, \cov_w), (\msa_w, \covsa_w) \big)^2 \leq \rho_w^2.
	\end{array}	
	\ee
	As $\| \mu_x - \msa_x \| \leq  \Gelbrich ( (\m_x, \cov_x), (\msa_x, \covsa_x) )$ and as this inequality is tight for $\cov_x=\covsa_x$, the extra constraint $\| \mu_x - \msa_x \| \leq \rho_x$ is actually redundant and merely ensures that the maximization problem over $\cov_x$ remains feasible for any admissible choice of $\mu_x$. An analogous statement holds for $\m_w$ and $\cov_w$. 
	By the definition of the Gelbrich distance, the innermost maximization problem over $(\cov_x,\cov_w)$ in~\eqref{eq:robust2} admits the Lagrangian~dual
	\be 
	\label{eq:robust-inner}
	\begin{array}{c@{\;}c@{}l}
		\Inf{\dualx, \dualw \geq 0} & \Sup{\cov_x, \cov_w\succeq 0} & \inner{K^\top K}{\cov_x + \mu_x \mu_x^\top} + \inner{\slope^\top \slope}{\cov_w + \mu_w \mu_w^\top} + \intercept^\top \intercept - 2 \mu_x^\top K^\top \slope \mu_w - 2 \intercept^\top ( K \mu_x - \slope \mu_w ) \\[-0.5ex]
		&& \hspace{1cm} + \dualx \big( \rho_x^2 - \| \mu_x - \msa_x \|^2 - \Tr{\cov_x + \covsa_x - 2 \big( \covsa_x^\half \cov_x \covsa_x^\half \big)^\half} \big) \\[2ex]
		&& \hspace{1cm} + \dualw\! \big( \rho_w^2 - \| \mu_w - \msa_w \|^2 - \Tr{\cov_w + \covsa_w - 2 \big( \covsa_w^\half \cov_w \covsa_w^\half \big)^\half} \big).
	\end{array}
	\ee
	Strong duality holds by \cite[Proposition~5.5.4]{ref:bertsekas2009convex}, which applies because the primal problem has a non-empty compact feasible set. Next, we observe that the inner maximization problem in~\eqref{eq:robust-inner} can be solved analytically by using Proposition~\ref{prop:lin_obj:Lag} in the appendix, and thus the dual problem~\eqref{eq:robust-inner} is equivalent to
	\begin{align}
		\label{eq:robust-inner2}
		\begin{array}{cl}
			\Inf{\substack{\dualx,\dualw \\ \dualx I_n \succ K^\top K \\ \dualw I_m \succ \slope^\top \slope}} & \inner{K^\top K}{\mu_x \mu_x^\top} + \inner{\slope^\top \slope}{\mu_w \mu_w^\top} + \intercept^\top \intercept 
			- 2 \mu_x^\top K^\top \slope \mu_w - 2 \intercept^\top ( K \mu_x - \slope \mu_w ) \\[-4ex]
			& \hspace{1cm} + \dualx \big(\rho_x^2 - \| \mu_x - \msa_x \|^2 -\Tr{ \covsa_x} +  \dualx  \inner{(\dualx I_n - K^\top K )^{-1}}{\covsa_x} \big) \\[2ex]
			& \hspace{1cm} + \dualw \big(\rho_w^2 - \| \mu_w - \msa_w \|^2 -\Tr{ \covsa_w} + \dualw \inner{(\dualw I_m - \slope^\top \slope)^{-1}}{\covsa_w} \big).
		\end{array}
	\end{align}
	Substituting~\eqref{eq:robust-inner2} back into~\eqref{eq:robust2} then allows us to reformulate the Gelbrich MMSE estimation problem~\eqref{prop:Gelbrich}~as 		
	\begin{align}
		\label{eq:robust3}
		\begin{array}{c@{\,}l}
			\Inf{\intercept } \, \Sup{\substack{\mu_x, \mu_w \\  \| \mu_x - \msa_x \| \leq \rho_x \\ \| \mu_w - \msa_w \| \leq \rho_w}} \,
			\Inf{\substack{\dualx,\dualw \\ \dualx I_n \succ K^\top K \\ \dualw I_m \succ \slope^\top \slope }} & \inner{K^\top K}{\mu_x \mu_x^\top} 
			+ \inner{\slope^\top \slope}{\mu_w \mu_w^\top} + \intercept^\top \intercept - 2 \mu_x^\top K^\top \slope \mu_w - 2 \intercept^\top ( K \mu_x - \slope \mu_w ) \\[-4ex]
			& \hspace{1cm} + \dualx \big(\rho_x^2 - \| \mu_x - \msa_x \|^2 -\Tr{ \covsa_x} +  \dualx  \inner{(\dualx I_n - K^\top K )^{-1}}{\covsa_x} \big) \\[2ex]
			& \hspace{1cm} + \dualw \big(\rho_w^2 - \| \mu_w - \msa_w \|^2 -\Tr{ \covsa_w} + \dualw \inner{(\dualw I_m - \slope^\top \slope)^{-1}}{\covsa_w} \big).
		\end{array}
	\end{align}
	The infimum of the inner minimization problem over $(\dualx, \dualw)$ in~\eqref{eq:robust3} is convex quadratic in~$\intercept$. Moreover, it is concave in~$(\mu_x,\mu_w)$ because $K^\top K - \dualx I_n\prec 0$ and $\slope^\top \slope - \dualw I_m\prec 0$ for any feasible choice of $(\dualx,\dualw)$ and because concavity is preserved under minimization. Finally, the feasible set for~$(\mu_x,\mu_w)$ is convex and compact. By Sion's classical minimax theorem, we may therefore interchange the infimum over $\intercept$ with the supremum over $(\mu_x, \mu_w)$. The minimization problem over $\intercept$ thus reduces to an unconstrained (strictly) convex quadratic program that has the unique optimal solution $ \intercept = K \mu_x - \slope \mu_w$. Substituting this expression back into~\eqref{eq:robust3} then yields
	\begin{align}
		\label{eq:robust4}
		\begin{array}{c@{\,}l}
			\Sup{\substack{\mu_x, \mu_w \\  \| \mu_x - \msa_x \| \leq \rho_x \\ \| \mu_w - \msa_w \| \leq \rho_w}} \,
			\Inf{\substack{\dualx,\dualw \\ \dualx I_n \succ K^\top K \\ \dualw I_m \succ \slope^\top \slope }} &
			\dualx \big( \rho_x^2 - \| \mu_x - \msa_x \|^2 - \Tr{\covsa_x} \big) + \dualx^2 \inner{(\dualx I_n - K^\top K)^{-1}}{\covsa_x}  \\[-4ex]
			& \hspace{0.5cm} + \dualw \big( \rho_w^2 - \| \mu_w - \msa_w \|^2 - \Tr{\covsa_w} \big) + \dualw^2 \inner{(\dualw I_m - \slope^\top \slope)^{-1}}{\covsa_w}.
		\end{array}
	\end{align}
	It is easy to verify that the resulting maximization problem over $(\mu_x, \mu_w)$ is solved by $\mu_x=\msa_x$ and $\mu_w= \msa_w$. Substituting the corresponding optimal value into~\eqref{eq:robust} finally yields
	\[
	\inf_{\substack{b}} \Sup{\QQ \in \G(\prior \PP)} \avrisk(\psi_{A,b}, \QQ) = \left\{
	\begin{array}{c@{\,}l}
	\Inf{\substack{\dualx,\dualw \\ \dualx I_n \succ K^\top K \\ \dualw I_m \succ \slope^\top \slope }} &
	\dualx \big( \rho_x^2 - \Tr{\covsa_x} \big) + \dualx^2 \inner{(\dualx I_n - K^\top K)^{-1}}{\covsa_x}  \\[-3ex]
	& \hspace{0.5cm} + \dualw \big( \rho_w^2 - \Tr{\covsa_w} \big) + \dualw^2 \inner{(\dualw I_m - \slope^\top \slope)^{-1}}{\covsa_w}.
	\end{array}
	\right.
	\]
	From the above equation and the definition of $K$ it is evident that the Gelbrich MMSE estimation problem
	\begin{equation} \label{eq:dro:approx:2}
	\inf_{\psi\in\Ac} \Sup{\QQ \in \G(\Pnom)} \avrisk(\psi, \QQ) = \inf_{\slope, \intercept} \Sup{\QQ \in \G(\Pnom)} \avrisk(\psi_{\slope,\intercept}, \QQ)
	\end{equation}
	is indeed equivalent to the finite convex optimization problem~\eqref{eq:VA}.	
	
	Assume now that $\rho_x>0$ and $\rho_w > 0$. In this case we know from Proposition~\ref{prop:Gelbrich}~\ref{prop:Gelbrich:solvability} that the Gelbrich MMSE estimation problem~\eqref{eq:dro:approx:2} admits an optimal affine estimator $\psi^\star(y)=\slope^\star y+\intercept^\star$ for some $\slope^\star\in\R^{n\times m}$ and $b^\star \in\R^m$. The reasoning in the first part of the proof then implies that $\slope^\star$ solves~\eqref{eq:VA}. Moreover, it implies that $\intercept^\star$ is optimal in~\eqref{eq:robust} when we fix $\slope=\slope^\star$. As~\eqref{eq:robust} is equivalent to~\eqref{eq:robust3} and as the unique optimal solution of~\eqref{eq:robust3} for $\slope=\slope^\star$ is given by $\intercept = \msa_x - \slope\opt (H \msa_x + \msa_w)$, we may finally conclude that
	\[
	\intercept\opt = \msa_x - \slope\opt (H \msa_x + \msa_w).
	\]
	By reversing these arguments, one can further show that if $\slope\opt$ solves~\eqref{eq:VA} and $\intercept\opt$ is defined as above, then the affine estimator $\psi^\star(y)=\slope^\star y+\intercept^\star$ is optimal in~\eqref{eq:dro:approx:2}. This observation completes the proof.
\end{proof}

Using Schur complement arguments, the convex program~\eqref{eq:VA} can be further simplified to a standard semidefinite program (SDP), which can be addressed with off-the-shelf solvers.

\begin{corollary}[SDP reformulation] \label{cor:primal:refor}
	The Gelbrich MMSE estimation problem~\eqref{eq:dro:approx} is equivalent to the SDP
	\be
	\label{eq:program:SDP:linear}
	\begin{array}{cl}
		\inf & \dualx \big( \rho_x^2 - \Tr{\covsa_x} \big) + \Tr{U_x} + \dualw \big( \rho_w^2 - \Tr{\covsa_w} \big) + \Tr{U_w} \vspace{1mm}\\
		\st & \slope \in \R^{n \times m}, \quad \dualx, \dualw \in \R_+ \\
		& U_x \in \PSD^{n}, \quad V_x \in \PSD^{n} , \quad U_w \in \PSD^{m}, \quad V_w \in \PSD^{m} \\
		& \begin{bmatrix}
			U_x ~ & \dualx \covsa_x^\half \\
			\dualx \covsa_x^\half ~ & V_x
		\end{bmatrix} \succeq 0, \quad\,\,
		\begin{bmatrix}
			\dualx I_n - V_x ~ & ~ I_n - H^\top \slope^\top \\
			I_n - \slope H ~& ~ I_n
		\end{bmatrix} \succeq 0 \\
		& \begin{bmatrix}
			U_w ~ & \dualw \covsa_w^\half \\
			\dualw \covsa_w^\half ~ & V_w
		\end{bmatrix} \succeq 0, \quad
		\begin{bmatrix}
			\dualw I_m - V_w ~ & \slope^\top \\
			\slope ~ & I_n
		\end{bmatrix} \succeq 0.
	\end{array} 
	\ee	
\end{corollary}
\begin{proof}[Proof of Corollary~\ref{cor:primal:refor}]
	Define the extended real-valued function $h_w:\R^{n\times m}\times \R_+\rightarrow (-\infty,\infty]$ through
	\[
	h_w(A,\dualw) = \left\{ \begin{array}{ll} \dualw^2 \inner{(\dualw I_m - \slope^\top \slope)^{-1}}{\covsa_w} &\text{if } \dualw I_m - \slope^\top \slope \succ 0, \\
	\infty & \text{otherwise.} \end{array}\right.
	\]
	If $\dualw I_m - \slope^\top \slope \succ 0$, then, we have
	\begin{align}
		h_w(A,\dualw) =&\inf_{U_w\succeq 0} \left\{ \Tr{U_w}\; :\; U_w \succeq \dualw^2 \covsa_w^\half (\dualw I_m - \slope^\top \slope)^{-1} \covsa_w^\half \right\} \notag \\
		=&\inf_{U_w\succeq 0, V_w\succ 0} \left\{ \Tr{U_w}\;:\; U_w \succeq \dualw^2 \covsa_w^\half V_w^{-1} \covsa_w^\half, ~\dualw I_m - \slope^\top \slope \succeq V_w 
		\right\} \notag \\
		=&\inf_{U_w\succeq 0, V_w\succ 0} \left\{ \Tr{U_w}\;:\;
		\begin{bmatrix}
			\dualw I_m - V_w & \slope^\top \\
			\slope & I_n
		\end{bmatrix} \succeq 0, ~
		\begin{bmatrix}
			U_w & \dualw \covsa_w^\half \\
			\dualw \covsa_w^\half & V_w
		\end{bmatrix} \succeq 0 
		\right\},
		\label{eq:SDP-reformulation}
	\end{align}
	where the first equality holds due to the cyclicity of the trace operator and because $U_w \succeq \bar U_w$ implies $\Tr{U_w} \geq \Tr{\bar U_w}$ for all $U_w,\bar U_w\succeq 0$, the second equality holds because $V_w\succeq \bar V_w$ is equivalent to $V_w^{-1}\preceq \bar V_w^{-1}$ for all $V_w,\bar V_w\succ 0$, and the last equality follows from standard Schur complement arguments; see, \eg, \cite[\S~A.5.5]{ref:boyd2004convex}. If $\dualw I_m - \slope^\top \slope \nsucc 0$, on the other hand, then the first matrix inequality in~\eqref{eq:SDP-reformulation} implies that $V_w$ must have at least one non-positive eigenvalue, which contradicts the constraint $V_w\succ 0$. The SDP~\eqref{eq:SDP-reformulation} is therefore infeasible, and its infimum evaluates to $\infty$. Thus, $h_w(\slope,\dualw)$ coincides with the optimal value of the SDP~\eqref{eq:SDP-reformulation} for all $\slope\in\R^{n\times m}$ and $\dualw\in\R_+$.
	
	A similar SDP reformulation can be derived for the function $h_x:\R^{n\times m}\times \R_+\rightarrow (-\infty,\infty]$ defined through
	\[
	h_x(\slope,\dualx)= \left\{ \begin{array}{ll} \dualx^2 \inner{[\dualx I_n - (I_n - \slope H)^\top (I_n - \slope H)]^{-1}}{\covsa_x} & \text{if } \dualx I_n - (I_n - \slope H)^\top (I_n - \slope H) \succ 0,\\
	\infty & \text{otherwise.}\end{array} \right.
	\]
	The claim now follows by substituting the SDP reformulations for $h_w(\slope,\dualw)$ and $h_x(\slope,\dualx)$ into~\eqref{eq:VA}. In doing so, we may relax the strict semidefinite inequalities $V_w\succ 0$ and $V_x\succ 0$ to weak inequalities $V_w\succeq 0$ and $V_x\succeq 0$, which amounts to taking the closure of the (non-empty) feasible set and does not change the infimum of problem~\eqref{eq:VA}. This observation completes the proof.
\end{proof}
\begin{remark}[Numerical stability]
	The SDP~\eqref{eq:program:SDP:linear} requires the square roots of the nominal covariance matrices as inputs. Unfortunately, iterative methods for computing matrix square roots often suffer from numerical instability in high dimensions. As a remedy, one may replace those matrix inequalities in~\eqref{eq:program:SDP:linear} that involve $\covsa^\half_x$ and $\covsa^\half_w$ with
	\[
	\begin{bmatrix}
	U_x ~ & \dualx \Lambda_x^\top \\
	\dualx \Lambda_x ~ & V_x
	\end{bmatrix} \succeq 0 \qquad \text{and} \qquad \begin{bmatrix}
	U_w ~ & \dualw \Lambda_w^\top \\
	\dualw \Lambda_w ~ & V_w
	\end{bmatrix} \succeq 0,
	\]
	where $\Lambda_x$ and $\Lambda_w$ represent the lower triangular Cholesky factors of $\covsa_x$ and $\covsa_w$, respectively. Thus, we have $\covsa_x = \Lambda_x \Lambda_x^\top$ and $\covsa_w = \Lambda_w \Lambda_w^\top$. We emphasize that $\Lambda_x$ and $\Lambda_w$ can be computed reliably in high dimensions.
\end{remark}

\section{The Dual Wasserstein MMSE Estimation Problem over Normal 
	Priors}
\label{sect:dual}

We now examine the dual Wasserstein MMSE estimation 
problem
\begin{align}
	\label{eq:dual-dro}
	\mathop{\text{maximize}}_{\QQ \in \mbb B(\Pnom)}\Inf{\psi\in\F}  \avrisk(\psi, \QQ) ,
\end{align}
which is obtained from~\eqref{eq:dro} by interchanging the order of minimization and maximization. Any maximizer $\QQ\opt$ of this dual estimation problem, if it exists, will henceforth be called a {\em least favorable prior}. Unfortunately, problem~\eqref{eq:dual-dro} is 
generically intractable. Below we will demonstrate, however, 
that~\eqref{eq:dual-dro} becomes tractable if the nominal 
distribution~$\Pnom$ is normal.

\begin{definition}[Normal distributions]
	\label{def:normal-dist} 
	We say that $\mbb P$ is a normal distribution on $\R^d$ with mean $\mu\in\R^d$ and covariance matrix $\cov\in\PSD^d$, that is, $\mbb P=\mc N(\mu,\cov)$, if $\mbb P$ is supported on $\text{\rm supp}(\mbb P)=\{\mu+Ev: v\in\R^k\}$, and if the density function of $\mbb P$ with respect to the Lebesgue measure on $\text{\rm supp}(\mbb P)$ is given by
	\[
	\varrho_{\mbb P}(\xi) = \frac{1}{\sqrt{(2\pi)^k\det(D)}} e^{-(\xi-\mu)^\top ED^{-1}E^\top(\xi-\mu)},
	\]
	where $k=\text{\rm rank}(\cov)$, $D\in \PD^k$ is the diagonal matrix of the positive eigenvalues of $\cov$, and $E\in\R^{d\times k}$ is the matrix whose columns correspond to the orthonormal eigenvectors of the positive eigenvalues of $\cov$.
\end{definition}

Definition~\ref{def:normal-dist} also accounts for degenerate normal distributions with singular covariance matrices. We now recall some basic properties of normal distributions that are 
crucial for the results of this paper.

\begin{proposition}[{Affine transformations 
		\cite[Theorem~2.16]{ref:fang1990symmetric}}]
	\label{prop:normal-affine}
	If $\xi \in\R^d$ follows the normal distribution 
	$\mc N(\cent, \cov)$, while $\slope \in \R^{k \times d}$ and 
	$\intercept \in \R^k$, then $\slope \xi + \intercept\in\R^k$ follows the 
	normal distribution $\mc N(\slope \cent + \intercept, \slope 
	\cov \slope^\top)$.
\end{proposition}

\begin{proposition}[{Affine conditional expectations
		\cite[Corollary~5]{ref:cambanis1981theory}}]
	\label{prop:normal-cond-exp}
	Assume that $\xi \in\R^d$ follows the normal distribution $\PP= 
	\mc N(\cent, \cov)$ and that
	\begin{align*}
		\xi = \begin{bmatrix} \xi_1 \\ \xi_2 \end{bmatrix}, \qquad 
		\cent = \begin{bmatrix} \cent_1 \\ \cent_2 \end{bmatrix} \qquad 
		\text{and} \qquad \cov = \begin{bmatrix}
			\cov_{11} & \cov_{12} \\ \cov_{21} & \cov_{22}
		\end{bmatrix},
	\end{align*}
	where $\xi_1,\cent_1\in\R^{d_1}$, $\xi_2,\cent_2\in\R^{d_2}$, 
	$\cov_{11}\in\R^{d_1\times d_1}$, $\cov_{22}\in\R^{d_2\times d_2}$ and 
	$\cov_{12}=\cov_{21}^\top \in\R^{d_1\times d_2}$ for some 
	$d_1,d_2\in\mathbb N$ with $d_1+d_2= d$. Then, there exist 
	$\slope\in\R^{d_1\times d_2}$ and $\intercept\in\R^{d_1}$ such that 
	$\EE_\PP[\xi_1| \xi_2]=\slope \xi_2 + \intercept $ $\PP$-almost surely.
\end{proposition}

Another useful but lesser known property of normal distributions is 
that their Wasserstein distances can be expressed analytically in terms of the distributions' first- and second-order moments.

\begin{proposition}[Wasserstein distance between normal distributions {\cite[Proposition~7]{ref:givens1984class}}]
	\label{prop:normal:distance}
	The Wasserstein distance between two normal distributions 
	$\QQ_1 = \mc N(\cent_1, \cov_1)$ and $\QQ_2 = \mc N(\cent_2, \cov_2)$ equals the Gelbrich distance between their mean vectors and covariance matrices, that is, $\Wass(\QQ_1, \QQ_2) = \Gelbrich( (\m_1, \cov_1), (\m_2, 
	\cov_2) )$.
\end{proposition}

Assume now that the nominal distributions of the parameter $x\in\R^n$ and the noise $w\in\R^m$ are normal, that is, assume that $\Pnom_x = \mc N(\prior 
\cent_x, \covsa_x)$ and $\Pnom_w = \mc N(\prior \cent_w, 
\covsa_w)$.
Thus, the joint nominal distribution $\Pnom = \Pnom_{x} \times 
\Pnom_{w}$ is also normal, that is,
\begin{align}
	\label{eq:nominal:normal}
	\Pnom=\mc N(\prior \cent, \covsa) \qquad \text{where} 
	\qquad \prior \cent = \begin{bmatrix} \prior \cent_x \\ \prior \cent_w 
	\end{bmatrix} \qquad \text{and} \qquad \covsa = \begin{bmatrix} 
		\covsa_{x} & 0 \\ 0 & \covsa_{w} \end{bmatrix}.
\end{align}


Armed with the fundamental results on normal distributions summarized above, we are now ready to address the dual Wasserstein MMSE estimation problem~\eqref{eq:dual-dro} with a normal nominal distribution. In analogy to Section~\ref{sect:approx}, where we proposed the Gelbrich MMSE estimation problem as an easier conservative approximation for the original {\em primal} estimation problem~\eqref{eq:dro}, we will now construct an easier conservative approximation for the original {\em dual} estimation problem~\eqref{eq:dual-dro}. To this end, we define the restricted ambiguity set
\begin{align}
	\notag
	\mbb B_{\mc N}(\Pnom) = \left\{ \QQ_x \times \QQ_w \in \M(\R^{n})\times 
	\M(\R^{m}) :
	\begin{array}{l}
		\exists \cov_x\in\PSD^n,~ \cov_w\in\PD^m \text{ with }\\
		\QQ_x = \mc N(\msa_x, \cov_x), ~\QQ_w = 
		\mc N(\msa_w, \cov_w), \\
		\Wass (\QQ_x, \Pnom_x ) \leq \rho_x,~\Wass (\QQ_w, \Pnom_w) \leq \rho_w
	\end{array}
	\right\}.
\end{align}

By construction, $\mbb B_{\mc N}(\Pnom)$ contains all {\em normal} 
distributions $\QQ=\QQ_x \times \QQ_w$ from within the original 
Wasserstein ambiguity set $\mbb B(\Pnom)$ that have the same mean vector 
$(\msa_x, \msa_w)$ as the nominal 
distribution~$\Pnom=\Pnom_x\times\Pnom_w$, and where the covariance matrix of $\QQ_w$ is strictly positive definite.
Thus, we have $\mbb B_{\mc N}(\Pnom) \subseteq \mbb B(\Pnom)$. Note also that $\mbb B_{\mc N}(\Pnom)$ is non-convex because mixtures of normal distributions usually fail to be normal.

By restricting the original Wasserstein ambiguity set $\mbb B(\Pnom)$ to 
its subset $\mbb B_{\mc N}(\Pnom)$, we obtain the following conservative 
approximation for the dual Wasserstein MMSE estimation 
problem~\eqref{eq:dual-dro}.
\begin{align}
	\label{eq:dual-dro:conservative}
	\mathop{\text{maximize}}_{\QQ \in \mbb 
		B_{\mc N}(\Pnom)}\Inf{\psi\in\F}  \avrisk(\psi, \QQ)
\end{align}
We will henceforth refer to~\eqref{eq:dual-dro:conservative} as the dual Wasserstein MMSE estimation problem {\em over normal priors}. The 
following main theorem shows that~\eqref{eq:dual-dro:conservative} is 
equivalent to a finite convex optimization problem.
\begin{theorem}[Dual Wasserstein MMSE estimation problem over 
	normal priors]
	\label{thm:least-favorable-prior}
	Assume that the Wasserstein ambiguity set $\mbb B_{\mc N}(\Pnom)$ is centered at a normal distribution $\Pnom$ of the form~\eqref{eq:nominal:normal}. Then, the dual  Wasserstein MMSE estimation problem over normal priors~\eqref{eq:dual-dro:conservative} is equivalent to the finite convex optimization problem
	\be
	\label{eq:program:dual}
	\begin{array}{cl}
		\sup & \Tr{\cov_x - \cov_x H^\top \left( H \cov_x H^\top + 
			\cov_w \right)^{-1} H \cov_x} \\
		\st &  \cov_x \in \PSD^{n}, \quad \cov_w \in \PD^{m} \\[0.5ex]
		& \Tr{\cov_x + \covsa_x - 2 \big( \covsa_x^\half 
			\cov_x \covsa_x^\half \big)^\half} \leq \rho_x^2 \\[0.5ex]
		& \Tr{\cov_w + \covsa_w - 2 \big( \covsa_w^\half 
			\cov_w \covsa_w^\half \big)^\half} \leq \rho_w^2.
	\end{array}
	\ee
	If $\covsa_w \succ 0$, then~\eqref{eq:program:dual} is solvable, and the maximizer denoted by~$(\cov_x\opt, \cov_w\opt)$ satisfies $\cov_x\opt \succeq \lambda_{\min}(\covsa_x) I_n$ and $\cov_w\opt \succeq \lambda_{\min}(\covsa_w) I_m$. Moreover, the supremum of~\eqref{eq:dual-dro:conservative} is attained by the normal distribution $\QQ\opt=\QQ\opt_x\times\QQ\opt_w$ defined through $\QQ\opt_x = \mc N(\msa_x, \cov_x\opt)$ and $\QQ\opt_w= \mc N(\msa_w, \cov_w\opt)$.
\end{theorem}

\begin{proof}[Proof of Theorem~\ref{thm:least-favorable-prior}]
	If $(x,w)$ is governed by a normal distribution $\QQ \in 
	\mbb B_{\mc N}(\Pnom)$, then the linear transformation $(x, y) = (x, Hx + 
	w)$ is also normally distributed by virtue of 
	Proposition~\ref{prop:normal-affine}, and the average risk $\avrisk(\psi, 
	\QQ)$ is minimized by the Bayesian MMSE estimator $\psi\opt_{\mc 
		B}(y)=\EE_{\PP_{x|y}} [x]$, which is affine due to 
	Proposition~\ref{prop:normal-cond-exp}. Thus, in the dual Wasserstein 
	MMSE estimation problem with normal priors, the set~$\F$ of {\em 
		all} estimators may be restricted to the set~$\Ac$ of all {\em affine} 
	estimators without sacrificing optimality, that is,
	\begin{equation}
	\label{eq:dual-simplification}
	\Sup{\QQ \in \mbb B_{\mc N}(\Pnom)} \Inf{\psi\in\F} \avrisk(\psi, 
	\QQ) = \Sup{\QQ \in \mbb B_{\mc N}(\Pnom)} \Inf{\psi\in\Ac}  \avrisk(\psi, 
	\QQ).
	\end{equation}
	As the average risk $\avrisk(\psi, \QQ)$ of an affine estimator 
	$\psi\in\Ac$ simply evaluates the expectation of a quadratic function in 
	$(x,w)$, it depends on $\QQ$ only through its first and second moments. 
	Moreover, as $\QQ$ and $\Pnom$ are normal distributions, their Wasserstein distance coincides with 
	the Gelbrich distance between their mean vectors and covariance 
	matrices; see Proposition~\ref{prop:normal:distance}. Thus, the 
	maximization problem over $\QQ$ on the right hand side 
	of~\eqref{eq:dual-simplification} can be recast as an equivalent 
	maximization problem over the first and second moments of~$\QQ$. 
	Specifically, by the definitions of $\avrisk(\psi, \QQ)$ and $ \mbb 
	B_\phi(\Pnom)$ we~find
	\begin{align*}
		\Sup{\QQ \in \mbb B_{\mc N}(\Pnom)} \Inf{\psi\in\Ac} \avrisk(\psi, \QQ)
		= & \left\{
		\begin{array}{cl}
			\Sup{\cov_x, \cov_w} & \Inf{\substack{\slope, K \\ K = 
					I_n-\slope H}} \, \Inf{\intercept} ~\; \inner{K^\top K}{\cov_x + \msa_x 
				\msa_x^\top} + \inner{\slope^\top \slope}{\cov_w + \msa_w \msa_w^\top} + 
			\intercept^\top \intercept \\[-2ex]
			& \hspace{2.5cm} - 2 \msa_x^\top K^\top \slope \msa_w - 
			2\intercept^\top (K\msa_x - \slope \msa_w)\\[1ex]
			\st & \Gelbrich\big((\msa_x, \cov_x), (\msa_x, 
			\covsa_x)\big)\leq \rho_x,~\Gelbrich\big( (\msa_w, \cov_w), (\msa_w, 
			\covsa_w) \big) \leq \rho_w \\
			& \cov_x \succeq 0,~\cov_w \succ 
			0,        \end{array}
		\right.
	\end{align*}
	where the auxiliary decision variable $K = I_n-\slope H$ has 
	been introduced to simplify the objective function. The innermost minimization problem over $b$ constitutes an unconstrained (strictly) convex quadratic program that has the unique optimal solution $\intercept= K\msa_x - \slope \msa_w$. Substituting this minimizer back into the objective function of the above problem and recalling the definition of the Gelbrich distance then yields
	\begin{align}
		\Sup{\QQ \in \mbb B_{\mc N}(\Pnom)} \Inf{\psi\in\Ac} \avrisk(\psi, 
		\QQ) =& \left\{
		\begin{array}{cl}
			\Sup{\cov_x, \cov_w} & \Inf{\substack{\slope, K \\ K = 
					I_n-\slope H}} \; \inner{K^\top K}{\cov_x} + \inner{\slope^\top 
				\slope}{\cov_w} \\
			\st & \Tr{\cov_x + \covsa_x - 2 \big( \covsa_x^\half 
				\cov_x \covsa_x^\half \big)^\half} \leq \rho_x^2 \\[2ex]
			& \Tr{\cov_w + \covsa_w - 2 \big( \covsa_w^\half 
				\cov_w \covsa_w^\half \big)^\half} \leq \rho_w^2 \\[2ex]
			& \cov_x \succeq 0,~\cov_w \succ 0.
		\end{array}
		\right.    \label{eq:least1}
	\end{align}
	By using the equality $K = I_n-\slope H$ to eliminate $K$, the inner minimization problem in~\eqref{eq:least1} can be reformulated as an unconstrained quadratic program in~$\slope$. As $\cov_w \succ 0$, this quadratic program is strictly convex, and an elementary calculation reveals that its unique optimal solution is given by
	\[
	\slope\opt =  \cov_x H^\top \left( H {\cov_x} H^\top + 
	\cov_w\right)^{-1}.
	\]
	Substituting $\slope\opt$ as well as the corresponding auxiliary decision variable $K\opt=I_n-\slope\opt H$ into the objective function of~\eqref{eq:least1} finally yields the postulated convex program~\eqref{eq:program:dual}. 
	
	Assume now that $\covsa_w \succ 0$, and define
	\[
	\mc S_x = \left\{ \cov_x \in \PSD^n: \Gelbrich\big((\msa_x, \cov_x), (\msa_x, 
	\covsa_x)\big)\leq \rho_x \right\}\quad \text{and}\quad
	\mc S_w = \left\{ \cov_w \in \PSD^m: \Gelbrich\big( (\msa_w, \cov_w), (\msa_w, 
	\covsa_w) \big) \leq \rho_w \right\}.
	\]
	Equations~\eqref{eq:dual-simplification} and~\eqref{eq:least1} imply that
	\begin{align}
		\Sup{\QQ \in \mbb B_{\mc N}(\Pnom)} \Inf{\psi\in\F} \avrisk(\psi, 
		\QQ) \le &  \Sup{\cov_x \in \mc S_x} \Sup{\cov_w \in \mc S_w}  \Inf{\substack{\slope, K \\ K = 
				I_n-\slope H}} \; \inner{K^\top K}{\cov_x} + \inner{\slope^\top 
			\slope}{\cov_w} \notag \\
		=& \Sup{\substack{\cov_x \in \mc S_x \\ \cov_x \succeq \lambda_{\min}(\covsa_x) I_n}} \Sup{\substack{\cov_w \in \mc S_w \\ \cov_w \succeq \lambda_{\min}(\covsa_w) I_m}} \Inf{\substack{\slope, K \\ K = 
				I_n-\slope H}} \; \inner{K^\top K}{\cov_x} + \inner{\slope^\top 
			\slope}{\cov_w}, \label{eq:least2}
	\end{align}
	where the inequality holds because we relax the requirement that $\cov_w$ be strictly positive definite, and the equality follows from applying Lemma~\ref{lemma:monotone loss} consecutively to each of the two maximization problems. If $\covsa_w \succ 0$, then problem~\eqref{eq:least2} constitutes a restriction of~\eqref{eq:least1} and therefore provides also a lower bound on the dual Wasserstein MMSE estimation problem. In summary, we thus have
	\begin{align}
		\Sup{\QQ \in \mbb B_{\mc N}(\Pnom)} \Inf{\psi\in\F} \avrisk(\psi, 
		\QQ) =& \left\{
		\begin{array}{cl}
			\Sup{\cov_x, \cov_w} & \Inf{\substack{\slope, K \\ K = 
					I_n-\slope H}} \; \inner{K^\top K}{\cov_x} + \inner{\slope^\top 
				\slope}{\cov_w} \\
			\st & \Tr{\cov_x + \covsa_x - 2 \big( \covsa_x^\half 
				\cov_x \covsa_x^\half \big)^\half} \leq \rho_x^2 \\[2ex]
			& \Tr{\cov_w + \covsa_w - 2 \big( \covsa_w^\half 
				\cov_w \covsa_w^\half \big)^\half} \leq \rho_w^2 \\[2ex]
			& \cov_x \succeq\lambda_{\min}(\covsa_x) I_n, ~\cov_w \succeq \lambda_{\min}(\covsa_w) I_m.
		\end{array}
		\right. \label{eq:least3}
	\end{align}
	This reasoning implies that if $\covsa_w \succ 0$, then the constraints $\cov_x \succeq \lambda_{\min}(\covsa_x) I_n$ and $\cov_w \succeq \lambda_{\min}(\covsa_w) I_m$ can be appended to problem~\eqref{eq:least1} and, consequently, to problem~\eqref{eq:program:dual} without altering their common optimal value. Problem~\eqref{eq:program:dual} with the additional constraints $\cov_x \succeq \lambda_{\min}(\covsa_x) I_n$ and $\cov_w \succeq \lambda_{\min}(\covsa_w) I_m$ has a continuous objective function over a compact feasible set and is thus solvable. Any of its optimal solutions is also optimal in problem~\eqref{eq:program:dual}, which has no redundant constraints. Thus, problem~\eqref{eq:program:dual} is solvable.
	
	It remains to show that $\QQ\opt$ as constructed in the theorem statement is optimal in~\eqref{eq:dual-dro:conservative}. The feasibility of $(\cov_x\opt, \cov_w\opt)$ in~\eqref{eq:program:dual} implies that $\QQ\opt \in \mbb B_{\mc N}(\Pnom)$, and thus $\QQ\opt$ is feasible in~\eqref{eq:dual-dro:conservative}. Moreover, we have
	\begin{align}
		\label{eq:q-optimality}
		\Sup{\QQ \in  \mbb B_{\mc N} (\Pnom)} \, \Inf{\psi \in \F} \, 
		\avrisk(\psi, \QQ) &\ge \Inf{\psi \in \F} \, \avrisk(\psi, \QQ\opt) = 
		\Tr{\cov_x\opt - \cov_x\opt H^\top \left( H \cov_x\opt H^\top + 
			\cov_w\opt \right)^{-1} H \cov_x\opt},
	\end{align}
	where the equality follows from elementary algebra, recalling that the affine estimator $\psi(y) = \slope\opt y+\intercept\opt$ with
	\[
	\slope\opt =  \cov_x\opt  H^\top \left( H {\cov_x\opt} 
	H^\top + \cov_w\opt\right)^{-1}\quad \text{and} \quad \intercept\opt= 
	\mu_x -  \slope\opt (H \msa_x + \msa_w)
	\]
	is the Bayesian MMSE estimator for the normal distribution
	$\QQ\opt$. As the right hand side of~\eqref{eq:q-optimality} coincides 
	with the maximum of~\eqref{eq:program:dual} and as 
	problem~\eqref{eq:program:dual} is equivalent to the dual Wasserstein 
	MMSE estimation problem~\eqref{eq:dual-dro:conservative} over normal 
	priors, we may thus conclude that the 
	inequality in~\eqref{eq:q-optimality} is tight. Thus, we find
	\[
	\Sup{\QQ \in  \mbb B_{\mc N}(\Pnom)} \, \Inf{\psi \in \F} \, 
	\avrisk(\psi, \QQ) = \Inf{\psi \in \F} \, \avrisk(\psi, \QQ\opt),
	\]
	which in turn implies that $\QQ\opt$ is optimal
	in~\eqref{eq:dual-dro:conservative}. This observation completes the proof.
\end{proof}

\begin{remark}[Singular covariance matrices]
	A nonlinear SDP akin to~\eqref{eq:program:dual} has been derived in~\cite{ref:shafieezadeh2018wasserstein} under the stronger assumption that the covariance matrix of the nominal distribution $\Pnom$ is non-degenerate, which implies that $\covsa_x \succ 0$ and $\covsa_w \succ 0$. However, the weaker condition $\covsa_w \succ 0$ is sufficient to ensure that the matrix inversion in the objective function of problem~\eqref{eq:program:dual} is well-defined. Therefore, Theorem~\ref{thm:least-favorable-prior} remains valid if the nominal covariance matrix $\covsa_x$ is singular, which occurs in many applications. 
	On the other hand, it is common to require that $\covsa_w = \sigma^2 I_m$ for some $\sigma>0$, see, e.g.,~\cite{ref:chang2000adaptive}.
\end{remark}

Corollary~\ref{corol:dual:refor} below asserts that the convex program~\eqref{eq:program:dual} admits a canonical linear SDP reformulation. The proof is omitted as it relies on standard Schur complement arguments familiar from the proof of Corollary~\ref{cor:primal:refor}. 
\begin{corollary}[SDP reformulation] \label{corol:dual:refor}
	Assume that the Wasserstein ambiguity set $\mbb B_{\mc N}(\Pnom)$ is centered at a normal distribution $\Pnom$ of the form~\eqref{eq:nominal:normal} with noise covariance matrix $\covsa_w \succ 0$. Then, the dual Wasserstein MMSE estimation problem~\eqref{eq:dual-dro:conservative} over normal priors is equivalent to the SDP
	\be \label{eq:SDP:dual}
	\begin{array}{cl}
		\max & \Tr{\cov_x} - \Tr{U}  \\[1ex]
		\st & \cov_x \in \PSD^n, \, \cov_w \in \PSD^m, \, V_x \in 
		\PSD^n, \, V_w \in \PSD^m, \, U \in \PSD^n \\[1ex]
		& \begin{bmatrix} \covsa_x^\half \cov_x \covsa_x^\half & 
			V_x \\ V_x & I_n \end{bmatrix} \succeq 0, \quad
		\begin{bmatrix} \covsa_w^\half \cov_w \covsa_w^\half & V_w 
			\\ V_w & I_m \end{bmatrix} \succeq 0\\[3ex]
		& \Tr{\cov_x + \covsa_x - 2V_x} \leq \rho_x^2, \quad 
		\Tr{\cov_w + \covsa_w - 2V_w} \leq \rho_w^2 \\[2ex]
		& \begin{bmatrix} U & \cov_x H^\top \\ H \cov_x & H \cov_x 
			H^\top + \cov_w \end{bmatrix} \succeq 0, \quad \cov_x \succeq 
		\lambda_{\min}(\covsa_x) I_n, \quad \cov_w \succeq 
		\lambda_{\min}(\covsa_w) I_m.
	\end{array}
	\ee
\end{corollary}
We emphasize that the lower bounds on $\cov_x $ and $\cov_w$ are redundant but have been made explicit in~\eqref{eq:SDP:dual}.
\section{Nash Equilibrium and Optimality of Affine Estimators}
\label{sect:nash}
If $\Pnom$ is a normal distribution of the form~\eqref{eq:nominal:normal}, then we have
\be \label{eq:strong:duality}
\Inf{\psi\in\Ac} \Sup{\QQ \in \mbb G(\Pnom)} \avrisk(\psi, \QQ) \ge
\Inf{\psi\in\F} \Sup{\QQ \in \mbb B(\Pnom)} \avrisk(\psi, \QQ) \ge
\Sup{\QQ \in \mbb B(\Pnom)}\Inf{\psi\in\F}  \avrisk(\psi, \QQ) \ge
\Sup{\QQ \in \mbb B_{\mc N}(\Pnom)}\Inf{\psi\in\F}  \avrisk(\psi, \QQ),
\ee
where the first inequality follows from the inclusions $\Ac \subseteq \F$ and $\mbb B(\Pnom) \subseteq \mbb G(\Pnom)$, the second inequality exploits weak duality, and the last inequality holds due to the inclusion $\mbb B_{\mc N}(\Pnom) \subseteq \mbb B(\Pnom)$. Note that the leftmost minimax problem is the Gelbrich MMSE estimation problem~\eqref{eq:dro:approx} studied in Section~\ref{sect:approx}, and the rightmost maximin problem is the dual Wasserstein MMSE estimation problem~\eqref{eq:dual-dro:conservative} over normal priors studied in Section~\ref{sect:dual}. We also highlight that these restricted primal and dual estimation problems sandwich the original Wasserstein estimation problems~\eqref{eq:dro} and~\eqref{eq:dual-dro}, which coincide with the second and third problems in~\eqref{eq:strong:duality}, respectively. The following theorem asserts that all inequalities in~\eqref{eq:strong:duality} actually collapse to equalities. 

\begin{theorem}[Sandwich theorem]
	\label{thm:sandwich}
	If $\Pnom$ is a normal distribution of the form~\eqref{eq:nominal:normal}, then the optimal values of the restricted primal and dual estimation problems~\eqref{eq:dro:approx} and~\eqref{eq:dual-dro:conservative} coincide, i.e.,
	\[
	\Inf{\psi\in\Ac} \Sup{\QQ \in \mbb G(\Pnom)} \avrisk(\psi, \QQ) = \Sup{\QQ \in \mbb B_{\mc N}(\Pnom)}\Inf{\psi\in\F}  \avrisk(\psi, \QQ).
	\]
\end{theorem}
\begin{proof}[Proof of Theorem~\ref{thm:sandwich}]
	By Theorem~\ref{thm:conservative}, the Gelbrich MMSE estimation problem~\eqref{eq:dro:approx} can be expressed as
	\begin{align*}
		\Inf{\psi\in\Ac} \Sup{\QQ \in \G(\Pnom)} \avrisk(\psi, \QQ) = 
		\left\{
		\begin{array}{ccl}
			\Inf{\substack{\slope, K \\ K = I_n - \slope H}} & 	\Inf{\substack{\dualx,\dualw \\ \dualx I_n \succ K^\top K \\ \dualw I_m \succ \slope^\top \slope }} &
			\dualx \big( \rho_x^2 - \Tr{\covsa_x} \big) + \dualx^2 \inner{(\dualx I_n - K^\top K)^{-1}}{\covsa_x}  \\[-4ex]
			&& \hspace{0.5cm} + \dualw \big( \rho_w^2 - \Tr{\covsa_w} \big) + \dualw^2 \inner{(\dualw I_m - \slope^\top \slope)^{-1}}{\covsa_w},
		\end{array}
		\right. 
	\end{align*}
	where the auxiliary variable $K = I_n-\slope H$ has been introduced to highlight the problem's symmetries.
	Next, we introduce the feasible sets
	\[
	\mc S_x = \left\{ \cov_x \in \PSD^n: \Tr{\cov_x + \covsa_x - 2 \big( \covsa_x^\half \cov_x \covsa_x^\half \big)^\half} \leq \rho_x^2  \right\}
	\]
	and
	\[
	\mc S_w = \left\{ \cov_w \in \PSD^m: \Tr{\cov_w + \covsa_w - 2 \big( \covsa_w^\half \cov_w \covsa_w^\half \big)^\half} \leq \rho_w^2  \right\},
	\]
	both of which are convex and compact by virtue of Lemma~\ref{lemma:compact:FS}. 
	Using Proposition~\ref{prop:quadratic}\,\ref{prop:quad:dual} to reformulate the inner minimization problem over $\dualvar_x$ and $\dualvar_w$, we then obtain
	\begin{align}
		\Inf{\psi\in\Ac} \Sup{\QQ \in \G(\Pnom)} \avrisk(\psi, \QQ) &= 
		\Inf{\substack{\slope, K \\ K = I_n - \slope H}} \Sup{\substack{\cov_x \in \mc S_x \\ \cov_w \in \mc S_w}} ~\inner{K^\top K}{\cov_x} + \inner{\slope^\top \slope}{\cov_w} \notag
		\\
		&= \Sup{\cov_x \in \mc S_x} \Inf{\substack{\slope, K \\ K = I_n - \slope H}} \Sup{ \cov_w \in \mc S_w} ~\inner{K^\top K}{\cov_x} + \inner{\slope^\top \slope}{\cov_w} \notag
	\end{align}
	where the second equality holds due to Sion's minimax theorem~\cite{ref:sion1958minimax}. Define now the auxiliary function 
	\[
	    f(A) = \Sup{\cov_w \in \mc S_w}~\inner{A^\top A}{\cov_w}.
	\]
	As $\cov_w \succeq 0$ for any $\cov_w \in \mc S_w$, $f$ constitutes a pointwise maximum of convex functions and is therefore itself convex. In addition, as the set $\mc S_w$ is compact by Lemma~\ref{lemma:compact:FS}, $f$ is everywhere finite and thus continuous thanks to \cite[Theorem~2.35]{ref:rockafellar2010variational}. This allows us to conclude that
	\begin{align}
		\Inf{\psi\in\Ac} \Sup{\QQ \in \G(\Pnom)} \avrisk(\psi, \QQ) 
		&= \Sup{\cov_x \in \mc S_x} \Inf{\substack{\slope, K \\ K = I_n - \slope H}} ~\inner{K^\top K}{\cov_x} + f(A) \notag \\
		&= \Sup{\substack{\cov_x \in \mc S_x \\ \cov_x \succeq \lambda_{\min}(\covsa_x) I_n}} \Inf{\substack{\slope, K \\ K = I_n - \slope H}} ~\inner{K^\top K}{\cov_x} + f(A) \notag \\
		&=  \Inf{\substack{\slope, K \\ K = I_n - \slope H}} \Sup{\substack{\cov_x \in \mc S_x \\ \cov_x \succeq \lambda_{\min}(\covsa_x) I_n}} \Sup{ \cov_w \in \mc S_w} ~\inner{K^\top K}{\cov_x} + \inner{\slope^\top \slope}{\cov_w}, \notag
	\end{align}
	where the second equality exploits Lemma~\ref{lemma:monotone loss}, and the last equality follows from Sion's minimax theorem~\cite{ref:sion1958minimax}. Another (trivial) application of Lemma~\ref{lemma:monotone loss} then allows us to append the constraint $\cov_w \succeq \lambda_{\min}(\covsa_w) I_m$ to the maximization problem over $\cov_w$. Sion's minimax theorem~\cite{ref:sion1958minimax} finally implies that
	\begin{align*}
		\Inf{\psi\in\Ac} \Sup{\QQ \in \G(\Pnom)} \avrisk(\psi, \QQ)&= \Sup{\substack{\cov_x \in \mc S_x \\ \cov_x \succeq \lambda_{\min}(\covsa_x) I_n}} \Sup{\substack{ \cov_w \in \mc S_w \\ \cov_w \succeq \lambda_{\min}(\covsa_w) I_m}} \Inf{\substack{\slope, K \\ K = I_n - \slope H}} ~\inner{K^\top K}{\cov_x} + \inner{\slope^\top \slope}{\cov_w} \notag \\
		&= \Sup{\QQ \in \mbb B_{\mc N}(\Pnom)} \Inf{\psi\in\F} \avrisk(\psi, \QQ), \notag
	\end{align*}
	where the last equality has already been established in the proof of Theorem~\ref{thm:least-favorable-prior}; see Equation~\eqref{eq:least3}. Thus, the claim follows.
\end{proof}
	

Theorem~\ref{thm:sandwich} suggests that solving any of the restricted estimation problems is tantamount to solving both original primal and dual estimation problems. This intuition is formalized in the following corollary.

\begin{corollary}[Nash equilibrium] \label{corol:nash}
	If $\Pnom$ is a normal distribution of the form~\eqref{eq:nominal:normal} with $\covsa_w \succ 0$, then the affine estimator~$\psi\opt$ that solves~\eqref{eq:dro:approx} is optimal in the primal Wasserstein MMSE estimation problem~\eqref{eq:dro}, while the normal distribution~$\QQ\opt$ that solves~\eqref{eq:dual-dro:conservative} is optimal in the dual Wasserstein MMSE estimation problem~\eqref{eq:dual-dro}. Moreover, $\psi\opt$ and $\QQ\opt$ form a Nash equilibrium for the game between the statistician and nature, that is,
	\begin{equation}
	\label{eq:nash}
	\avrisk(\psi\opt, \QQ) \le \avrisk(\psi\opt, \QQ\opt) \le \avrisk(\psi, {\QQ}\opt) \quad \forall \psi \in \F,~ \QQ \in  \mbb B(\Pnom) \,.
	\end{equation}
\end{corollary}
\begin{proof}[Proof of Corollary~\ref{corol:nash}]
	As $\covsa_w \succ 0$, the Gelbrich MMSE estimation problem~\eqref{eq:dro:approx} is solved by the affine estimator $\psi\opt$ defined in Theorem~\ref{thm:conservative}, and the dual Wasserstein MMSE estimation problem~\eqref{eq:dual-dro} over normal priors is solved by the normal distribution $\QQ\opt$ defined in Theorem~\ref{thm:least-favorable-prior}. Thus, we have
	\begin{align*}
		\avrisk (\psi\opt, \QQ\opt) \ge \Inf{\psi \in \F} \avrisk (\psi, \QQ\opt) = \Max{\QQ \in \mbb B_{\mc N}(\Pnom)} \Inf{\psi\in\F} \avrisk(\psi, \QQ) = \Min{\psi\in\Ac} \Sup{\QQ \in \G(\Pnom)} \avrisk(\psi, \QQ) = \Sup{\QQ \in \G(\Pnom)} \avrisk(\psi\opt, \QQ) \ge \avrisk (\psi\opt, \QQ\opt),
	\end{align*}
	where the three equalities follow from the definition of $\QQ\opt$,  Theorem~\ref{thm:sandwich} and the definition of $\psi\opt$, respectively. As the left and the right hand sides of the above expression coincide, we may then conclude that
	\[
	\avrisk(\psi\opt, \QQ) \le \avrisk(\psi\opt, \QQ\opt) \le \avrisk(\psi, {\QQ}\opt) \quad \forall \psi \in \F,~ \QQ \in  \mbb G(\Pnom).
	\]
	Moreover, as $\mbb B(\Pnom) \subseteq \mbb G(\Pnom)$, the above relation implies~\eqref{eq:nash}.
	
	It remains to be shown that $\psi\opt$ and $\QQ\opt$ solve the primal and dual Wasserstein MMSE estimation problems~\eqref{eq:dro} and~\eqref{eq:dual-dro}, respectively. As for $\psi\opt$, we have
	\[
	\Sup{\QQ \in \mbb B(\Pnom)} \avrisk(\psi\opt, \QQ) \le \Sup{\QQ \in \G(\Pnom)} \avrisk(\psi\opt, \QQ) = \Inf{\psi\in\Ac} \Sup{\QQ \in \mbb G(\Pnom)} \avrisk(\psi, \QQ) = \Inf{\psi\in\F} \Sup{\QQ \in \mbb B(\Pnom)} \avrisk(\psi, \QQ),
	\]
	where the inequality holds because $\mbb B(\Pnom) \subseteq \G(\Pnom)$. The first equality follows from  the definition of $\psi\opt$, while the second equality exploits Theorem~\ref{thm:sandwich}, which implies that all inequalities in~\eqref{eq:strong:duality} are in fact equalities. This reasoning shows that $\psi\opt$ is optimal in~\eqref{eq:dro}. The optimality of $\QQ\opt$ in~\eqref{eq:dual-dro} can be proved similarly. 
\end{proof}

Corollary~\ref{corol:nash} implies that $\psi\opt$ can be viewed as a Bayesian estimator for the least favorable prior $\QQ\opt$ and that $\QQ\opt$ represents a worst-case distribution for the optimal estimator $\psi\opt$. Next, we will argue that $\psi\opt$ can not only be constructed from the solution of the convex program~\eqref{eq:VA}, which is equivalent to the Gelbrich MMSE estimation problem~\eqref{eq:dro:approx}, but also from the solution of the convex program~\eqref{eq:program:dual}, which is equivalent to the dual MMSE estimation problem~\eqref{eq:dual-dro:conservative} over normal~priors. This alternative construction is useful because problem~\eqref{eq:program:dual} is amenable to highly efficient first-order methods to be derived in Section~\ref{sect:algorithm}.

\begin{corollary}[Dual construction of the optimal estimator]
	\label{corol:alternative}
	If $\Pnom$ is a normal distribution of the form~\eqref{eq:nominal:normal} with $\covsa_w\succ 0$, and $(\cov_x\opt, \cov_w\opt)$ is a maximizer of~\eqref{eq:program:dual}, then the affine estimator $\psi\opt(y) = \slope\opt y + \intercept\opt$ with
	\be
	\label{eq:alternative:def}
	\slope\opt =  \cov_x\opt H^\top \left( H {\cov_x\opt} H^\top + \cov_w\opt\right)^{-1} \quad \text{and} \quad 
	\intercept\opt= \msa_x - \slope\opt (H \msa_x + \msa_w)
	\ee
	solves the Wasserstein MMSE estimation problem~\eqref{eq:dro}.
\end{corollary}

\begin{proof}[Proof of Corollary~\ref{corol:alternative}]
	Define $\psi\opt$ as the affine estimator that solves~\eqref{eq:dro:approx} and $\QQ\opt$ as the normal distribution that solves~\eqref{eq:dual-dro:conservative}. By Corollary~\ref{corol:nash}, the second inequality in~\eqref{eq:nash} holds for all admissible estimators $\psi\in\F$, which implies that $\psi^\star \in \arg\min_{\psi\in\F} \avrisk(\psi,\QQ\opt)$, that is, $\psi\opt$ solves the Bayesian MMSE estimation problem corresponding to $\QQ\opt$. As any Bayesian MMSE estimator satisfies $\psi\opt(y)=\EE_{\QQ\opt_{x|y}} [x]$ for $\QQ\opt$-almost all $y$ and as $\cov\opt_w\succ 0$, we may use the known formulas for conditional normal distributions to conclude that the unique affine Bayesian MMSE estimator for $\QQ\opt$ is of the form $\psi\opt(y) = \slope\opt y + \intercept\opt$ with parameters defined as in~\eqref{eq:alternative:def}.
\end{proof}

\section{Non-normal Nominal Distributions}
\label{sect:elliptical}

We will first show that the results of Sections~\ref{sect:approx}--\ref{sect:nash} remain valid if $\Pnom$ is an arbitrary elliptical (but maybe non-normal) distribution. To this end, we first review some basic results on elliptical distributions.

\begin{definition}[Elliptical distributions]
	\label{definition:elliptical}
	The distribution $\PP$ of $\xi\in\R^d$ is called elliptical if the 
	characteristic function $\Phi_\PP(t)= \EE_\PP[ \exp(i t^\top \xi)]$ 
	of $\PP$ is given by $\Phi_\PP(t)=\exp(i t^\top \cent) \phi(t^\top \scat 
	t)$ for some location parameter ${\cent \in \R^d}$, dispersion matrix 
	$\scat \in \PSD^{d}$ and characteristic generator $\phi:\R_+\! \rightarrow 
	\R$. In this case we write~${\PP=\Ec^d_\phi(\cent, \scat)}$. The class of 
	all $d$-dimensional elliptical distributions with characteristic 
	generator $\phi$ is denoted by~$\Ec^{d}_\phi$.
\end{definition}

The class of elliptical distributions was introduced 
in~\cite{ref:Kelker-1970} with the aim to generalize the family of 
normal distributions, which are obtained by setting the characteristic 
generator to~$\phi(u)=e^{-u/2}$. We emphasize that, unlike the 
moment-generating function $M_\PP(t)= \EE_\PP[ \exp(t^\top \xi)]$, 
the characteristic function $\Phi_\PP(t)$ is always finite for all 
$t\in\R^d$ even if some moments of $\PP$ do not exist. Thus, 
Definition~\ref{definition:elliptical} is general enough to cover also 
heavy-tailed distributions with non-zero tail dependence coefficients 
\cite{ref:hult2002advances}. Examples of elliptical distributions include the Laplace, logistic and $t$-distribution etc. Useful theoretical properties 
of elliptical distributions are discussed in~\cite{ref:cambanis1981theory,
	ref:fang1990symmetric}. We also highlight that elliptical distributions are 
central to a wide spectrum of diverse applications ranging from genomics 
\cite{ref:posekany2011biological} and medical imaging \cite{ref:ruttimann1998statistical} 
to finance \cite[\S~6.2.1]{ref:jondeau2007financial}, to name a few.

If the dispersion matrix $\scat\in\PSD^d$ has rank $r$, then there 
exists $\Lambda\in\R^{d\times r}$ with $\scat=\Lambda\Lambda^\top$, and 
there exists a generalized inverse $\Lambda^{-1}\in\R^{r\times d}$ with 
$\Lambda^{-1} \Lambda=I_r=\Lambda^\top(\Lambda^{-1})^\top$. One easily 
verifies that if $\xi\in\R^d$ follows an elliptical distribution 
$\PP=\Ec^d_\phi(\cent, \scat)$, then $\tilde \xi= 
\Lambda^{-1}(\xi-\cent)\in\R^r$ follows the spherically symmetric 
distribution $\tilde\PP=\Ec^d_\phi(0, I_r)$ with characteristic function 
$\Phi_{\tilde\PP}(t)=\phi(\|t\|^2)$. Thus, the choice of the 
characteristic generator $\phi$ is constrained by the implicit condition 
that $\phi(\|t\|^2)$ must be an admissible characteristic function. For 
instance, the normalization of probability distributions necessitates 
that $\phi(0) = 1$, while the dominated convergence theorem implies that 
$\phi$ must be continuous etc. As any distribution is uniquely 
determined by its characteristic function, and as $\phi(\|t\|^2)$ 
depends only on the norm of~$t$, the spherical distribution $\tilde \PP$ 
is indeed invariant under rotations. This implies that $\EE_{\tilde 
	\PP}[\tilde \xi]=0$ and, via the linearity of the expectation, that 
$\EE_{\PP}[\xi]=\cent$ provided that $\tilde\xi$ and $\xi$ are 
integrable, respectively. Thus, the location parameter $\cent$ of an 
elliptical distribution coincides with its mean vector whenever the mean 
exists. By the definition of the characteristic function, the covariance 
matrix of $\tilde \PP$, if it exists, can be expressed as
\[
\tilde \cov=-\left.\nabla_t^2\Phi_{\tilde \PP}(t)\right|_{t=0} = 
-\left.\nabla_t^2\phi(\|t\|^2)\right|_{t=0}= -2 \phi'(0) I_r,
\]
where $\phi'(0)$ denotes the right derivative of $\phi(u)$ at $u=0$. 
Hence, $\tilde \cov$ exists if and only if $\phi'(0)$ exists and is 
finite. Similarly, the covariance matrix of $\PP$ is given by $\cov=-2 
\phi'(0)\scat$, if it exists \cite[Theorem~4]{ref:cambanis1981theory}. Below 
we will focus on elliptical distributions with finite first- and 
second-order moments ({\em i.e.}, we will only consider characteristic 
generators with $|\phi'(0)|<\infty$), and we will assume that 
$\phi'(0)=-\half$, which ensures that the dispersion matrix $\scat$ 
equals the covariance matrix $\cov$. The latter assumption does not 
restrict generality. In fact, changing the characteristic generator to 
$\phi(\frac{-u}{2\phi'(0)})$ and the dispersion matrix to $-2 
\phi'(0)\scat$ has no impact on the elliptical distribution $\PP$ but 
matches the dispersion matrix $\scat$ with the covariance matrix $\cov$.

The elliptical distributions inherit many desirable properties from the normal distributions but are substantially more expressive as they include also heavy- and light-tailed distributions. For example, any class of elliptical distributions with a common characteristic generator is closed under affine transformations and affine conditional expectations; see {\em e.g.}, \cite[Theorem~1 and Corollary~5]{ref:cambanis1981theory}. Moreover, the Wasserstein distance between two elliptical distributions with the same characteristic generator equals the Gelbrich distance between their mean vectors and covariance matrices \cite[Theorem~2.4]{ref:gelbrich1990formula}. Thus, the Propositions~\ref{prop:normal-affine}, \ref{prop:normal-cond-exp} and~\ref{prop:normal:distance} extend verbatim from the class of normal distributions to {\em any} class of elliptical distributions that share the same characteristic generator. For the sake of brevity, we do not restate these results for elliptical distributions.

The above discussion suggests that the results of Sections~\ref{sect:approx}--\ref{sect:nash} carry over almost verbatim to MMSE estimation problems involving elliptical nominal distributions. In the following we will therefore assume that
\begin{align}
	\label{eq:nominal:elliptical}
	\Pnom=\Ec^{n+m}_\phi(\msa, \covsa) \qquad \text{with} 
	\qquad \prior \cent = \begin{bmatrix} \prior \cent_x \\ \prior \cent_w 
	\end{bmatrix} \qquad \text{and} \qquad \covsa = \begin{bmatrix} 
		\covsa_{x} & 0 \\ 0 & \covsa_{w} \end{bmatrix},
\end{align}
where $\phi$ denotes a prescribed characteristic generator. As the class of all elliptical distributions with characteristic generator $\phi$ is closed under affine transformations, the marginal distributions $\Pnom_x$ and $\Pnom_w$ of $x$ and $w$ under $\Pnom$ are also elliptical distributions with the same characteristic generator $\phi$.

Note that while the signal~$x$ and the noise~$w$ are uncorrelated under $\Pnom$ irrespective of~$\phi$, they fail to be independent unless~$\Pnom$ is a normal distribution. When working with generic elliptical nominal distributions, we must therefore abandon any independence assumptions. 
Otherwise, the ambiguity set would be empty for small radii $\rho_x$ and $\rho_w$. This insight prompts us to redefine the Wasserstein ambiguity set as
\be
\label{eq:Ambi-elliptical}
\mbb B(\Pnom) = \left\{ \QQ\in \M(\R^{n+m}):
\EE_\QQ[xw^\top]= \EE_\QQ[x]\cdot\EE_\QQ[w]^\top,~\Wass(\QQ_x, \Pnom_x) \leq \rho_x,~ \Wass(\QQ_w, \Pnom_w) \leq \rho_w \right\},
\ee
which relaxes the independence condition in~\eqref{eq:Ambi} and merely requires $x$ and $w$ to be uncorrelated. When using the new ambiguity set~\eqref{eq:Ambi-elliptical} to model the distributional uncertainty, we can again compute a Nash equilibrium between the statistician and nature by solving a tractable convex optimization problem.

\begin{theorem}[Elliptical distributions] \label{thm:nash-elliptical}
	Assume that $\Pnom$ is an elliptical distribution of the form~\eqref{eq:nominal:elliptical} with characterisic generator $\phi$ and noise covariance matrix $\covsa_w \succ 0$, and define the ambiguity set $\mbb B(\Pnom)$ as in~\eqref{eq:Ambi-elliptical}. If $(\cov_x\opt, \cov_w\opt)$ solves the finite convex program~\eqref{eq:dual-dro:conservative}, then the affine estimator $\psi\opt(y) = \slope\opt y + \intercept\opt$ with
	\[
	\slope\opt =  \cov_x\opt H^\top \left( H {\cov_x\opt} H^\top + \cov_w\opt\right)^{-1} \quad \text{and} \quad 
	\intercept\opt= \msa_x - \slope\opt (H \msa_x + \msa_w)
	\]
	solves the Wasserstein MMSE estimation problem~\eqref{eq:dro}, while the elliptical distribution
	\[
	\QQ\opt=\mc E_\phi^{n+m}(\msa,\cov\opt)\quad \text{with} \quad \cov\opt= \begin{bmatrix} 
	\cov_{x}\opt & 0 \\ 0 & \cov_{w\opt} \end{bmatrix}
	\]
	solves the dual Wasserstein MMSE estimation problem~\eqref{eq:dual-dro}. Moreover, $\psi\opt$ and $\QQ\opt$ form a Nash equilibrium for the game between the statistician and nature, that is,
	\[
	\avrisk(\psi\opt, \QQ) \le \avrisk(\psi\opt, \QQ\opt) \le \avrisk(\psi, {\QQ}\opt) \quad \forall \psi \in \F,~ \QQ \in  \mbb B(\Pnom) \,.
	\]
\end{theorem}
\begin{proof}[Proof of Theorem~\ref{thm:nash-elliptical}]
	The proof replicates the arguments used to establish Theorems~\ref{thm:conservative}, \ref{thm:least-favorable-prior} and \ref{thm:sandwich} as well as Corollary~\ref{corol:nash} with obvious minor modifications. Details are omitted for brevity. 
\end{proof}

Theorem~\ref{thm:nash-elliptical} asserts that the optimal estimator depends only on the first and second moments of the nominal elliptical distribution~$\Pnom$ but {\em not} on its characteristic generator. Whether~$\Pnom$ displays heavier or lighter tails than a normal distribution has therefore no impact on the prediction of the signal. Note, however, that the characteristic generator of~$\Pnom$ determines the shape of the least favorable prior. 

If the nominal distribution fails to be elliptical, the minimum of the Gelbrich MMSE estimation problem~\eqref{eq:dro:approx} may strictly exceed the maximum of the dual Wasserstein MMSE estimation problem~\eqref{eq:dual-dro:conservative} over normal priors. Note that in this case the ambiguity set $\mbb B_{\mc N}(\Pnom)$ may even be empty. Moreover, while typically suboptimal for the original Wasserstein MMSE estimation problem~\eqref{eq:dro}, the usual affine estimator constructed from a solution of the nonlinear SDP~\eqref{eq:program:dual} still solves the Gelbrich MMSE estimation problem~\eqref{eq:dro:approx}.
\begin{proposition}[Non-elliptical nominal distributions] \label{prop:equivalence}
    	Suppose that $\Pnom = \Pnom_{x} \times \Pnom_{w}$, where $\Pnom_x$ and $\Pnom_w$ are arbitrary signal and noise distributions with mean vectors $\msa_x$ and $\msa_w$ and covariance matrices $\covsa_x \succeq 0$ and $\covsa_w \succ 0$, respectively. Then, the nonlinear SDP~\eqref{eq:program:dual} is solvable, and for any optimal solution~$(\cov_x\opt, \cov_w\opt)$ of~\eqref{eq:program:dual} the affine estimator $\psi\opt(y) = \slope\opt y + \intercept\opt$ with
	\[
	\slope\opt =  \cov_x\opt H^\top \left( H {\cov_x\opt} H^\top + \cov_w\opt\right)^{-1} \quad \text{and} \quad 
	\intercept\opt= \msa_x - \slope\opt (H \msa_x + \msa_w)
	\]
	solves the Gelbrich MMSE estimation problem~\eqref{eq:dro:approx}.
\end{proposition}
\begin{proof}[Proof of Proposition~\ref{prop:equivalence}]
Denote by~$\Pnom' = \Pnom'_{x} \times \Pnom'_{w}$ the normal distribution with the same first and second moments as~$\Pnom$. As $\covsa_w\succ 0$, the nonlinear SDP~\eqref{eq:program:dual} is then solvable by virtue of Theorem~\ref{thm:least-favorable-prior}. Theorem~\ref{thm:sandwich} further implies that the first inequality in~\eqref{eq:strong:duality} with $\Pnom'$ instead of $\Pnom$ collapses to the equality
    \begin{equation}
        \label{eq:gebrich-meets-wasserstein}
        \Inf{\psi\in\Ac} \Sup{\QQ \in \mbb G(\Pnom')} \avrisk(\psi, \QQ)= \Inf{\psi\in\F} \Sup{\QQ \in \mbb B(\Pnom')} \avrisk(\psi, \QQ).
    \end{equation}
    In addition, Corollary~\ref{corol:alternative} ensures that the affine estimator $\psi\opt$ defined in the proposition statement solves the modified Wasserstein MMSE estimation problem with normal nominal distribution $\Pnom'$ on the right hand side of~\eqref{eq:gebrich-meets-wasserstein}. Because $\psi\opt$ is affine, it is also feasible in the modified Gelbrich MMSE estimation problem on the left hand side. In addition, the average risk of any affine estimator depends only on the mean vectors and covariance matrices of $x$ and $w$. If we denote by $T$ the mean-covariance projection that maps any distribution $\QQ = \QQ_x \times \QQ_w$ of $(x,w)$ to the mean vectors and covariance matrices of $x$ and $w$ under $\QQ$, then the images of the ambiguity sets $\mbb G(\Pnom')$ and $\mbb B(\Pnom')$ under $T$ coincide by Proposition~\ref{prop:normal:distance}.
    These observations imply that the affine estimator $\psi\opt$ also solves the estimation problem on the left hand side of~\eqref{eq:gebrich-meets-wasserstein}.
    As $\Pnom$ and $\Pnom'$ share the same first and second moments, the Gebrich ball $\G(\Pnom)$ around the generic distribution~$\Pnom$ coincides with the Gelbrich ambiguity set $\G(\Pnom')$ around the normal distribution~$\Pnom'$. Thus, we find
    \begin{align*}
          \Sup{\QQ \in \G(\Pnom)} \avrisk(\psi\opt, \QQ) = \Sup{\QQ \in \G(\Pnom')} \avrisk(\psi\opt, \QQ) =  \Inf{\psi\in\Ac} \Sup{\QQ \in \mbb G(\Pnom')} \avrisk(\psi, \QQ) =  \Inf{\psi\in\Ac} \Sup{\QQ \in \mbb G(\Pnom)} \avrisk(\psi, \QQ),
    \end{align*}
    where the second equality holds because $\psi^\star$ solves the Gelbrich MMSE estimation problem with the normal nominal distribution $\Pnom'$ on the left hand side of~\eqref{eq:gebrich-meets-wasserstein}. Hence, $\psi\opt$ solves the Gelbrich MMSE estimation problem~\eqref{eq:dro:approx} with the generic nominal distribution $\Pnom$.
\end{proof}


\section{Numerical Solution of Wasserstein MMSE Estimation Problems}
\label{sect:algorithm}

By Corollaries~\ref{cor:primal:refor} and~\ref{corol:dual:refor}, the primal and dual Wasserstein MMSE estimation problems~\eqref{eq:dro} and~\eqref{eq:dual-dro} can be addressed with off-the-shelf SDP solvers. Unfortunately, however, general-purpose interior-point methods quickly run out of memory when the signal dimension~$n$ and the noise dimension~$m$ grow. It is therefore expedient to look for customized first-order algorithms that can handle larger problem instances. 

In this section we develop a Frank-Wolfe method for the nonlinear SDP~\eqref{eq:program:dual}, which is equivalent to the dual Wasserstein MMSE estimation problem~\eqref{eq:dual-dro}. This approach is meaningful because any solution to~\eqref{eq:program:dual} allows us to construct both an optimal estimator as well as a least favorable prior that form a Nash equilibrium in the sense of Corollary~\ref{corol:nash}; see also Corollary~\ref{corol:alternative}. Addressing the nonlinear SDP~\eqref{eq:program:dual} directly with a Frank-Wolfe method has great promise because the subproblems that identify the local search directions can be shown to admit quasi-closed form solutions and can therefore be solved very quickly.

In Section~\ref{sec:fully-adaptive-FW} we first review three variants of the  Frank-Wolfe algorithm corresponding to a static, an adaptive and a more flexible {\em fully} adaptive stepsize rule, and we prove that the fully adaptive rule offers a linear convergence guarantee under standard regularity conditions. In Section~\ref{sect:FW-MMSE} we then show that the nonlinear SDP~\eqref{eq:program:dual} is amenable to the fully adaptive Frank-Wolfe algorithm and can thus be solved efficiently. 

\subsection{Frank-Wolfe Algorithm for Generic Convex Optimization Problems}
\label{sec:fully-adaptive-FW}
Consider a generic convex minimization problem of the form
\begin{equation}
\label{eq:generic-convex}
f\opt = \min_{s \in \mc S}~f(s)
\end{equation}
with a convex compact feasible set $\mc S \subseteq \R^d$ and a convex differentiable objective function $f:\mc S\rightarrow \mbb R$. We assume that for each precision $\delta \in [0, 1]$ we have access to an inexact oracle $F: \mc S \rightarrow \mc S$ that maps any~$s\in\mc S$ to a $\delta$-approximate solution of an auxiliary problem linearized around~$s$. More precisely, we assume that
\begin{align}\label{oracle}
	\left( F(s) - s \right)^\top \nabla f(s) \leq \delta \Min{z \in \mc S}~ \left( z - s \right)^\top \nabla f(s). 
\end{align}
By the standard optimality condition for convex optimization problems, the minimum on the right hand side of~\eqref{oracle} vanishes if and only if~$s$ solves the original problem~\eqref{eq:generic-convex}. Otherwise, the minimum is strictly negative. If $\delta = 1$, then the oracle returns an exact mininizer of the linearized problem. If $\delta=0$, on the other hand, then the oracle returns any solution that is weakly preferred to $s$ in the linearized problem. Given an oracle satisfying~\eqref{oracle}, one can design a Frank-Wolfe algorithm whose iterates obey the recursion 
\begin{align}\label{frank-wolfe}
	s_{t+1} = s_t + \eta_t (F(s_t) - s_t) \quad \forall t\in\mbb N\cup\{0\},
\end{align}
where $s_0\in\mc S$ is an arbitrary initial feasible solution, $\delta$ is a prescribed precision, and $\eta_t \in[0,1]$ is a stepsize that may depend on the current iterate~$s_t$.
The Frank-Wolfe algorithm was originally developed for quadratic programs~\cite{ref:frank1956algorithm} and later extended to general convex programs with differentiable objective functions and compact convex feasible sets \cite{ref:levitin1966constrained, ref:demyanov1970approximate, ref:dunn1978conditional, ref:dunn1979rates, ref:dunn1980convergence}. Convergence guarantees for the Frank-Wolfe algorithm typically rely on the assumption that the gradient of $f$ is Lipschitz continuous \cite{ref:levitin1966constrained, ref:dunn1979rates, ref:dunn1980convergence, ref:garber2015faster, ref:freund2016new}, that $f$ has a bounded curvature constant \cite{ref:clarkson2010coresets, ref:jaggi2013revisiting}, or that the gradient of $f$ is H\"older continuous \cite{ref:nesterov2018complexity}. 

Throughout this section we will assume that the decision variable can be represented as $s=(s^{[1]},\ldots, s^{[K]})$, where $s^{[k]}\in\mbb R^{d_k}$ and $\sum_{k=1}^{K} d_k=d$. Moreover, we will assume that the feasible set~$\mc S = \times_{k=1}^K \mc S^{[k]}$ constitutes a $K$-fold Cartesian product, where the marginal feasible set~$\mc S^{[k]}\subseteq \mbb R^{d_k}$ is convex and compact for each $k=1,\ldots, K$. This assumption is unrestrictive because we are free to set $K=1$ and $\mc S^{[1]}=\mc S$. For ease of notation, we use from now on $\nabla_{[k]}$ to denote the partial gradient with respect to the subvector~$s^{[k]}\in\mc S^{[k]}$,~$k=1,\ldots,K$. 

The subsequent convergence analysis will rely on the following regularity conditions.

\begin{assumption} [Regularity conditions] 
	\label{a:FW} ~
	\begin{enumerate}[label = $(\roman*)$]
		\item \label{a:FW:smooth}
		The objective function is $\beta$-smooth for some $\beta>0$, i.e.,
		\[
		\|\nabla f(s) - \nabla f(\bar s)\| \leq \beta \|s - \bar s\|\quad \forall s,\bar s\in\mc S.
		\]
		
		\item \label{a:FW:set}
		The marginal feasible sets are $\alpha$-strongly convex with respect to $f$ for some $\alpha>0$, i.e.,
		\[
		\theta s^{[k]} + (1-\theta) \bar s^{[k]} - \theta(1-\theta)\frac{\alpha}{2}\left\|s^{[k]} - \bar s^{[k]} \right\|^2 \frac{\nabla_{[k]} f(s)}{\|\nabla_{[k]} f(s)\|} \in \mc S^{[k]} \quad \forall s, \bar s \in \mc S,~\theta \in [0,1],~k=1,\ldots, K.
		\]
		
		\item \label{a:FW:lower}
		The objective function is $\eps$-steep for some $\eps > 0$, i.e.,
		\[
		\| \nabla_{[k]} f(s) \| \ge \eps\quad \forall s \in \mc S,~k=1,\ldots, K.
		\]
	\end{enumerate}
\end{assumption}

Assumption~\ref{a:FW}\,\ref{a:FW:set} relaxes the standard strong convexity condition prevailing in the literature, which is obtained by setting $K=1$ and requiring that the condition stated here remains valid when the normalized gradient~${\nabla_{[1]} f(s) / \|\nabla_{[1]} f(s)\|}$ is replaced with any other vector in the Euclidean unit ball, see, {\em e.g.},~\cite[Equation~(25)]{ref:journee2010generalized}. We emphasize that our weaker condition is sufficient for the standard convergence proofs of the Frank-Wolfe algorithm but is necessary for our purposes because the feasible set of problem~\eqref{eq:program:dual} fails to be strongly convex in the traditional sense. Similarly, Assumption~\ref{a:FW}\,\ref{a:FW:lower} generalizes the usual $\eps$-steepness condition from the literature, which is recovered by setting~$K=1$, see, {\em e.g.}, \cite[Assumption~1]{ref:journee2010generalized}. Under this assumption the gradient never vanishes on $\mc S$, and the minimum of~\eqref{eq:generic-convex} is attained on the boundary of~$\mc S$.

In the following we will distinguish three variants of the Frank-Wolfe algorithm with different stepsize rules. The {\em vanilla Frank-Wolfe} algorithm employs the harmonically decaying static stepsize
\begin{equation*}
	\eta_t = \frac{2}{2 + t},
\end{equation*}
which results in a sublinear $\mc O(1/t)$ convergence whenever Assumption~\ref{a:FW}\,\ref{a:FW:smooth} holds \cite{ref:frank1956algorithm, ref:dunn1978conditional}. The {\em adaptive Frank-Wolfe} algorithm uses the stepsize 
\begin{align}\label{adp_step}
	\eta_t = \min \left\{1, \frac{(s_t - F (s_t))^\top \nabla f(s_t)}{\beta \| s_t - F(s_t) \|^2 } \right\}, 
\end{align}
which adapts to the iterate $s_t$. If all of the Assumptions~\ref{a:FW}\,\ref{a:FW:smooth}--\ref{a:FW:lower} hold, then the adaptive Frank-Wolfe algorithm enjoys a linear~$\mc O (c^t)$ convergence guarantee, where $c\in(0,1)$ is an explicit function of the oracle precision~$\delta$, the smoothness parameter~$\beta$, the strong convexity parameter~$\alpha$ and the steepness parameter~$\eps$~\cite{ref:levitin1966constrained,ref:garber2015faster}. Note that the stepsize~\eqref{adp_step} is constructed as the unique solution of the univariate quadratic program
\begin{equation*}
	\min_{\eta \in [0, 1]} ~ f(s_{t}) - \eta \big(s_{t} - F(s_{t})\big)^\top \nabla f(s_{t}) + \frac{1}{2}{\beta \eta^2} \left\| s_{t} - F(s_{t}) \right\|^2,
\end{equation*}
which minimizes a quadratic majorant of the objective function $f$ along the line segment from $s_t$ to $F(s_t)$. 

The adaptive stepsize rule~\eqref{adp_step} has undergone further scrutiny in \cite{ref:pedregosa2018stepsize}, where it was discovered that one may improve the algorithm's convergence behavior by replacing the global smoothness parameter~$\beta$ in~\eqref{adp_step} with an adaptive smoothness parameter~$\beta_t$ that captures the smoothness of $f$ along the line segment from~$s_t$ to~$F(s_t)$. This extra flexibility is useful because~$\beta_t$ can be chosen smaller than the unnecessarily conservative global smoothness parameter~$\beta$ and because $\beta_t$ is easier to estimate than~$\beta$, which may not even be accessible. 

Following~\cite{ref:pedregosa2018stepsize}, we will henceforth only require that $\beta_t>0$ satisfies the inequality
\begin{align} \label{eq:stepsize_full}
	f\Big( s_t - \eta_t(\beta_t) \big(s_t - F(s_t)\big) \Big) \le f(s_t) - \eta_t(\beta_t) \big(s_t - F(s_t)\big)^\top \nabla f(s_t) + \frac{1}{2}{\beta_t \eta_t(\beta_t)^2} \big\|s_t - F(s_t)\big\|^2,
\end{align}
where $\eta_t(\beta_t)$ is defined as the adaptive stepsize~\eqref{adp_step} with $\beta$ replaced by $\beta_t$. As it adapts both to $s_t$ and $\beta_t$, we will from now on refer to $\eta_t=\eta_t(\beta_t)$ as the {\em fully} adaptive stepsize. The above discussion implies that~\eqref{eq:stepsize_full} is always satisfiable if Assumption~\ref{a:FW}\,\ref{a:FW:smooth} holds, in which case one may simply set~$\beta_t$ to the global smoothness parameter~$\beta$. In practice, however, the inequality~\eqref{eq:stepsize_full} is often satisfiable for much smaller values~$\beta_t \ll \beta$ that may not even be related to the smoothness properties of the objective function. A close upper bound on the smallest~$\beta_t>0$ that satisfies~\eqref{eq:stepsize_full} can be found efficiently via backtracking line search. Specifically, the {\em fully adaptive Frank-Wolfe} algorithm sets $\beta_t$ to the smallest element of the discrete search space $\frac{\beta_{t-1}}{\zeta}\cdot\{1, \tau, \tau^2, \tau^3,\ldots\}$ that satisfies~\eqref{eq:stepsize_full}, where  $\tau>1$ and $\zeta>1$ are prescribed line search parameters. A detailed description of the fully adaptive Frank-Wolfe algorithm in pseudocode is provided in Algorithm~\ref{algorithm:FAFW}.

It has been shown in~\cite{ref:pedregosa2018stepsize} that Algorithm~\ref{algorithm:FAFW} enjoys the same sublinear~$\mc O(1/t)$ convergence guarantee as the vanilla Frank-Wolfe algorithm when Assumption~\ref{a:FW}\,\ref{a:FW:smooth} holds. Below we will leverage techniques from~\cite{ref:levitin1966constrained, ref:garber2015faster} to show that Algorithm~\ref{algorithm:FAFW} offers indeed a linear convergence rate if all of the Assumptions~\ref{a:FW}\,\ref{a:FW:smooth}--\ref{a:FW:lower} hold.

\begin{table}[th]
	\begin{minipage}{0.71\columnwidth}
		\begin{algorithm}[H]
			\caption{Fully adaptive Frank-Wolfe algorithm}
			\label{algorithm:FAFW}
			\begin{algorithmic}
				\REQUIRE initial feasible point $s_0 \in \mc S$, initial smoothness parameter $\beta_{-1} > 0$ \\
				\hspace{2.2em} 
				line search parameters $\tau > 1$, $\zeta > 1$, initial iteration counter $t=0$ \\[0.5ex]
				\WHILE{stopping criterion is not met} \vspace{0.25em}
				\STATE solve the oracle subproblem to find $\tilde s_t = F(s_t)$ 
				\STATE set $d_t \leftarrow \tilde s_{t} - s_t$ and $ g_t \leftarrow -d_t^\top{\nabla f(s_{t})}$ \vspace{0.1em}
				\STATE set $\beta_t \leftarrow \beta_{t-1} / \zeta$ and $\eta \leftarrow \min \{1, g_t/(\beta_t \| d_t \|^2) \} $
				\WHILE{$ \displaystyle f(s_t + \eta d_t) > f(s_t) - \eta g_t + \frac{\eta^2 \beta_t}{2} \| d_t \|^2 $}
				\STATE $\beta_t \leftarrow \tau \beta_t$ and $\eta \leftarrow \min \{1, g_t/(\beta_t \| d_t \|^2) \}$
				\ENDWHILE
				\STATE set $\eta_t \leftarrow \eta$ and $s_{t+1} \leftarrow s_t + \eta_t d_t$
				\STATE set $t \leftarrow t + 1$
				\ENDWHILE \vspace{0.5ex}
				\ENSURE $s_t$
			\end{algorithmic}
		\end{algorithm}
	\end{minipage}
\end{table}

\begin{theorem}[Linear convergence of the fully adaptive Frank-Wolfe algorithm]\label{theorem:algorithm}
	If Assumption~\ref{a:FW} holds and $\overline{\beta} = \max \{ \tau \beta, \beta_{-1} \}$, then Algorithm~\ref{algorithm:FAFW} enjoys the linear convergence guarantee
	\[f(s_t) - f\opt \le \max \left\{ 1 - \frac{\delta}{2} , 1 - \frac{(1-\sqrt{1 - \delta})\alpha \eps}{4 \overline \beta} \right\}^t (f(s_{0}) - f\opt) \quad \forall t\in\mbb N.\]
\end{theorem}
The proof of Theorem~\ref{theorem:algorithm} relies on the following preparatory lemma.

\begin{lemma}[Bounds on the surrogate duality gap]
	\label{lem:FW}
	The surrogate duality gap~$g_t = -d_t^\top \nabla f(s_t)$ corresponding to the search direction~$d_t = F(s_t) - s_t $ admits the following lower bounds.
	\begin{enumerate}[label = $(\roman*)$,itemsep = 2mm]
		\item \label{lem:FW:convex}
		If the objective function $f$ is convex, then $g_t \ge \delta(f(s_t) - f\opt)$.
		\item \label{lem:FW:set}
		If the marginal feasible sets are $\alpha$-strongly convex with respect to $f$ for some $\alpha>0$ in the sense of Assumption~\ref{a:FW}\,\ref{a:FW:set}, then 
		\[
		g_t \geq \min_{k \in \{ 1, \hdots, K \}} \frac{(1-\sqrt{1 - \delta}) \alpha} {2 \delta} \| d_t \|^2 \|\nabla_{[k]} f(s_t) \|.
		\]
	\end{enumerate}
\end{lemma}	    

\begin{proof}[Proof of Lemma~\ref{lem:FW}]
	By the definition of~$g_t$ we have 
	\begin{align}\label{eq:lem:gk}
		g_t = \big(s_t - F(s_t)\big)^\top \nabla f(s_t) \ge \delta \big(s_t - s\big)^\top \nabla f(s_t) \qquad \forall s \in \mc S,
	\end{align}
	where the inequality follows from the defining property~\eqref{oracle} of the inexact oracle with precision~$\delta$. Setting~$s$ in~\eqref{eq:lem:gk} to a global minimizer~$s\opt$ of~\eqref{eq:generic-convex} then implies via the first-order convexity condition for~$f$ that
	$$g_t \ge \delta \big(s_t - s\opt\big)^\top \nabla f(s_t) \ge \delta\big(f(s_t) - f\opt\big)  \,.$$
	This observation establishes assertion~\ref{lem:FW:convex}. To prove assertion~\ref{lem:FW:set}, we first rewrite the estimate~\eqref{eq:lem:gk} as
	\begin{equation}
	\label{eq:cartesian}
	g_t \ge \delta (s_t - s )^\top \nabla f(s) = \delta \sum_{t=1}^T \big(s_t^{[k]} - s^{[k]} \big)^\top \nabla_{[k]} f(s_t) \quad \forall s \in \mc S \,.
	\end{equation}
	In the following, we denote by $F^{[k]}:\mc S\to \mc S^{[k]}$ the $k$-th suboracle for $k=1,\ldots, K$, which is defined through the identity $F=(F^{[1]},\ldots, F^{[K]})$. Similarly, for any $\theta\in[0,1]$, we define $s(\theta)=(s^{[1]}(\theta), \ldots,s^{[K]}(\theta))$ through 
	\begin{align*}
		s^{[k]}(\theta) = \theta F^{[k]}(s_t) + (1-\theta)s_t^{[k]} - \frac{\alpha}{2}\theta(1-\theta)\left\|F^{[k]}(s_t) - s_t^{[k]}\right\|^2 \frac{\nabla_{[k]} f(s_t)}{\|\nabla_{[k]} f(s_t)\|}.
	\end{align*}
	By Assumption~\ref{a:FW}\,\ref{a:FW:set}, we have $s^{[k]}(\theta)\in\mc S^{[k]}$ for every $k=1,\ldots, K$. Thanks to the rectangularity of the feasible set this implies that $s(\theta)\in\mc S$. Setting $s$ in~\eqref{eq:cartesian} to $s(\theta)$, we thus find
	\begin{align*}
		g_t &\ge \delta \sum_{t=1}^T \Big(\theta \big(s_t^{[k]} - F^{[k]}(s_t)\big) + \frac{\alpha}{2} \theta(1-\theta)\left\|F^{[k]}(s_t) - s_t^{[k]}\right\|^2\frac{\nabla_{[k]} f(s_t)}{\|\nabla_{[k]} f(s_t)\|} \Big)^\top \nabla_{[k]} f(s_t)  \\
		& = \delta \left( \theta g_t + \frac{\alpha}{2} \theta (1-\theta) \left[ \sum_{k=1}^K \left\|F^{[k]}(s_t) - s_t^{[k]}\right\|^2 \|\nabla_{[k]} f(s_t)\| \right] \right) \\
		&\geq \delta \left( \theta g_t + \frac{\alpha}{2} \theta (1-\theta) \|F(s_t) - s_t \|^2 \Big(\min_{k \in \{ 1, \hdots, K \} } \|\nabla_{[k]} f(s_t)\| \Big) \right) \quad \forall \theta \in [0,1] \,,
	\end{align*}
	where the equality follows from the definition of~$g_t$, and the last inequality exploits the Pythagorean theorem. Reordering the above inequality to bring $g_t$ to the left hand side yields
	\begin{align}\label{eq:lem:gk2}
		g_t \ge \min_{k \in \{ 1, \hdots, K \} } \frac{\alpha}{2} \|F(s_t) - s_t\|^2 \|\nabla_{[k]} f(s_t)\| \frac{\delta \theta(1-\theta)}{1 - \delta \theta}  \qquad \forall \theta \in [0,1].
	\end{align}
	A tedious but straightforward calculation shows that the lower bound on the right hand side of~\eqref{eq:lem:gk2} is maximized by $\theta\opt=(1 - \sqrt{1 - \delta}) / \delta$. Assertion~\ref{lem:FW:set} then follows by substituting $\theta\opt$ into~\eqref{eq:lem:gk2}. 
\end{proof}

\begin{proof}[Proof of Theorem~\ref{theorem:algorithm}]
	By Assumption~\ref{a:FW}\,\ref{a:FW:smooth} the function $f$ is $\beta$-smooth, and thus one can show that
	\begin{align} 
		\label{eq:surrogate-objective}
		f(s_t + \eta d_t) \leq f(s_t) - \eta g_t + \frac{\eta^2 \beta}{2} \| d_t \|^2 \quad \forall \eta \in [0,1] \,,
	\end{align}
	where the surrogate duality gap~$g_t\ge 0$ and the search direction~$d_t\in\mbb R^d$ are defined as in Lemma~\ref{lem:FW}. We emphasize that~\eqref{eq:surrogate-objective} holds in fact for all $\eta\in\mbb R$. However, the next iterate $s_{t+1}=s_t + \eta d_t$ may be infeasible unless $\eta \in [0,1]$.  The inequality~\eqref{eq:surrogate-objective} implies that any $\beta_t\ge \beta$ satisfies the condition of the inner while loop of Algorithm~\ref{algorithm:FAFW}, and thus the loop must terminate at the latest after $\lceil \log(\zeta\beta/\beta_{-1})/\log(\tau)\rceil$ iterations, outputting a smoothness parameter~$\beta_t$ and a stepsize~$\eta_t$ that satisfy the inequality~\eqref{eq:stepsize_full}.
	We henceforth denote by $h_t = f(s_t) - f\opt$ the suboptimality of the $k$-th iterate and note that
	\begin{equation}
	\label{eq:h_t-recursion}
	h_{t+1} = f(s_t+\eta_t d_t)-f(s_t)+ h_t \le -g_t + \frac{1}{2}\beta_t\eta_t^2 \|d_t\|^2+h_t\,,
	\end{equation}
	where the inequality exploits~\eqref{eq:stepsize_full} and the definitions of $g_t$ and $d_t$. In order to show that $h_t$ decays geometrically, we distinguish the cases $(i)$ $g_t / (\beta_t \|d_t\|^2) \ge 1$ and $(ii)$ $g_t / (\beta_t \|d_t\|^2) < 1$. In case~$(i)$, the stepsize $\eta_t$ defined in~\eqref{eq:stepsize_full} satisfies $\eta_t = \min \{1, g_t / (\beta_t \|d_t\|^2) \} = 1$, and thus we have
	\begin{equation}
	\label{eq:h_t-recursion_i}
	h_{t+1} \leq \left( \frac{\beta_t \| d_t \|^2}{2 g_t} - 1 \right) g_t +h_t \leq -\frac{g_t}{2}+h_t \leq  \left( 1 - \frac{\delta}{2} \right) h_t,
	\end{equation}
	where the first inequality follows from~\eqref{eq:h_t-recursion}, while the third inequality holds due to Lemma~\ref{lem:FW}\,\ref{lem:FW:convex}. 
	
	In case~$(ii)$, the stepsize satisfies~$\eta_t = g_t / \beta_t \|d_t\|^2 < 1$, and thus we find
	\begin{align}
		\nonumber
		h_{t+1} &\leq - g_t + \frac{g_t^2}{2\beta_t \| d_t \|^2} + h_t \leq - \frac{g_t^2}{2 \beta_t \| d_t \|^2} + h_t \leq \left(1 - \frac{\delta g_t}{2 \beta_t \| d_t \|^2}\right) h_t\\
		\label{eq:h_t-recursion_ii}
		&\le \left( 1 - \min_{k \in \{ 1, \hdots, K \} } \frac{(1-\sqrt{1 - \delta}) \alpha}{4 \beta_t} \|\nabla_{[k]} f(s_t)\|\right) h_t
		\le \left( 1 - \frac{(1-\sqrt{1 - \delta}) \alpha\eps}{4 \overline\beta} \right) h_t\,,
	\end{align}
	where the first and the second inequalities follow from~\eqref{eq:h_t-recursion} and from multiplying $-g_t$ with $\eta_t<1$, respectively, while the third and the fourth inequalities exploit Lemmas~\ref{lem:FW}\,\ref{lem:FW:convex} and~\ref{lem:FW}\,\ref{lem:FW:set}, respectively. The last inequality in~\eqref{eq:h_t-recursion_ii} holds because of Assumption~\ref{a:FW}\,\ref{a:FW:lower} and because $\beta_t \leq \overline{\beta}$ for all $t\in\mbb N$; see~\cite[Proposition~2]{ref:pedregosa2018stepsize}. By the estimates~\eqref{eq:h_t-recursion_i} and~\eqref{eq:h_t-recursion_ii}, the suboptimality of the current iterate decays at least~by
	\[
	\max \left\{ 1 - \frac{\delta}{2} , 1 - \frac{(1-\sqrt{1 - \delta})\alpha \eps}{4 \overline \beta} \right\} <1
	\]
	in each iteration of the algorithm. This observation completes the proof.
\end{proof}

\subsection{Frank-Wolfe Algorithm for Wasserstein MMSE Estimation Problems}
\label{sect:FW-MMSE}

We now use the fully adaptive Frank-Wolfe algorithm of Section~\ref{sec:fully-adaptive-FW} to solve the nonlinear SDP~\eqref{eq:program:dual}, which is equivalent to the dual Wasserstein MMSE estimation problem over normal priors. Recall from Corollary~\ref{corol:alternative} that any solution of~\eqref{eq:program:dual} can be used to construct a least favorable prior and an optimal estimator that form a Nash equilibrium. Unlike the generic convex program~\eqref{eq:generic-convex}, the nonlinear SDP~\eqref{eq:program:dual} is a convex {\em maximization} problem. This prompts us to apply Algorithm~\ref{algorithm:FAFW} to the convex minimization problem obtained from problem~\eqref{eq:program:dual} by turning the objective function upside down.

Throughout this section we assume that $\covsa_x\succ 0$, $\covsa_w\succ 0$, $\rho_x>0$ and $\rho_w>0$, which implies via Theorem~\ref{thm:least-favorable-prior} that the nonlinear SDP~\eqref{eq:program:dual} is solvable and can be reformulated more concisely as
\be
\label{eq:program:dual:concise}
\max_{\cov_x\in\mc S^+_x, \cov_w\in\mc S^+_w}~ f(\cov_x, \cov_w)\,,
\ee
where the objective function $f:\mc S^+_x \times \mc S^+_w\to \mbb R$ is defined through
\[
f(\cov_x, \cov_w) = \Tr{\cov_x - \cov_x H^\top \left( H \cov_x H^\top + \cov_w \right)^{-1} H \cov_x}\,,
\]
and where the separate feasible sets for $\cov_x$ and $\cov_w$ are given by
\[
\mc S^+_x = \left\{ \cov_x \in \PSD^n: \Tr{\cov_x + \covsa_x - 2 \big( \covsa_x^\half \cov_x \covsa_x^\half \big)^\half} \leq \rho_x^2,~\cov_x\succeq \lambda_{\min}(\covsa_x)I_n \right\}
\]
and
\[
\mc S^+_w = \left\{ \cov_w \in \PSD^m: \Tr{\cov_w + \covsa_w - 2 \big( \covsa_w^\half \cov_w \covsa_w^\half \big)^\half} \leq \rho_w^2,~\cov_w\succeq \lambda_{\min}(\covsa_w)I_m  \right\}\,,
\]
respectively. One readily verifies that $f$ is concave and differentiable. Moreover, in the terminology of Section~\ref{sec:fully-adaptive-FW}, the feasible set of the nonlinear SDP~\eqref{eq:program:dual:concise} constitutes a Cartesian product of $K=2$ marginal feasible sets~$\mc S^+_x$ and~$\mc S^+_w$, both of which are convex and compact thanks to Lemma~\ref{lemma:compact:FS}. Note that~$\mc S_x^+$ and~$\mc S_w^+$ constitute restrictions of the feasible sets~$\mc S_x$ and~$\mc S_w$, respectively, which appeared in the proofs of Theorems~\ref{thm:least-favorable-prior} and~\ref{thm:sandwich}. The oracle problem that linearizes the objective function of the nonlinear SDP~\eqref{eq:program:dual:concise} around a fixed feasible solution $\cov_x\in\mc S^+_x$ and $\cov^+_w\in\mc S_w$ can now be expressed concisely as
\be
\label{eq:oracle}
\max_{L_x\in\mc S^+_x, L_w\in\mc S^+_w} \inner{\lin_x - \cov_x}{\direc_x} + \inner{\lin_w - \cov_w}{\direc_w}\,,
\ee
where $\direc_x = \nabla_{\Sigma_x} f\big(\Sigma_x, \Sigma_w\big)$ and $\direc_w =\nabla_{\Sigma_w} f\big(\Sigma_x, \Sigma_w\big)$ represent the gradients of $f$ with respect to~$\cov_x$ and~$\cov_w$. Lemma~\ref{lem:Taylor-f} offers analytical formulas for $\direc_x$ and $\direc_w$ and shows that they are both positive semidefinite.

The oracle problem~\eqref{eq:oracle} is manifestly separable in $\lin_x$ and $\lin_w$ and can therefore be decomposed into a sum of two structurally identical marginal subproblems. The Frank-Wolfe algorithm is an ideal method to address the nonlinear SDP~\eqref{eq:program:dual:concise} because these two marginal oracle subproblems admit quasi-closed form solutions. Specifically, Proposition~\ref{prop:quadratic}\,\ref{prop:quad:cov_min} in the appendix implies that problem~\eqref{eq:oracle} is uniquely solved by
\[
L_x\opt = (\dualvar_x\opt)^2 (\dualvar_x\opt I_n - \direc_x)^{-1} \covsa_x (\dualvar_x\opt I_n - \direc_x)^{-1} \quad \text{and}\quad 
L_w\opt = (\dualvar_w\opt)^2 (\dualvar_w\opt I_m - \direc_w)^{-1} \covsa_w (\dualvar_w\opt I_m - \direc_w)^{-1},
\]
where $\dualvar_x\opt\in(\lambda_{\max}(\direc_x),\infty)$ and $\dualvar_w\opt\in(\lambda_{\max}(\direc_w),\infty)$ are the unique solutions of the algebraic equations  
\be
\label{eq:algebraic-equations}
\rho_x^2 - \inner{\covsa_x}{\big( I_n - \dualvar_x\opt (\dualvar_x\opt I_n - D_x)^{-1} \big)^2} = 0\quad \text{and}\quad 
\rho_w^2 - \inner{\covsa_w}{\big( I_m - \dualvar_w\opt (\dualvar_w\opt I_m - D_w)^{-1} \big)^2} = 0,
\ee
respectively. In practice, these algebraic equations need to be solved numerically. 
However, the numerical errors in~$\dualvar_x\opt$ and~$\dualvar_w\opt$ must be contained to ensure that~$L_x\opt$ and~$L_w\opt$ give rise to a $\delta$-approximate solution for~\eqref{eq:oracle} in the sense of~\eqref{oracle}. In the following we will show that $\delta$-approximate solutions for each of the two oracle subproblems in~\eqref{eq:oracle} and for each $\delta\in(0,1)$ can be computed with an efficient bisection algorithm.

\begin{table}[th]
	\begin{minipage}{0.60\columnwidth}
		\begin{algorithm}[H]
			\caption{Bisection algorithm for the oracle subproblem}
			\label{algorithm:bisection}
			\begin{algorithmic}
				\REQUIRE nominal covariance matrix $\covsa \in \PD^d$, radius $\rho \in\mbb R_{++}$, \\	\hspace{2.3em}  reference covariance matrix $\cov \in \PSD^d$ feasible in~\eqref{eq:oracle}, \\
				\hspace{2.3em} gradient matrix $\direc \in \PSD^d$, $\direc\neq 0$, precision $\delta \in (0, 1)$,\\
				\hspace{2.3em} dual objective function $\varphi(\dualvar)$ defined in Theorem~\ref{thm:oracle}
				\vspace{0.21em}
				\STATE set $\lambda_1\leftarrow \lambda_{\max}(D)$, and let $v_1\in\mbb R^d$ be an eigenvector for $\lambda_1$
				\STATE set $\underline\dualvar \leftarrow \lambda_1 ( 1 + (v_1^\top \covsa v_1)^\half / \rho ) $ and $\overline\dualvar \leftarrow \lambda_1 ( 1 + \Tr{\covsa}^{\half} / \rho ) $
				\REPEAT
				\STATE Set $\tilde \dualvar \leftarrow (\overline\dualvar + \underline\dualvar) / 2$ and $\tilde L \leftarrow (\tilde \dualvar)^2 (\tilde \dualvar I_d - \direc)^{-1} \covsa (\tilde \dualvar I_d - \direc)^{-1}$
				\STATE \textbf{if } $\frac{{\rm d} \varphi}{{\rm d}\dualvar}(\tilde \dualvar)<0$ \textbf{ then } Set $\underline\dualvar \leftarrow \tilde \dualvar$ \textbf{ else } Set $\overline\dualvar \leftarrow \tilde \dualvar$ \textbf{ endif}
				\UNTIL{$\frac{{\rm d} \varphi}{{\rm d}\dualvar}(\tilde \dualvar)>0$ and $\inner{\tilde L - \cov}{\direc}  \geq \delta\,\varphi(\tilde \dualvar) }$ \vspace{0.21em}
				\ENSURE $\tilde L$
			\end{algorithmic}
		\end{algorithm}
	\end{minipage}
\end{table}

\begin{theorem}[Approximate oracle] \label{thm:oracle}
	For any fixed $\rho\in \mbb R_{++}$,  $\covsa\in\PD^d$ and $\direc \in\PSD^d$, $\direc\neq 0$, consider the generic oracle subproblem
	\be
	\label{eq:oracle-primal}
	\begin{array}{cl}
		\displaystyle \max_{L \in \PSD^d} & \inner{L - \cov}{\direc}\\[-1.5ex]
		\st & \Tr{L + \covsa - 2 \big( \covsa^\half L \covsa^\half \big)^\half} \leq \rho^2,~L\succeq \lambda_{\min}(\covsa)I_d\,,
	\end{array}
	\ee
	where $\cov\in\PSD^d$ represents a feasible reference solution. Moreover, denote the feasible set of problem~\eqref{eq:oracle-primal} by~$\mc S^+$, let $\delta\in(0,1)$ be the desired oracle precision, and define $\varphi(\dualvar)= \dualvar(\rho^2 + \inner{\dualvar(\dualvar I_d - D)^{-1} - I_d}{\covsa})-\inner{\cov}{D}$ for any $\dualvar>\lambda_{\max}(\direc)$. Then, Algorithm~\ref{algorithm:bisection} returns in finite time a matrix $\tilde L\in\PSD^d$ with the following~properties.
	\begin{enumerate}[label = $(\roman*)$,itemsep = 2mm]
		\item Feasibility: $\tilde L\in\mc S^+$
		\item $\delta$-Suboptimality: $\inner{\tilde L - \cov}{\direc} \geq \delta \max_{L \in \mc S^+} ~ \inner{L - \cov}{\direc}$
	\end{enumerate}
\end{theorem}
\begin{proof}[Proof of Theorem~\ref{thm:oracle}]
	Proposition~\ref{prop:quadratic}\,\ref{prop:quad:cov_min} in the appendix guarantees that the lower bound on $L$ in~\eqref{eq:oracle-primal} is redundant and can be omitted without affecting the problem's optimal value. By Proposition~\ref{prop:quadratic}\,\ref{prop:quad:dual}, the oracle subproblem~\eqref{eq:oracle-primal}
	thus admits the strong Lagrangian dual
	\[
	\min_{\dualvar > \lambda_{\max}(D)}~\varphi(\dualvar)\,,
	\]
	where the convex and differentiable function $\varphi(\dualvar)$ is defined as in the theorem statement. In the following we denote by~$\lambda_1>0$ the largest eigenvalue of $\direc$ and let $v_1\in\mbb R^d$ be a corresponding eigenvalue. By Proposition~\ref{prop:quadratic}\,\ref{prop:quad:cov_min}, the dual oracle subproblem admits a minimizer $\dualvar\opt\in[\underline \dualvar,\overline\dualvar]$ that is uniquely determined by the first-order optimality condition $\frac{{\rm d} \varphi}{{\rm d}\dualvar}(\dualvar\opt)=0$, where
	\[
	\underline\dualvar= \lambda_1\left( 1+\sqrt{v_1^\top\covsa v_1}/\rho \right)\quad \text{and}\quad\overline\dualvar= \lambda_1\left( 1+\sqrt{\Tr{\covsa}}/\rho\right),
	\]
	while the primal problem~\eqref{eq:oracle-primal} admits a unique maximizer $L\opt=L(\dualvar\opt)$, where
	\[
	L(\dualvar) = \dualvar^2 (\dualvar I_d - \direc)^{-1} \covsa (\dualvar I_d - \direc)^{-1}.
	\]
	From the proof of Proposition~\ref{prop:quadratic}\,\ref{prop:quad:cov_min} it is evident that $L(\dualvar)\succ \lambda_{\min}(\covsa)I_d$ for every $\dualvar>0$.
	A direct calculation further shows that
	\begin{align*}
		\frac{{\rm d}\varphi}{{\rm d}\dualvar}(\dualvar)=\rho^2 - \inner{\covsa}{\big( I_d - \dualvar (\dualvar I_d - \direc)^{-1} \big)^2} 
		=\rho^2- \Tr{L(\dualvar) + \covsa - 2 \big( \covsa^\half L(\dualvar) \covsa^\half \big)^\half}.
	\end{align*}
	Recalling that $\varphi(\dualvar)$ is convex and that its derivative is non-negative for all~$\dualvar\ge \dualvar\opt$, the above reasoning implies that $L(\dualvar)\in\mc S^+$ for all $\dualvar\ge \dualvar\opt$. Note also that the optimal value of the primal problem~\eqref{eq:oracle-primal} is non-negative because $\cov \in \mc S^+$. The continuity of $\inner{L(\dualvar)-\cov}{\direc}$ at $\dualvar=\dualvar\opt$ thus ensures that there exists $\delta'>0$ with
	\[
	\inner{L(\dualvar)- \cov}{\direc} \leq \delta \inner{L(\dualvar\opt)- \cov}{\direc} = \delta \max_{L \in \mc S^+} ~ \inner{L - \cov}{\direc}\quad 
	\forall \dualvar\in[\dualvar\opt,\dualvar\opt+\delta'].
	\]
	In summary, computing a feasible and $\delta$-suboptimal matrix $\tilde L\in\PSD^d$ is tantamount to finding $\tilde \dualvar\in[\dualvar\opt,\dualvar\opt+\delta']$. Algorithm~\ref{algorithm:bisection} uses bisection over the interval $[\underline \dualvar, \overline\dualvar]$ to find a $\tilde\dualvar$ with these properties. 
\end{proof}

Theorem~\ref{thm:oracle} complements~\cite[Theorem~3.2]{ref:shafieezadeh2018wasserstein}, which constructs an approximate oracle for a nonlinear SDP similar to~\eqref{eq:program:dual:concise} that offers an {\em additive} error guarantee. The {\em multiplicative} error guarantee of the oracle constructed here is needed to ensure the linear convergence of the fully adaptive Frank-Wolfe algorithm. Next, we prove that the nonlinear SDP~\eqref{eq:program:dual:concise} satisfies all regularity conditions listed in Assumption~\ref{a:FW}.

\begin{proposition}[Regularity conditions of the nonlinear SDP~\eqref{eq:program:dual:concise}] \label{prop:regularity}
	If $\rho_x \in \RR_{++}$, $\rho_w \in \RR_{++}$, $\covsa_x \in \PSD^n$ and $\covsa_w \in \PSD^m$, then the nonlinear SDP~\eqref{eq:program:dual:concise} obeys the following regularity conditions.
	
	\begin{enumerate}[label = $(\roman*)$,itemsep = 2mm]
		\item \label{prop:regularity-i} The objective function of problem~\eqref{eq:program:dual:concise} is $\beta$-smooth in the sense of Assumption~\ref{a:FW}\,\ref{a:FW:smooth}, where
		\[
		\beta = 2\lambda_{\min}^{-1}(\covsa_w) \left( C + C \; \lambda_{\max}^2(H^\top H) + \lambda_{\max}(H^\top H) \right),
		\]
		which depends on the auxiliary constant $C = \lambda_{\max} ( H^\top H ) \cdot \lambda_{\min}^{-2}(\covsa_w) \cdot ( \rho_x +  \Tr{\covsa_x}^\half )^4$.
		
		\item \label{prop:regularity-ii} The marginal feasible sets $\mc S_x^+$ and $\mc S^+_w$ of problem~\eqref{eq:program:dual:concise} are $\alpha$-strongly convex with respect to $-f$ in the sense of Assumption~\ref{a:FW}\,\ref{a:FW:set}, where $\alpha = \min \, \{\alpha_x, \alpha_w \}$, which depends on the auxiliary constants
		\[
		\alpha_x = \frac{\lambda_{\min}^{\frac{5}{4}}(\covsa_x)}{2 \rho_x \big( \rho_x + \Tr{\covsa_x}^\half \big)^{\frac{7}{2}} } \quad\text{and}\quad \alpha_w = 		\frac{\lambda_{\min}^{\frac{5}{4}}(\covsa_w)}{2 \rho_w \big( \rho_w + \Tr{\covsa_w}^\half \big)^{\frac{7}{2}} }\,.
		\]
		\item \label{prop:regularity-iii}
		The objective function of problem~\eqref{eq:program:dual:concise} is $\eps$-steep in the sense of Assumption~\ref{a:FW}\,\ref{a:FW:lower}, where $\eps = \min \, \{\eps_x, \eps_w \}$, which depends on the auxiliary constants
		\begin{align*}
			\eps_x &= \left( \frac{\lambda_{\min}(\covsa_w)}{\big( \rho_x + \Tr{\covsa_x}^\half \big)^2 \lambda_{\max}(H^\top H)+ \big(\rho_w + \Tr{\covsa_w}^{\half} \big)^2} \right)^2 
		\end{align*}
		and
		\begin{align*}
			\eps_w &= \lambda_{\max}(H^\top H) \left( \frac{\lambda_{\min}(\covsa_x)}{\big( \rho_w + \Tr{\covsa_w}^\half \big)^2 + \lambda_{\min}(\covsa_x) \lambda_{\max}(H^\top H)} \right)^2.
		\end{align*}
	\end{enumerate}
\end{proposition}

\begin{proof}[Proof of Proposition~\ref{prop:regularity}]
	The proof repeatedly uses the fact that, for any $A \in \RR^{d_1 \times d_2}$ and $ B \in \PSD^{d_2}$, we have
	\begin{equation}\label{eq:eig}
	\lambda_{\max}(A B A^\top)
	= \lambda_{\max}(A^\top A B)
	\leq \lambda_{\max}(A^\top A) \, \lambda_{\max}(B).
	\end{equation}
	The equality in~\eqref{eq:eig} holds because all eigenvalues of $A B A^\top$ are non-negative and because the non-zero spectrum of $A B A^\top$ is identical to that of $A^\top A B$ due to \cite[Proposition~4.4.10]{ref:bernstein2009matrix}. The inequality follows from the observation that $\lambda_{\max}(A^\top A)$ and $\lambda_{\max}(B)$ coincide with the operator norms of the positive semidefinite matrices $A^\top A$ and $B$, respectively. 
	
	As for assertion~\ref{prop:regularity-i}, recall first that the objective function$~f$ of the nonlinear SDP~\eqref{eq:program:dual} is concave. In order to show that $f$ is $\beta$-smooth for some $\beta>0$, it thus suffices to prove that the largest eigenvalue of the positive semidefinite Hessian matrix of $-f$ admits an upper bound uniformly across $\mc S^+_x \times \mc S^+_w$. By Lemma~\ref{lem:Taylor-f}, the partial gradients of~$f$ evaluated at~$\cov_x\succ 0$ and~$\cov_w\succ 0$ are given by
    \begin{align*}
	\direc_x &= \nabla_{\Sigma_x} f\big(\Sigma_x, \Sigma_w\big) = (I_n - \cov_x H^\top G^{-1}H )^\top (I_n - \cov_x H^\top G^{-1}H)\\
	\direc_w &= \nabla_{\Sigma_w} f\big(\Sigma_x, \Sigma_w\big) = G^{-1} H \cov_x^2 H^\top G^{-1}\,,
    \end{align*}
    where $G= H\cov_x H^\top +\cov_w$. Moreover, the Hessian matrix 
    \[
	\mc H = 
	\begin{bmatrix}
	\mc H_{xx} & \mc H_{xw} \\
	\mc H_{xw}^\top & \mc H_{ww}
	\end{bmatrix} \succeq 0
	\]
	of the convex function~$-f$ evaluated at $\cov_x\succ 0$ and~$\cov_w\succ 0$ consists of the submatrices
    \begin{align*}
    \begin{array}{l@{\;}l@{\;}l@{\;\;}l}
	\mc H_{xx} & = -\nabla^2_{xx} f(\cov_x, \cov_w) &=& 2 \direc_x \otimes H^\top G^{-1} H \\[0.5ex]
	\mc H_{xw} &= -\nabla^2_{xw} f(\cov_x, \cov_w) &=& H^\top G^{-1} \otimes (H^\top \direc_w - \cov_x H^\top G^{-1}) +  (H^\top \direc_w - \cov_x H^\top G^{-1}) \otimes H^\top G^{-1} \\[0.5ex]
	\mc H_{ww} & = -\nabla^2_{ww} f(\cov_x, \cov_w) &=& 2 \direc_w \otimes G^{-1},
	\end{array}
	\end{align*}
	where $\nabla_x$ and $\nabla_w$ are used as shorthands for the nabla operators with respect to $\vect (\cov_x)$ and $\vect (\cov_w)$, respectively. To construct an upper bound on $\lambda_{\max}(\mc H)$ uniformly across $\mc S^+_x \times \mc S^+_w$, we note first that
	\be
	\label{eq:hessian-bound}
	\lambda_{\max}(\mc H) \leq \lambda_{\max}(\mc H_{xx}) + \lambda_{\max}(\mc H_{ww}) = 2 \left( \lambda_{\max} \left( \direc_x \right) \lambda_{\max} \left( H^\top G^{-1} H \right) + \lambda_{\max} \left( \direc_w \right) \lambda_{\max} \left( G^{-1} \right) \right),
	\ee
	where the inequality follows from \cite[Fact~5.12.20]{ref:bernstein2009matrix} and the subadditivity of the maximum eigenvalue, whereas the equality exploits the trace rule of the Kronecker product  \cite[Proposition~7.1.10]{ref:bernstein2009matrix}. In the remainder of the proof, we derive an upper bound for each term on the right hand side of the above expression.
	
	By the definition of $G$ and because $\cov_w\in\mc S^+_w$, we have $G \succeq \lambda_{\min}(\cov_w) I_m \succeq \lambda_{\min}(\covsa_w) I_m$. As $\covsa_w\succ 0$ by assumption, we may thus conclude that $\lambda_{\max} ( G^{-1} ) \leq \lambda_{\min}^{-1}(\covsa_w)$, which in turn implies via~\eqref{eq:eig} that
	\[
	\lambda_{\max} ( H^\top G^{-1} H ) \leq \lambda_{\max}(G^{-1}) \, \lambda_{\max}(H H^\top) \leq \lambda_{\min}^{-1}(\covsa_w) \, \lambda_{\max}(H^\top H).
	\]
	Similarly, by the definition of~$\direc_w$ we find
	\begin{align*}
		\lambda_{\max} ( \direc_w )
		&= \lambda_{\max} (G^{-1} H \cov_x^2 H^\top G^{-1}) \\
		&\leq \lambda_{\max}^2 ( \cov_x ) \, \lambda_{\max}^2 ( G^{-1} ) \, \lambda_{\max} ( H^\top H ) \leq \big( \rho_x + \Tr{\covsa_x}^{\half} \big)^4 \, \lambda_{\min}^{-2} ( \covsa_w ) \, \lambda_{\max} ( H^\top H ),
	\end{align*}
	where the first inequality follows from applying the estimate~\eqref{eq:eig} twice, while the last inequality reuses the bound on~$\lambda_{\max} ( G^{-1} )$ derived above and exploits Lemma~\ref{lemma:compact:FS}. Finally, by the definition of~$\direc_w$ we have
	\begin{align*}
		\lambda_{\max} ( \direc_x ) &\leq 1+\lambda_{\max} ( -H^\top G^{-1} H \cov_x )+\lambda_{\max} ( -\cov_x H^\top G^{-1} H) + \lambda_{\max} ( H^\top G^{-1} H \cov_x^2 H^\top G^{-1} H ) \\&\leq 1 + \lambda_{\max} ( H^\top G^{-1} H \cov_x^2 H^\top G^{-1} H ) \\
		&\leq 1 + \lambda_{\max}^2 ( \cov_x ) \, \lambda_{\max}^2 ( G^{-1} ) \, \lambda_{\max}^2 ( H^\top H ) \\
		&\leq 1 + \big( \rho_x + \Tr{\covsa_x}^{\half} \big)^4 \, \lambda_{\min}^{-2}(\covsa_w) \, \lambda_{\max}^2 ( H^\top H ).
	\end{align*}
	where the first inequality holds due to the subadditivity of the maximum eigenvalue and~\cite[Proposition~4.4.10]{ref:bernstein2009matrix}, which implies that the nonzero spectra of $-\cov_x H^\top G^{-1} H$ and $-H^\top G^{-1} H \cov_x$ are both real and coincide with the nonzero spectrum of the negative semidefinite matrix $-\cov_x^\frac{1}{2} H^\top G^{-1} H \cov_x^\frac{1}{2}$. The third inequality follows from applying the estimate~\eqref{eq:eig} three times, and the fourth inequality reuses the bound on~$\lambda_{\max} ( G^{-1} )$ and exploits Lemma~\ref{lemma:compact:FS}. Substituting all the above bounds into~\eqref{eq:hessian-bound} completes the proof of assertion~\ref{prop:regularity-i}.

	As for assertion~\ref{prop:regularity-ii}, we first show that the feasible set $\mc S^+_x$ is $\alpha_x$-strongly convex with respect to $-f$ in the sense of Assumption~\ref{a:FW}\,\ref{a:FW:set}. To see this, fix any $\cov_x, \cov_x' \in \mc S^+_x$ and $\theta \in [0,1]$, and set
	\begin{align}\label{cov_theta}
		\cov_{\theta} = \theta \cov_x + (1-\theta) \cov_x' + \theta(1-\theta)\frac{\alpha_x}{2} \| \cov_x - \cov_x' \|^2 \frac{\direc_x}{\| \direc_x \|},
	\end{align}
	where $\alpha_x>0$ is defined as in the proposition statement, and $\direc_x$ denotes again the partial gradient of $f$ with respect to $\cov_x$. To prove strong convexity of $\mc S^+_x$ with respect to $-f$, we will show that $\cov_{\theta}\in \mc S^+_x$. Note first that $\cov_\theta \succeq \lambda_{\min}(\covsa_x) I_n$ because $\cov_x, \cov_x' \in \mc S^+_x$ and because $\direc_x$ is positive semidefinite. Next, define
	\[
	\overline {\mc S}_x = \left\{ \cov_x \in \PSD^n: \lambda_{\min}(\covsa_x) I_n \preceq \cov_x \preceq (\rho_x + \Trace\big[\covsa_x\big]^{\half})^2 I_n  \right\},
	\]
	and note that $\mc S^+_x \subseteq \overline{\mc S}_x$ by Lemma~\ref{lemma:compact:FS}. Moreover, \cite[Theorem~1]{ref:bhatia2018strong} implies that the function $g_x:\PSD^n\to\mbb R$ defined through $g_x(\cov_x)=\Tr{\cov_x + \covsa_x - 2(\covsa_x^{\half} \cov_x \covsa_x^{\half})^{\half}}$ is $\kappa_1$-strongly convex and $\kappa_2$-smooth over~$\overline {\mc S}_x$, where
	\begin{equation*}
		\kappa_1 = \frac{\lambda_{\min}^{\half}(\covsa_x)}{2 (\rho_x + \Trace\big[ \covsa_x \big]^\half)^3} \quad\text{and}\quad 
		\kappa_2 = \frac{\rho_x + \Trace\big[\covsa_x\big]^\half}{2 \lambda_{\min}^{\frac{3}{2}}(\covsa_x)}.
	\end{equation*}
	By \cite[Theorem~12]{ref:journee2010generalized}, the sublevel set $\mc S_x^+=\{\cov_x\in\overline{\mc S}_x:g_x(\cov_x)\le \rho_x^2\}$ is thus strongly convex in the canonical sense---relative to $\overline{\mc S}_x$---with convexity parameter $\alpha_x=\kappa_1 / (\sqrt{2 \kappa_2} \rho_x)$. This insight implies that $g_x(\overline \cov_x)\le \rho_x^2$, which in turn ensures that $\overline\cov_x\in\mc S_x^+$. As $\theta\in[0,1]$ was chosen arbitrarily, we may conclude that~$\mc S_x^+$ is $\alpha_x$-strongly convex with respect to~$f$. Using an analogous argument, one can show that~$\mc S^+_w$ is $\alpha_w$-strongly convex with respect to $f$, where $\alpha_w>0$ is defined as in the proposition statement. In summary, $\mc S^+_x \times \mc S^+_w$ is therefore $\alpha$-strongly convex with respect to $f$ in the sense of Assumption~\ref{a:FW}\,\ref{a:FW:set}, where $\alpha=\min \{ \alpha_x, \alpha_w \}$.
	
	In order to prove assertion~\ref{prop:regularity-iii}, we will establish lower bounds on $\lambda_{\max}(D_x)$ and $\lambda_{\max}(D_w)$ uniformly across $\mc S^+_x \times \mc S^+_w$. The claim then follows from the observation that $\lambda_{\max}(D_x)\leq \|\direc_x\|$ and $\lambda_{\max}(D_w)\leq\|\direc_w\|$. We first derive a lower bound on~$\lambda_{\max}(D_x)$. To this end, set $T_1 = I_n - \cov_x H^\top G^{-1}H$ and $T_2 = H \cov_x H^\top G^{-1}$, and note that both $T_1$ and $T_2$ have real spectra thanks to~\cite[Proposition~4.4.10]{ref:bernstein2009matrix}. As $D_x = T_1^\top T_1$, we find
	\begin{align} \label{eq:Dx-bound-1}
		\lambda_{\max}(D_x) 
		= \lambda_{\max} (T_1 T_1^\top)
		\geq \max_{\lambda \in \spec(T_1)} | \lambda |^2 = \max_{\lambda \in \spec(T_2)} | 1 - \lambda |^2 = \lambda_{\max}^2(I - T_2) = \lambda_{\max}^2(\cov_w G^{-1}),
	\end{align}			
	where $\spec(T)$ denotes the eigenvalue spectrum of any square matrix $T$. The inequality in~\eqref{eq:Dx-bound-1} follows from Browne's theorem \cite[Fact~5.11.21]{ref:bernstein2009matrix}, the second equality holds because the non-zero spectrum of $\cov_x H^\top G^{-1}H$ matches that of $ H \cov_x H^\top G^{-1}$ thanks to~\cite[Proposition~4.4.10]{ref:bernstein2009matrix}, and the last equality follows from the identity
	$T_2= I_m - \cov_w G^{-1}$.
	Notice that all eigenvalues of $\cov_w G^{-1}$ are real because $T_2$ has a real spectrum. 
	
	The estimate~\eqref{eq:Dx-bound-1} implies that a uniform lower bound on the largest eigenvalue of $D_x$ can be obtained by maximizing the largest eigenvalue of $\cov_w G^{-1}$ over $\mc S^+_x \times \mc S^+_w$. By the definition of $G$, we have
	\be
	\label{eq:Dx-bound-2}
	\begin{aligned}
		\inf_{\substack{\cov_x \in \mc S^+_x \\ \cov_w \in \mc S^+_w}} \lambda_{\max}(\cov_w G^{-1})
		&\ge \inf_{\substack{\cov_x \in \mc S^+_x \\ \cov_w \in \mc S^+_w}} \lambda_{\min} \big(\cov_w (H \cov_x H^\top + \cov_w)^{-1} \big) \\
		&\ge \inf_{\substack{\cov_x \in \mc S^+_x \\ \cov_w \in \mc S^+_w}} \lambda_{\min}(\cov_w) \,\lambda_{\min} \left( (H \cov_x H^\top + \cov_w)^{-1} \right) \\
		&\ge \inf_{\substack{\cov_x \in \mc S^+_x \\ \cov_w \in \mc S^+_w}} \frac{\lambda_{\min}(\cov_w)}{\lambda_{\max}(\cov_x) \lambda_{\max}(H^\top H)+ \lambda_{\max}(\cov_w)} \\
		&\ge \frac{\lambda_{\min}(\covsa_w)}{\big( \rho_x + \Trace\big[\covsa_x\big]^\half \big)^2 \lambda_{\max}(H^\top H)+ \big(\rho_w + \Trace\big[\covsa_w\big]^{\half} \big)^2}, \end{aligned}
	\ee
	where the second inequality holds because $\lambda_{\min}(T)=\lambda_{\max}^{-1}(T^{-1})$ for any $T\succ 0$ and because the maximum eigenvalue of a positive definite matrix coincides with its operator norm. The third inequality exploits~\eqref{eq:eig} and the subadditivity of the maximum eigenvalue, and the last inequality follows from Lemma~\ref{lemma:compact:FS}.
	Combining~\eqref{eq:Dx-bound-1} and~\eqref{eq:Dx-bound-2} shows that
	$ \| D_x \| \geq \lambda_{\max}(D_x) \geq \eps_x, $
	where $\eps_x$ is defined as in the proposition statement.
	
	Using similar arguments, we can also derive a uniform lower bound on~$\lambda_{\max}(D_w)$. Specifically, we have
	\be
	\label{eq:lb1}
	\begin{aligned}
		\lambda_{\max}(D_w) \lambda_{\max}(H^\top H)
		&\geq \lambda_{\max}(H^\top D_w H)
		= \lambda_{\max} \big(H^\top G^{-1} H \cov_x \, (H^\top G^{-1} H \cov_x)^\top \big)\\
		&\geq \lambda_{\max}^2(H \cov_x H^\top G^{-1})
		= \left( 1 - \lambda_{\min}(\cov_w G^{-1}) \right)^2\\ & = \left( 1 - \frac{1}{1 + \lambda_{\max}(H \cov_x H^\top \cov_w^{-1})} \right)^2\\ & = \left( 1 - \frac{1}{1 + \lambda_{\max}\left(\cov_w^{-\frac{1}{2}} H \cov_x H^\top \cov_w^{-\frac{1}{2}}\right)} \right)^2,
	\end{aligned}
	\ee
	where the first inequality follows from~\eqref{eq:eig}, and the first equality holds due to the definition of $\direc_w$. Moreover, the second inequality exploits Brown's theorem~\cite[Fact~5.11.21]{ref:bernstein2009matrix}, and the second equality uses the definition of~$G$. Finally, the third equality follows from the relation $\lambda_{\min}(\cov_w G^{-1}) = \lambda_{\max}^{-1}(G \cov_w)$, and the fourth equality holds due to~\cite[Proposition~4.4.10]{ref:bernstein2009matrix}.
	A uniform lower bound on $\direc_w$ can thus be obtained from the estimate  
	\[
	\cov_w^{-\frac{1}{2}} H \cov_x H^\top \cov_w^{-\frac{1}{2}} \succeq \lambda_{\min}(\cov_x) \cov_w^{-\frac{1}{2}} H H^\top \cov_w^{-\frac{1}{2}} \succeq \frac{\lambda_{\min}(\cov_x)}{\lambda_{\max}(\cov_w)} H H^\top,
	\]
	which implies via Lemma~\ref{lemma:compact:FS} that
	\begin{align}
		\label{eq:lb2}
		\inf_{\substack{\cov_x \in \mc S^+_x \\ \cov_w \in \mc S^+_w}} \lambda_{\max}(\cov_w^{-\frac{1}{2}} H \cov_x H^\top \cov_w^{-\frac{1}{2}}) 
		&\geq \inf_{\substack{\cov_x \in \mc S^+_x \\ \cov_w \in \mc S^+_w}} \frac{\lambda_{\min}(\cov_x)}{\lambda_{\max}(\cov_w)} \lambda_{\max}(H^\top H) \geq \frac{\lambda_{\min}(\covsa_x)}{\big(\rho_w + \Trace\big[\covsa_w\big]^{\half} \big)^2} \lambda_{\max}(H^\top H).
	\end{align}
	Combining \eqref{eq:lb1} and \eqref{eq:lb2} shows that $\| \direc_w \| \geq \lambda_{\max}(\direc_w) \geq \eps_w$, where $\eps_w$ is defined as in the proposition statement. We thus conclude that $f$ is $\eps$-steep in the sense of Assumption~\ref{a:FW}\,\ref{a:FW:lower} with $\eps = \min \{\eps_x, \eps_w \}$.
\end{proof}

By Theorem~\ref{theorem:algorithm}, which is applicable because of Proposition~\ref{prop:regularity}, the fully adaptive Frank-Wolfe algorithm (see Algorithm~\ref{algorithm:FAFW}) solves the (minimization problem equivalent to the) nonlinear SDP~\eqref{eq:program:dual:concise} at a linear convergence rate. Moreover,  Theorem~\ref{thm:oracle} ensures that the oracle problem~\eqref{eq:oracle}, which needs to be solved in each iteration of Algorithm~\ref{algorithm:FAFW}, can be solved highly efficiently via bisection (see Algorithm~\ref{algorithm:bisection}). 

We emphasize that if $\covsa_x$ is singular, then the strong convexity parameter~$\alpha$ of Assumption~\ref{a:FW}\,\ref{a:FW:set} vanishes, and therefore Theorem~\ref{theorem:algorithm} is no longer applicable. In that case, however, Algorithm~\ref{algorithm:FAFW} is still guaranteed to converge, albeit at a sublinear rate; see \cite{ref:pedregosa2018stepsize} for further details.

\section{Numerical Experiments}
\label{sect:numerical}

All experiments are run on an Intel XEON CPU with 3.40GHz clock speed and 16GB of RAM. All (linear) SDPs are solved with MOSEK~8 using the YALMIP interface~\cite{ref:lofberg2004YALMIP}. In order to ensure the reproducibility of our experiments, we make all source codes available at \url{https://github.com/sorooshafiee/WMMSE}.

\subsection{Scalability of the Frank-Wolfe Algorithm}
\label{sec:scalability}
We first compare the convergence behavior of the Frank-Wolfe algorithm developed in Section~\ref{sect:algorithm} against that of MOSEK. Each experiment consists of $10$ independent simulation runs, in all of which we fix the signal and noise dimensions to $n=m=d$ and the Wasserstein radii to $\rho_x = \rho_w = \sqrt{d}$ for some~$d\in \mathbb N$. In each simulation run we randomly generate two nominal covariance matrices $\covsa_x$ and $\covsa_w$ as follows. First, we sample $Q_x$ and $Q_w$ from the standard normal distribution on $\RR^{d\times d}$, and we denote by $R_x$ and $R_w$ the orthogonal matrices whose columns correspond to the orthonormal eigenvectors of $Q_x + Q_x^\top$ and $Q_w+Q_w^\top$, respectively. Then, we set $\covsa_x = R_x \Lambda_x (R_x)^\top$ and $\covsa_w= R_w \Lambda_w R_w^\top$, where $\Lambda_x$ and $\Lambda_w$ are diagonal matrices whose main diagonals are sampled uniformly from $[1,5]^d$ and $[1,2]^d$, respectively. Finally, we set $\msa_x=0$ and $\msa_w=0$. The Wasserstein MMSE estimator can then be computed either by solving the nonlinear SDP~\eqref{eq:program:dual} with a Frank-Wolfe algorithm or by solving the linear SDP~\eqref{eq:SDP:dual} with MOSEK. Figures~\ref{fig:algorithm:a} and~\ref{fig:algorithm:b} show the execution time and the number of iterations needed by the vanilla, adaptive and fully adaptive versions of the Frank-Wolfe (FW) algorithm as well as by MOSEK to drive the (surrogate) duality gap below~$10^{-3}$. MOSEK runs out of memory for all dimensions~$d>100$. Figure~\ref{fig:algorithm:c} visualizes the empirical convergence behavior of the three different Frank-Wolfe algorithms. We observe that the fully adaptive Frank-Wolfe algorithm finds highly accurate solutions already after $20$ iterations for problem instances of dimension $d=1{,}000$.

\begin{figure*}
	\centering
	\subfigure[Scaling of execution time]{\label{fig:algorithm:a} \includegraphics[width=0.31\columnwidth]{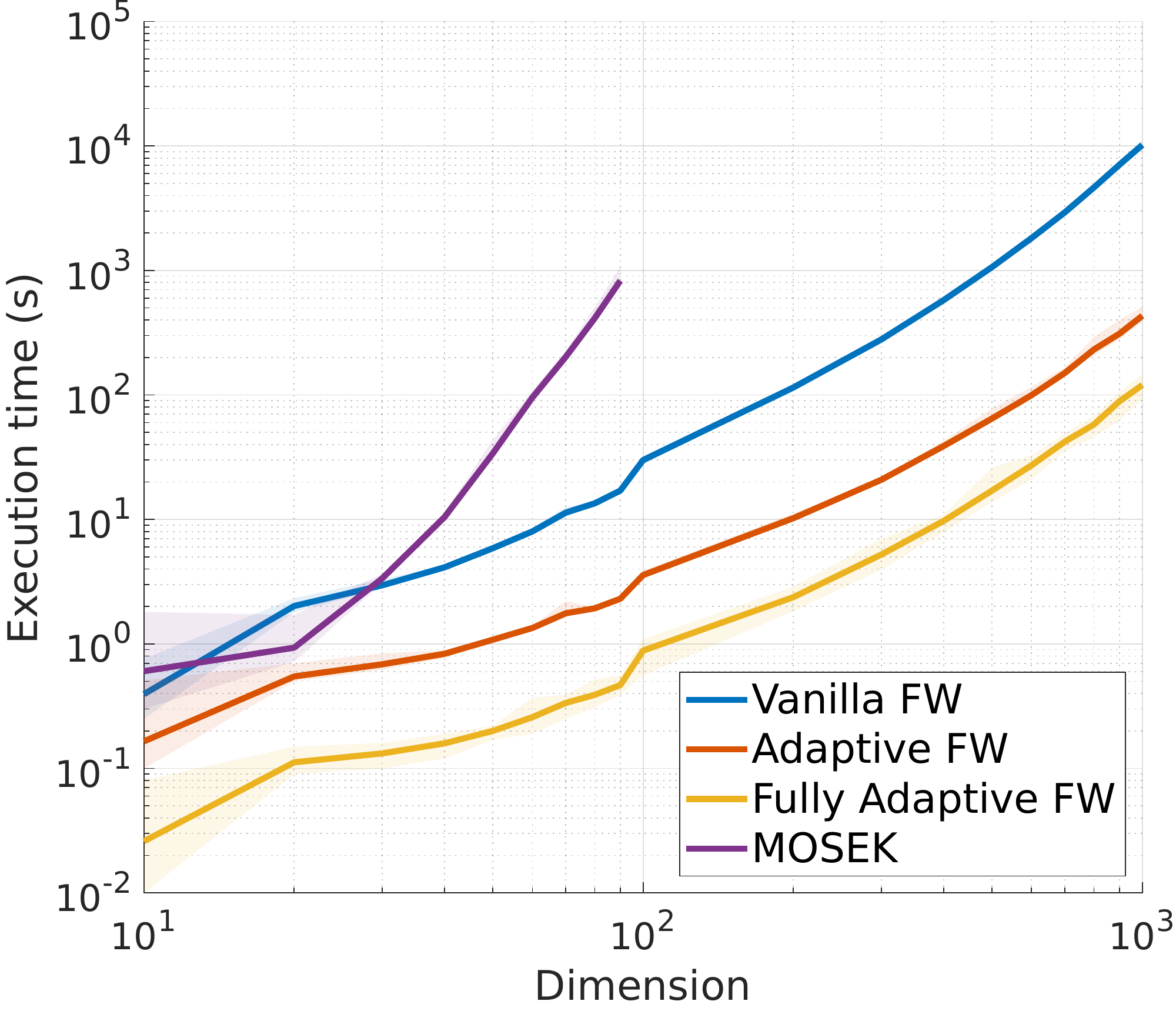}} \hspace{0pt}
	\subfigure[Scaling of iteration count]{\label{fig:algorithm:b} \includegraphics[width=0.31\columnwidth]{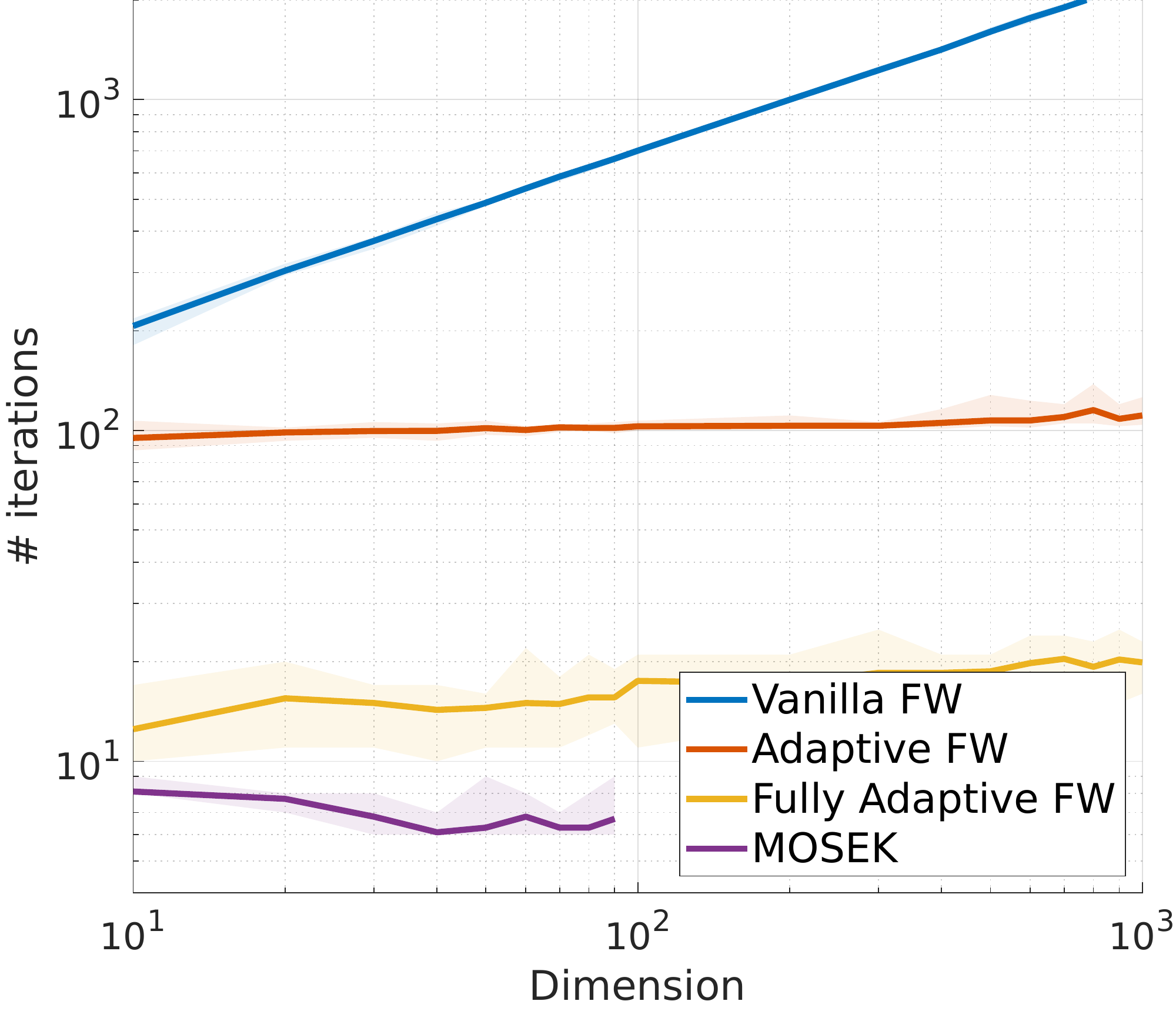}} \hspace{0pt}
	\subfigure[Convergence for $d=1000$]{\label{fig:algorithm:c} \includegraphics[width=0.31\columnwidth]{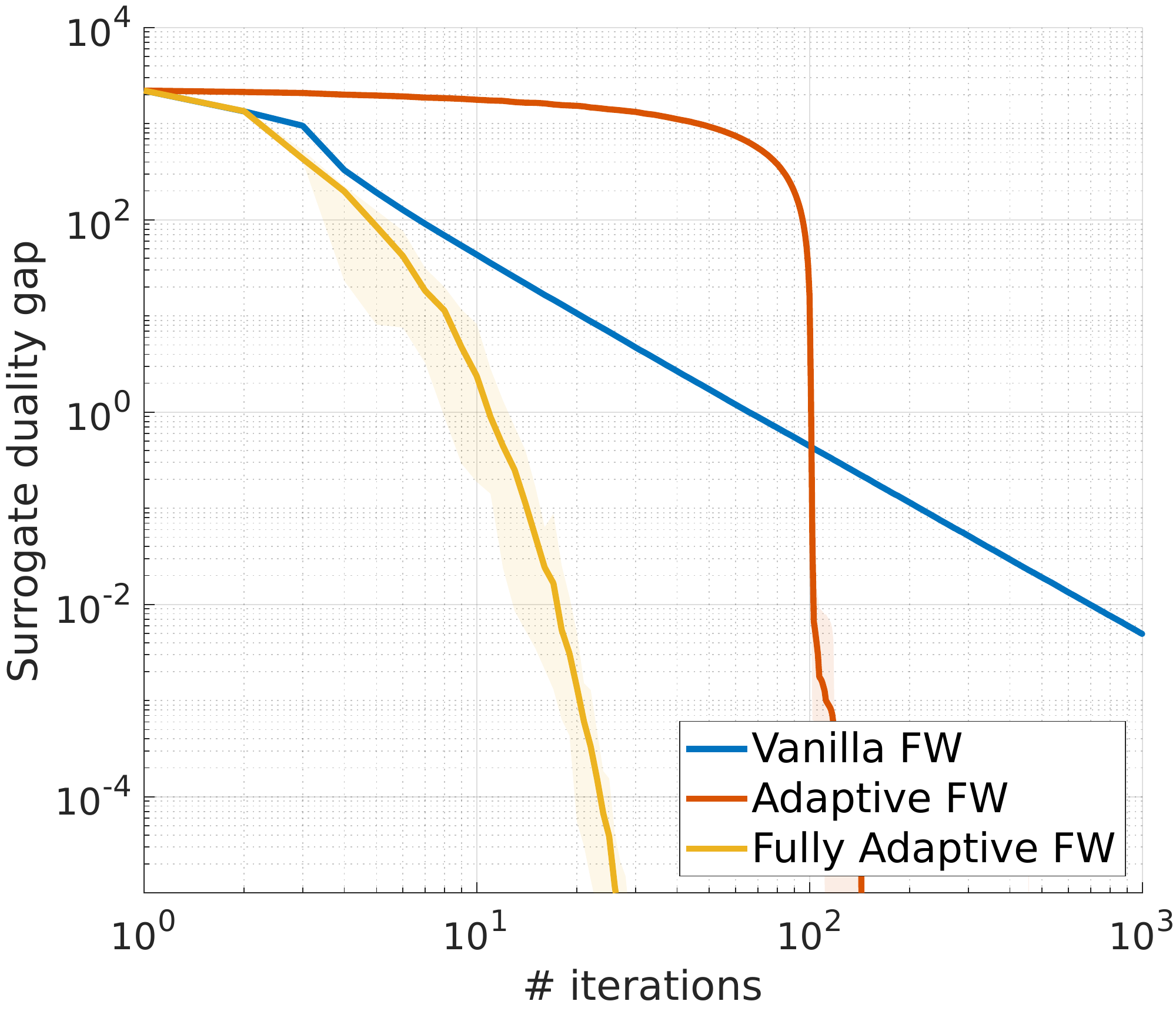}}
	\caption{Scalability properties of different methods for computing the Wasserstein MMSE estimator (shown are the means (solid lines) and the ranges (shaded areas) of the respective performance measures across $10$ simulation runs)}
	\label{fig:algorithm}
\end{figure*}

\subsection{The Value of Structural Information}
In the second experiment we study the predictive power of different MMSE estimators. 
The experiment consists of $100$ independent simulation runs. In each run, we use the same procedure as in Section~\ref{sec:scalability} to generate two random covariance matrices $\cov_x$ and $\cov_w$ of dimensions $n=m = 50$, and we set the true signal and noise distributions to $\PP_x = \mc N(0, \cov_x)$ and $\PP_w = \mc N(0, \cov_w)$, respectively. Next, we define $\covsa_x$ and $\covsa_w$ as the sample covariance matrices corresponding to $100$ independent samples from $\PP_x$ and $\PP_w$, respectively. Moreover, we set $H=I_n$. To assess the value of structural information, we compare the Wasserstein MMSE estimator proposed in this paper against the Bayesian MMSE estimator associated with the nominal signal and noise distributions and the unstructured Wasserstein MMSE estimator proposed in~\cite{ref:shafieezadeh2018wasserstein}. The latter uses a single Wasserstein ball to model the ambiguity of the joint distribution of~$x$ and~$y$, thereby ignoring the structural information that~$w=Hy-x$ is independent of~$x$. Both the structured and unstructured Wasserstein MMSE estimators collapse to the nominal Bayesian MMSE estimator when the underlying Wasserstein radii are set to zero. Recall also that the nominal Bayesian MMSE estimator is optimal in distributionally robust estimation problems whose ambiguity sets are defined via information divergences \cite{ref:levy2004robust, ref:levy2012robust, ref:zorzi2016robust, ref:zorzi2017robustness}. This robustness property makes the nominal Bayesian MMSE estimator an interesting benchmark.

We quantify the performance of a given estimator by its {\em regret}, which is defined as the difference between the estimator's average risk and the least possible average risk of {\em any} estimator under the unknown true distributions $\PP_x$ and $\PP_w$. Note that the minimum average risk is attained by the (affine) Bayesian MMSE estimator corresponding to the (normal) distributions $\PP_x$ and $\PP_w$.  Figure~\ref{fig:perf:a} shows the regret of the Wasserstein MMSE estimator with $\rho_x=\rho$ and $\rho_w=0$, the Wasserstein MMSE estimator with $\rho_x=0$ and $\rho_w=\rho$ and the unstructured Wasserstein MMSE estimator from~\cite{ref:shafieezadeh2018wasserstein} with Wasserstein radius $\rho$ for all $\rho\in[0.1,10]$. The solid lines represent the averages and the shaded areas the ranges of the regret across all $100$ simulation runs. The regret of the structured Wasserstein MMSE estimator with Wasserstein radii $\rho_x,\rho_w\in[0.1,10]$ averaged across all $100$ simulation runs is visualized by the surface plot in Figure~\ref{fig:perf:b}.

We observe that the average regret of the nominal Bayesian MMSE estimator amounts to~$16.7$ (the leftmost value of all curves in Figure~\ref{fig:perf:a}), while the best unstructured Wasserstein MMSE estimator attains a significantly lower average regret of~$13.1$ (the minimum of the blue curve in Figure~\ref{fig:perf:a}). The best structured Wasserstein MMSE estimator without noise ambiguity ($\rho_w=0$) displays a similar performance, attaining an average regret of~$13.2$ (the minimum of the red curve in Figure~\ref{fig:perf:a}), while the one without signal ambiguity ($\rho_x=0$) further reduces the average regret by more than $50\%$ to $6.0$ (the minimum of the yellow curve in Figure~\ref{fig:perf:a}). Finally, the best among {\em all} structured Wasserstein MMSE estimators, which is obtained by tuning both $\rho_x$ and $\rho_w$, attains an even lower average regret of~$2.2$ (the minimum of the surface plot in Figure~\ref{fig:perf:b}). This experiment confirms our hypothesis that structural (independence) information as well as information about distributional ambiguity can improve the predictive power of MMSE estimators.

Unlike in data-driven optimization, where the nominal distribution and the radii of the ambiguity sets can be tuned from data, we assume here that the nominal distribution and the radii of the ambiguity sets reflect the modeler's prior distributional information. Thus, they are reminiscent of prior distributions in Bayesian statistics.
The radii could be tuned using standard cross validation techniques, for example, if we had access to samples $(\hat x_i, \hat y_i)$, $i=1,\ldots, N$ drawn independently from the true joint distribution of $(x,y)$ under~$\mathbb{P}$. This conflicts, however, with our central assumption that only $y$ is observable. In this case, it is fundamentally impossible to assess the empirical performance of an estimator and to tune the radii via cross validation.

\begin{figure*}[t]
	\centering
	\subfigure[Regret of different Wasserstein MMSE estimators involving a single hyperparameter $\rho$.]{\label{fig:perf:a} \includegraphics[width=0.45\columnwidth]{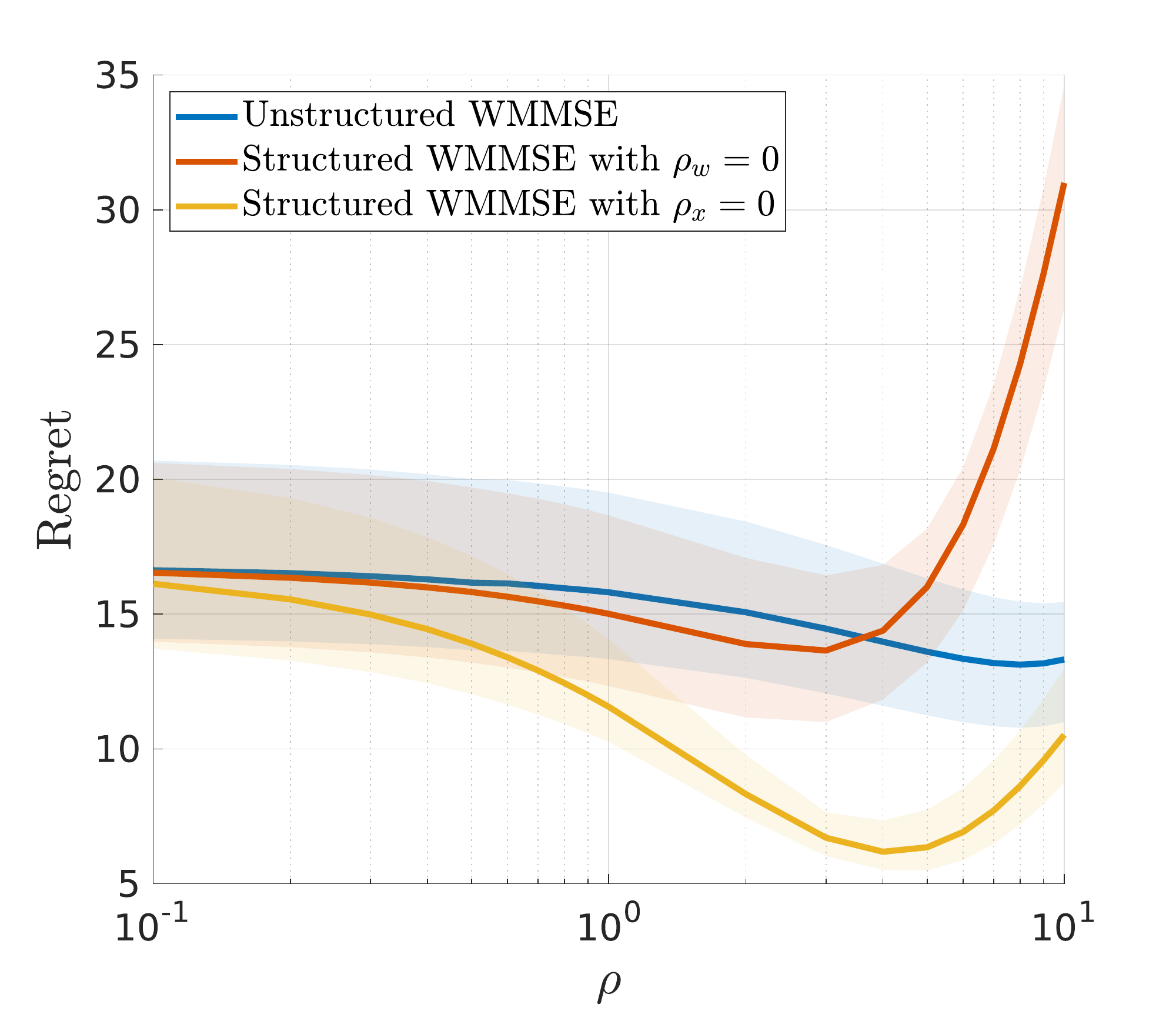}} \hspace{0pt}
	\subfigure[Regret of the Wasserstein MMSE estimator involving two hyperparameters $\rho_x$ and $\rho_w$.]{\label{fig:perf:b} \includegraphics[width=0.45\columnwidth]{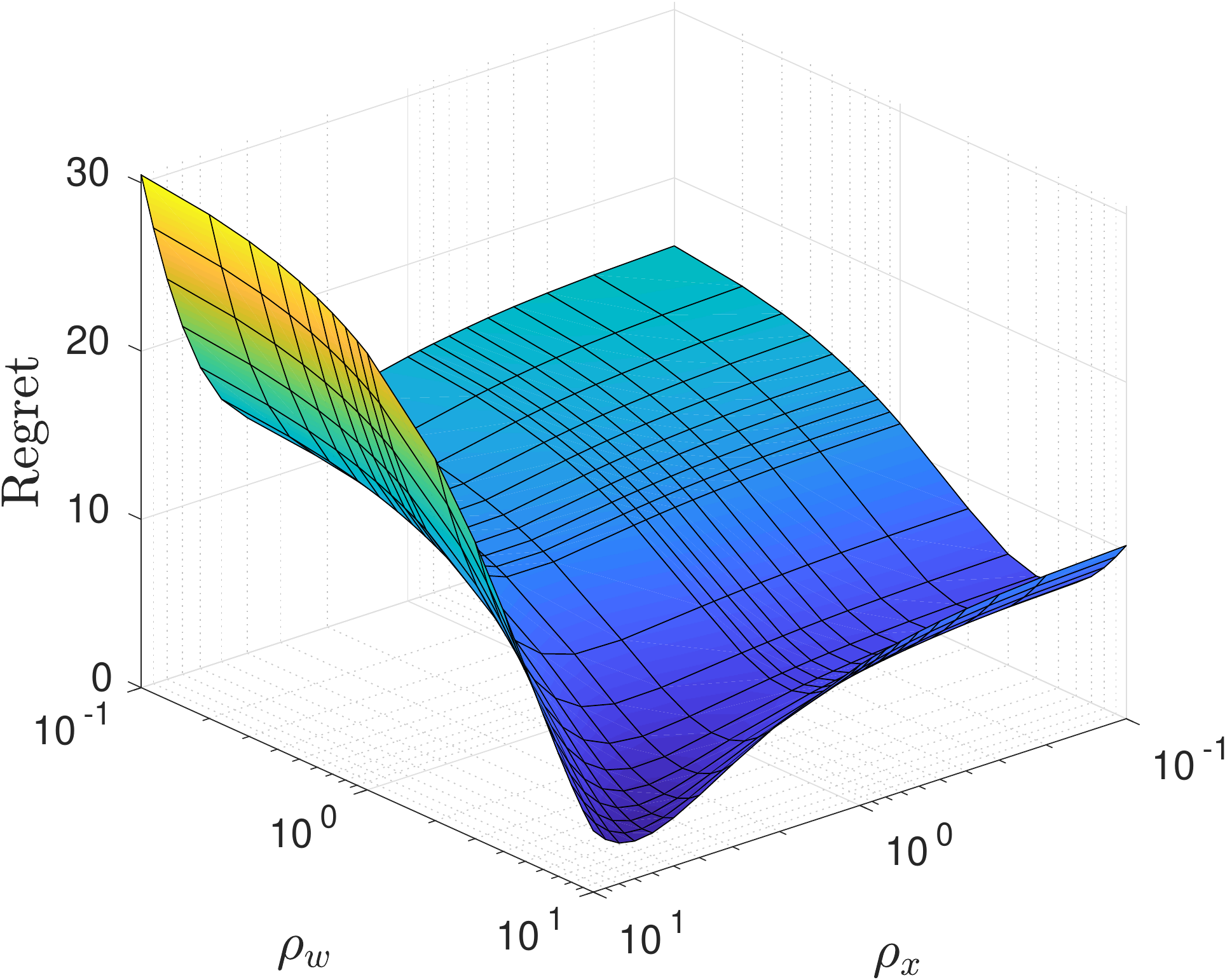}}
	\caption{Regret of different estimators averaged across 100 simulation runs.}
	\label{fig:perf}
\end{figure*}

\noindent \textbf{Acknowledgments.} We are grateful to Erick Delage for valuable comments on an earlier version of this paper. This research was supported by the Swiss National Science Foundation grant number BSCGI0\_157733.
	
\appendix
\section{Auxiliary Results}
\label{sec:auxiliary-results}
We first prove Pythagoras' theorem for the Wasserstein distance.
\begin{proof}[Proof of Lemma~\ref{lem:wasserstein-decomposition}]
	By the definition of the Wasserstein distance and by Pythagoras' theorem for the Euclidean distance we have
	\begin{align*}
		& ~\Wass(\QQ_x^1\times \QQ_w^1, \QQ_x^2\times \QQ_w^2)^2  = \Inf{\pi\in\Pi(\QQ_x^1\times \QQ_w^1, \QQ_x^2\times \QQ_w^2)}
		\int_{\R^{n+m} \times \R^{n+m}} \norm{x_1 - x_2}^2 +  \norm{w_1 - w_2}^2 \, \pi({\rm d}x_1, {\rm d} w_1, {\rm d} x_2, {\rm d}w_2 ) \\
		& \qquad \qquad \le ~  \Inf{\pi_x\in\Pi(\QQ_x^1, \QQ_x^2)} \int_{\R^{n} \times \R^{n}} \norm{x_1 - x_2}^2 \, \pi_x({\rm d}x_1,{\rm d} x_2 ) 
		+ \Inf{\pi_w\in\Pi(\QQ_w^1, \QQ_w^2)} \int_{\R^{m} \times \R^{m}} \norm{w_1 - w_2}^2 \, \pi_w({\rm d}w_1,{\rm d} w_2 ) \\[2ex]
		& \qquad \qquad = ~ \Wass(\QQ_x^1, \QQ_x^2)^2 + \Wass(\QQ_w^1, \QQ_w^2)^2,
	\end{align*}
	where the inequality follows from the restriction to factorizable transportation plans of the form $\pi=\pi_x\times \pi_w$ for some $\pi_x\in\Pi(\QQ_x^1, \QQ_x^2)$ and $\pi_w\in\Pi(\QQ_w^1, \QQ_w^2)$. To prove the converse inequality, we define $\Pi_x(\QQ_x^1, \QQ_x^2)$ as the set of all joint distributions $\pi\in \M(\R^{n+m}\times \R^{n+m})$ of $(x_1, w_1) \in\R^{n+m}$ and $(x_2, w_2) \in\R^{n+m}$ under which $x_1$ and $x_2$ have marginal distributions $\QQ_x^1$ and  $\QQ_x^2$, respectively. Similarly, we define $\Pi_w(\QQ_w^1, \QQ_w^2)$ as the set of all joint distributions $\pi\in \M(\R^{n+m}\times \R^{n+m})$ of $(x_1, w_1) \in\R^{n+m}$ and $(x_2, w_2) \in\R^{n+m}$ under which $w_1$ and $w_2$ have marginal distributions $\QQ_w^1$ and  $\QQ_w^2$, respectively. Using this notation, we find
	\begin{align*}
		\Wass(\QQ_x^1\times \QQ_w^1, \QQ_x^2\times \QQ_w^2)^2
		\ge & \Inf{\pi\in\Pi(\QQ_x^1\times \QQ_w^1, \QQ_x^2\times \QQ_w^2)}
		\int_{\R^{n+m} \times \R^{n+m}} \norm{x_1 - x_2}^2 \, \pi({\rm d}x_1, {\rm d} w_1, {\rm d} x_2, {\rm d}w_2 )\\
		& \qquad + \Inf{\pi\in\Pi(\QQ_x^1\times \QQ_w^1, \QQ_x^2\times \QQ_w^2)}
		\int_{\R^{n+m} \times \R^{n+m}} \norm{w_1 - w_2}^2 \, \pi({\rm d}x_1, {\rm d} w_1, {\rm d} x_2, {\rm d}w_2 ) \\
		\ge & \Inf{\pi\in\Pi_x(\QQ_x^1, \QQ_x^2)}
		\int_{\R^{n+m} \times \R^{n+m}} \norm{x_1 - x_2}^2 \, \pi({\rm d}x_1, {\rm d} w_1, {\rm d} x_2, {\rm d}w_2 )\\
		& \qquad + \Inf{\pi\in\Pi_w(\QQ_w^1, \QQ_w^2)}
		\int_{\R^{n+m} \times \R^{n+m}} \norm{w_1 - w_2}^2 \, \pi({\rm d}x_1, {\rm d} w_1, {\rm d} x_2, {\rm d}w_2 ) \\[2ex]
		= &~ \Wass(\QQ_x^1, \QQ_x^2)^2 + \Wass(\QQ_w^1, \QQ_w^2)^2,
	\end{align*}
	where the first inequality exploits the superadditivity of the infimum operator, while the second inequality holds because $\Pi(\QQ_x^1\times \QQ_w^1, \QQ_x^2\times \QQ_w^2)$ contains both $\Pi_x(\QQ_x^1, \QQ_x^2)$ and $\Pi_w(\QQ_w^1, \QQ_w^2)$ as subsets. The equality in the last line follows from the observation that for any $\pi \in\Pi_x(\QQ_x^1, \QQ_x^2)$ the marginal distribution $\pi_x$ of $(x_1,x_2)$ is an element of $\Pi(\QQ_x^1, \QQ_x^2)$, and for any $\pi \in\Pi_w(\QQ_x^1, \QQ_x^2)$ the marginal distribution $\pi_w$ of $(w_1,w_2)$ is an element of $\Pi(\QQ_w^1, \QQ_w^2)$. Thus, the claim follows.
\end{proof}

In order to prove Proposition~\ref{prop:Gelbrich} in the main text, we establish first general conditions for the solvability and stability of parametric minimax problems and prove that the matrix square root is H\"older continuous.
\begin{lemma}[Parametric minimax problems]
	\label{lem:minimax}
	Consider the parametric minimax problem 
	\begin{align}
	\label{J(theta)}
	J(\theta) = \inf_{u \in \set U} \sup_{v \in V(\theta)} f(u,v),
	\end{align}
	where $\set U$, $\set V$ and $\Theta$ are non-empty convex subsets of Euclidean spaces equipped with the respective subspace topologies, $f:\set U\times \set V\rightarrow \R$ and $g:\set V\times\Theta\to \R$ are continuous functions, while $V:\Theta\rightrightarrows \set V$ and $V^\circ:\Theta\rightrightarrows \set V$ are set-valued mappings defined through $V(\theta)=\{v\in \set V:g(v,\theta)\le 0\}$ and $V^\circ(\theta)=\{v\in \set V:g(v,\theta)< 0\}$, respectively. Assume that for each $\theta'\in\Theta$ there exists a compact neighborhood $\Theta'\subseteq \Theta$ of $\theta'$ such that
	\begin{enumerate}[label = $(\roman*)$, itemsep = 2mm]
	    \item \label{assumption:compactness} $V(\theta)$ is convex and bounded uniformly across all $\theta\in \Theta'$;
	    \item \label{assumption:interior} $V^\circ(\theta)$ is non-empty and coincides with the interior of $V(\theta)$ for all $\theta\in\Theta'$;
	    \item \label{assumption:compactlevelsets} there exist $v'\in\set V$ and $J'\in\R$ such that $v'\in V(\theta)$ and $J(\theta)\le J'$ for all $\theta\in\Theta'$ and such that the set $\set{U}'=\{u\in\set U:f(u,v')\leq J'\}$ is non-empty and compact.
	\end{enumerate}
	Then, the minimax problem~\eqref{J(theta)} is solvable for all $\theta\in\Theta$, and $J$ is continuous on $\Theta$.
	\end{lemma} 

\begin{proof}[Proof of Lemma~\ref{lem:minimax}]
	Define $F(u,\theta) = \sup_{v \in V(\theta)}f(u,v)$, and note that $F(u,\theta)$ is finite because it is the maximum of a continuous function over a compact feasible set. As~$V$ is locally bounded thanks to assumption~\ref{assumption:compactness} and as the graph of~$V$ is closed thanks to the continuity of~$g$, the closed graph theorem~\cite[Proposition~1.4.8]{ref:aubinfrankowska1990} ensures that the set-valued mapping~$V$ is upper semicontinuous in the sense of~\cite[Chapter~VI]{ref:berge1963topological}. By~\cite[Theorem~2, p.~116]{ref:berge1963topological}, the optimal value function~$F$ is thus upper semicontinuous on~$\set U\times\Theta$. Next, note that $F(u,\theta)= \sup_{v \in V^\circ(\theta)}f(u,v)$ because~$f$ is continuous and because $V^\circ(\theta)$ coincides with the non-empty interior of~$V(\theta)$ thanks to assumption~\ref{assumption:interior}. As~$V^\circ$ is convex-valued thanks to assumption~\ref{assumption:compactness} and as the graph of $V$ is open thanks to the continuity of~$g$, the open graph theorem~\cite[Proposition~2]{ref:zhou1995} implies that the set-valued mapping~$V$ is lower semicontinuous in the sense of~\cite[Chapter~VI]{ref:berge1963topological}. By~\cite[Theorem~1, p.~115]{ref:berge1963topological}, the function~$F$ is thus lower semicontinuous on~$\set U\times\Theta$. In summary, $F$ is therefore continuous on~$\set U\times\Theta$.
	
	To prove that $J$ is lower semicontinuous, we select an arbitrary $\theta'\in \Theta$ and a compact neighborhood $\Theta'\subseteq \Theta$ of $\theta'$ that satisfies assumption~\ref{assumption:compactlevelsets}. For any $\theta\in\Theta'$, we may then restrict the feasible set of the outer mininmization problem in~\eqref{J(theta)} to~$\set{U}'$ without affecting $J(\theta)$. Indeed, any $u\notin \set{U}'$ and $\theta\in\Theta'$ satisfies
	\[
	    F(u,\theta)=\sup_{v \in V(\theta)} f(u,v) \ge f(u,v') > J'\ge J(\theta)\,,
	\]
	where the three inequalities hold because $v'\in V(\theta)$ for all~$\theta\in\Theta'$, $u\notin\set{U}'$ and $J(\theta)\le J'$ for all~$\theta\in\Theta'$, respectively. Thus, $u\notin \set{U}'$ is strictly suboptimal, and we have $J(\theta)=\inf_{u\in \set{U}'} F(u,\theta)$ for all~$\theta\in\Theta'$. As $\set{U}'$ is compact and $F(u,\theta)$ is continuous, the restricted minmimax problem with $\set{U}'$ replacing $\set U$ is solvable, and any minimizer also solves~\eqref{J(theta)}. Moreover, the constant feasible set mapping that assigns each $\theta\in\Theta'$ the same set $\set{U}'$ is continuous in the sense of~\cite[Chapter~VI]{ref:berge1963topological}. By Berge's maximum theorem~\cite[Theorem~1, p.~115]{ref:berge1963topological}, the function~$J$ is thus continuous on~$\Theta'$ and, as $\theta'\in\Theta$ was chosen arbitrarily, on all of~$\Theta$.
\end{proof}

\begin{lemma}[H\"older continuity of the matrix square root]
    \label{lemma:hoelder}
    The square root of a positive semidefinite matrix is H\"older-continuous with exponent $1/2$. More precisely, we have
    \[ \|\sqrt{A_1} - \sqrt{A_2}\| \le 2\sqrt{d}\,\|A_1 - A_2\|^{1/2} \quad \forall A_1, A_2 \in \PSD^d \,. \]
\end{lemma}
\begin{proof}[Proof of Lemma~\ref{lemma:hoelder}]
    The proof exploits the following two facts. 
    \begin{enumerate}[label = $(\roman*)$, itemsep = 2mm]
        \item \label{root-hoelder-i} By \cite[Lemma 2.2]{ref:schmitt1992perturbation}, we have 
        \[\|\sqrt{A_1} - \sqrt{A_2}\| \le \frac{1}{\lambda_{\min}(A_1) + \lambda_{\min}(A_2)} \|A_1 - A_2 \| \quad \forall A_1, A_2 \in \PD^d \,.\]
        
        \item \label{root-hoelder-ii} One can show that 
        \[\|\sqrt{A + \eps I_d} - \sqrt{A}\| \le d \sqrt{\eps}\quad \forall A \in \PSD^d, ~\eps\ge 0\,.
        \]
    \end{enumerate}
    To prove assertion~\ref{root-hoelder-ii}, denote by $\lambda_i\ge 0$, $i=1,\ldots, d$, the eigenvalues of $A\in\PSD^d$ and by $a_i\in\mbb R^d$, $i=1,\ldots,d$, the corresponding orthonormal eigenvectors, which implies that $A = \sum_{i=1}^d \lambda_i a_i a_i^\top$. Thus, we find 
    \begin{align*}
    \|\sqrt{A + \eps I_d} - \sqrt{A}\| & = \Big\|\sum_{i =1}^d  \big(\sqrt{(\lambda_i + \eps)} - \sqrt{\lambda_i}\big)a_i a_i^\top  \Big\| = \Big\|\sum_{i=1}^d  \eps\big(\sqrt{(\lambda_i + \eps)} + \sqrt{\lambda_i}\big)^{-1}a_i a_i^\top  \Big\| \\    
    & \le \sum_{i=1}^d \frac{\eps}{\sqrt{(\lambda_i + \eps)} + \sqrt{\lambda_i}}\|a_i a_i^\top\| \le d \sqrt{\eps}\,,
    \end{align*}
    where the first inequality follows from the triangle inequality, and the second inequality holds because the denominator of the $i$-th fraction grows with $\lambda_i$ and is therefore minimal for $\lambda_i = 0$. 
    
    To prove the lemma, select now arbitrary $A_1, A_2\in\PSD^d$ and $\eps>0$. The triangle inequality implies that
    \begin{align*}
        \|\sqrt{A_1} - \sqrt{A_2}\| & \le \|\sqrt{A_1} - \sqrt{A_1 + \eps I_d}\| + \|\sqrt{A_1 + \eps I_d} - \sqrt{A_2 + \eps I_d}\| + \|\sqrt{A_2 + \eps I_d} - \sqrt{A_2}\| \\
        & \le d \sqrt{\eps} + \frac{1}{\lambda_{\min}(\sqrt{A_1 + \eps I_d}) + \lambda_{\min}(\sqrt{A_2+ \eps I_d}) } \|A_1-A_2\| + d \sqrt{\eps} \le 2d\sqrt{\eps} + \frac{\|A_1-A_2\|}{2 \sqrt{\eps}} \,,
    \end{align*}
    where the second exploits the assertions~\ref{root-hoelder-i} and~\ref{root-hoelder-ii}. The claim now follows by setting $\eps = {\|A_1-A_2\| / 4d}$, which actually minimizes the right hand side of the last expression. 
\end{proof}

Armed with Lemmas~\ref{lem:minimax} and~\ref{lemma:hoelder}, we are now ready to prove Proposition~\ref{prop:Gelbrich}.

\begin{proof}[Proof of Proposition~\ref{prop:Gelbrich}]
	The Gelbrich MMSE estimation problem~\eqref{eq:dro:approx} upper bounds the Wasserstein MMSE estimation problem~\eqref{eq:dro} because $\Ac \subseteq \F$ and $\G(\Pnom)\supseteq \mbb B(\Pnom)$; see Corollary~\ref{cor:wasserstein-in-gelbrich}. Thus, assertion~\ref{prop:Gelbrich:conservative} follows. Recall now that any $\psi\in\Ac$ can be represented as $\psi(y)=\slope y+\intercept$ for some $\slope \in\R^{n\times m}$ and $\intercept \in\R^n$. Moreover, for any distribution $\QQ=\QQ_x\times \QQ_w\in  \G(\Pnom)$, denote by $\mu_x$ and $\mu_w$ the mean vectors and by $\cov_x$ and $\cov_w$ the covariance matrices of $\QQ_x$ and $\QQ_w$, respectively. Hence, the objective function and the constraints of~\eqref{eq:dro:approx} depend on $\psi$ and $\QQ$ only through $u=(A,b)$, which ranges over $\set{U}=\R^{n\times m}\times \R^n$, and $v=(\mu_x,\mu_w,\cov_x,\cov_w)$, which ranges over $\mbb V= \R^n\times\R^m\times\PSD^n\times\PSD^m$. Indeed, the average risk of $\psi$ under $\QQ$ satisfies	 
	\begin{align*}
		\avrisk(\psi,\QQ)&= \EE_{\QQ}\left[ \|x-\slope(Hx+w)-\intercept \|^2 \right] = \EE_{\QQ}\left[ \|(I_n- \slope H)x-\slope w-\intercept\|^2 \right]\\
		& =  \inner{(I_n- \slope H)^\top (I_n- \slope H)}{\cov_x + \mu_x \mu_x^\top} + \inner{\slope^\top \slope}{\cov_w + \mu_w \mu_w^\top} + \intercept^\top \intercept \\
		& \qquad - 2 \mu_x^\top (I_n- \slope H)^\top \slope \mu_w - 2 \intercept^\top ( (I_n- \slope H) \mu_x - \slope \mu_w ) = f(u,v).
	\end{align*}
	Similarly, the constraint $\QQ\in  \G(\Pnom)$ can be reformulated as $g(v,\theta)\le 0$, where $\theta = (\msa_x, \msa_w, \covsa_x, \covsa_w)$ is a shorthand for the problem's input parameters ranging over the set $\Theta= \R^n\times\R^m\times\PSD^n\times\PSD^m$, and where
	\[
	    g(v,\theta)=\max\left\{ \Gelbrich \big( (\m_x, \cov_x), (\msa_x, \covsa_x) \big)^2 - \rho_x^2,\, \Gelbrich \big( (\m_w, \cov_w), (\msa_w, \covsa_w) \big)^2 - \rho_w^2 \right\}.
	\]
	Thus, the Gelbrich MMSE estimation problem~\eqref{eq:dro:approx} can be re-expressed more concisely as 
	\begin{align} 
		\label{eq:dro:approx-reformulation}
		J(\theta)=\Inf{u\in\set{U}} \Sup{v\in V(\theta)} f(u,v)\,,
	\end{align}
	where $V(\theta)=\{v\in\set{V}:g(v,\theta)\leq 0\}$ for all $\theta\in\Theta$. Assume now that $\rho_x>0$ and $\rho_w>0$. In the remainder we will prove that the optimal value $J(\theta)$ of the minimax problem~\eqref{eq:dro:approx-reformulation} is attained and continuous in $\theta$ by showing that all assumptions of Lemma~\ref{lem:minimax} in the appendix are satisfied.
	
	To this end, note first that~$f$ is continuous by construction and that~$g$ is convex and continuous thanks to Proposition~\ref{prop:gelbrich-convexity}. Next, select an arbitrary reference point $\theta' = (\msa_x', \msa_w', \covsa_x', \covsa_w') \in \Theta$, and define
	\begin{equation*}
	\Theta' = \left\{ (\msa_x, \msa_w, \covsa_x, \covsa_w) \in \Theta: \begin{array}{l@{\,}l}
	       \Gelbrich \big( (\msa_x, \covsa_x), (\msa'_x, \covsa'_x) \big) &\leq \frac{\rho_x}{2} \\[1ex]
	       \Gelbrich \big( (\msa_w, \covsa_w), (\msa'_w, \covsa'_w) \big) &\leq \frac{\rho_w}{2}
	\end{array} 
	\right\}.
	\end{equation*}
	Note that $\Theta'$ is a neighborhood of $\theta'$ because $\rho_x$ and $\rho_w$ are striclty positive and because Proposition~\ref{prop:gelbrich-convexity} ensures that the (squared) Gelbrich distance is continuous. Moreover, $\Theta'$ is compact due to Lemma~\ref{lemma:compact:FS}.
	
	In order to verify assumption~\ref{assumption:compactness} of Lemma~\ref{lem:minimax}, we note first that $V(\theta)$ is a convex set for every $\theta\in\Theta$ because~$g$ is a convex function. Moreover, we have
	\begin{align*}
	    V(\theta) &= \left\{ (\mu_x, \mu_w, \cov_x, \cov_w) \in \set{V}:
	    \begin{array}{l@{\,}l}
	         \Gelbrich \big( (\m_x, \cov_x),\, (\msa_x, \covsa_x) \big) &\leq \rho_x \\[1ex]
	         \Gelbrich \big( (\m_w, \cov_w),\, (\msa_w, \covsa_w) \big) &\leq \rho_w
	    \end{array}
	    \right\}\\[1ex]
	    & \subseteq \left\{ (\mu_x, \mu_w, \cov_x, \cov_w) \in \set{V}:
	    \begin{array}{l@{\,}l}
	         \Gelbrich \big( (\m_x, \cov_x),\, (\msa'_x, \covsa'_x) \big) &\leq \frac{3\rho_x}{2} \\[1ex]
	         \Gelbrich \big( (\m_w, \cov_w),\, (\msa'_w, \covsa'_w) \big) & \leq \frac{3\rho_w}{2}
	    \end{array}
	    \right\}=V'\quad \forall \theta\in\Theta',
	\end{align*} 
	where the first equality holds due to the definition of $V(\theta)$, the inclusion holds due to the definition of $\Theta'$ and because the Gelbrich distance satisfies the triangle inequality, and the last equality defines the set~$V'$, which is compact thanks to Lemma~\ref{lemma:compact:FS}. This reasoning shows that $V(\theta)$ is uniformly bounded on $\Theta'$. 
	
	In order to verify assumption~\ref{assumption:interior} of Lemma~\ref{lem:minimax}, define
	\begin{equation*}
	    V^\circ(\theta) =\left\{v\in\set{V}:g(v,\theta)< 0\right\}= \left\{ (\mu_x, \mu_w, \cov_x, \cov_w) \in \set{V}:
	    \begin{array}{l@{\,}l}
	         \Gelbrich \big( (\m_x, \cov_x),\, (\msa_x, \covsa_x) \big) &< \rho_x \\[1ex]
	         \Gelbrich \big( (\m_w, \cov_w),\, (\msa_w, \covsa_w) \big) &< \rho_w
	    \end{array}
	    \right\}
	\end{equation*}
	for any $\theta\in\Theta$. As the Gelbrich distance satisfies the identity of indiscernibles and as $\rho_x$ and $\rho_w$ are strictly positive, we have $\theta\in V^\circ(\theta)$, which implies that $V^\circ(\theta)$ is non-empty. It remains to be shown that $V^\circ(\theta)$
    coincides with the interior of $V(\theta)$. As the Gelbrich distance is continuous by virtue of Proposition~\ref{prop:gelbrich-convexity}, it is clear that $V^\circ(\theta)$ is an open subset of $V(\theta)$ and thus contained in $\text{int}(V(\theta))$. To prove the converse inclusion, assume for the sake of argument that $V^\circ(\theta)$ is a strict subset of $\text{int}(V(\theta))$. Thus, there must exist an open set $O\subseteq \text{int}(V(\theta))\backslash V^\circ(\theta)$. Otherwise, each point in $\text{int}(V(\theta))\backslash V^\circ(\theta)$ would belong to the boundary of $\text{int}(V(\theta))$, which is impossible because $\text{int}(V(\theta))$ is open. As $O\subseteq V(\theta)\backslash V^\circ(\theta)$, it is clear that at least one of the equalities $\Gelbrich( (\m_x, \cov_x),\, (\msa_x, \covsa_x)) = \rho_x$ or $\Gelbrich( (\m_w, \cov_w),\, (\msa_w, \covsa_w)) = \rho_w$ is satisfied at any point in~$O$. In fact, as the Gelbrich distance is continuous, one of these equalities holds throughout an open set $O'\subseteq O$. Without loss of generality we may thus assume that $\Gelbrich( (\m_x, \cov_x),\, (\msa_x, \covsa_x))^2 = \rho_x^2$ on an open set $O'\subseteq O$, which implies that any point in $O'$ is a local minimizer of the squared Gelbrich distance. As the squared Gelbrich distance is convex due to Proposition~\ref{prop:gelbrich-convexity}, this means that all points in $O'$ are in fact global minimizers. This is not possible, however, because the (squared) Gelbrich distance adopts its global minimum only at points where $\mu_x=\msa_x$ and $\cov_x=\covsa_x$. No such point belongs to $O'$ because $O'\cap V^\circ(\theta)=\emptyset$. Thus, we have $\text{int}(V(\theta))= V^\circ(\theta)$.
        
    In order to verify assumption~\ref{assumption:compactlevelsets} of Lemma~\ref{lem:minimax}, select a point $v'=(\mu_x', \mu_w', \cov_x', \cov_w')\in \Theta'$ with $\cov_w'\succ 0$. Such a point exists because $\rho_w>0$ and because the (squared) Gelbrich distance is continuous by virtue of Proposition~\ref{prop:gelbrich-convexity}, which implies that 
    \[
        (\mu_x', \mu_w', \cov_x', \cov_w')=(\msa_x, \msa_w, \covsa_x, \covsa_w+\lambda\cdot I_m) \in \Theta'
    \]
    for all sufficiently small~$\lambda>0$. The triangle inequality for the Gelbrich distance then ensures that $v'\in V(\theta)$ for all $\theta\in\Theta'$. Next, set $J'=\sup_{v\in V'}f(0,v)$, which is finite because $V'$ is compact and $f$ is continuous, and note that
    \[
	    J(\theta)=\inf_{u\in\set{U}} \sup_{v\in V(\theta)} f(u,v)\leq \sup_{v\in V'}f(0,v)=J'\quad \forall \theta\in\Theta',
	\]
    where the inequality holds because $V(\theta)\subseteq V'$ whenever $\theta\in\Theta'$. Finally, define $\set{U}'=\{u \in\set{U}: f(u,v')\leq J'\}$, which is non-empty. Indeed, $\set{U}'$ contains at least the point $u=0$ because $v'\in\Theta'\subseteq V'$. As $\cov_w'\succ 0$, it is easy to verify that $f(u,v')$ is strictly convex and quadratic in~$u$, which implies that $\set{U}'$ is a compact ellipsoid.
    
    In summary, problem~\eqref{eq:dro:approx-reformulation}, which is equivalent to the Gelbrich MMSE estimation problem~\eqref{eq:dro:approx}, satisfies all assumptions of Lemma~\ref{lem:minimax}. Therefore, its optimal value $J(\theta)$ is attained and continuous in $\theta$.
\end{proof}

In order to derive a tractable reformulation for the Gelbrich MMSE estimation problem studied in Section~\ref{sect:approx}, we need to be able to solve nonlinear SDPs of the form
\begin{align} 
    \label{lin_obj:Lag}
	J\opt & = \, \sup_{\cov \succeq 0} ~ \inner{D}{\cov} - \dualvar \Tr{\cov - 2 \big( \covsa^\half \cov \covsa^\half \big)^\half}
\end{align}
parameterized by $D \in \mbb S^d$, $\covsa \in \PSD^d$ and $\dualvar \in\mbb R_+$. It is known that, under certain regularity conditions, problem~\eqref{lin_obj:Lag} admits a unique optimal solution that is available in closed form~\cite{ref:nguyen2018distributionally}. In the following we review the construction of this optimal solution under slightly more general conditions.

\begin{proposition}[Closed form solution of \eqref{lin_obj:Lag}] \label{prop:lin_obj:Lag}
	For any $D \in \mbb S^d$, $\covsa \in \PSD^d$ and $\dualvar \in \R_+$ the optimal value of the nonlinear SDP~\eqref{lin_obj:Lag} is given by
	\[
	J\opt = 
		\begin{cases}
	        \dualvar^2 \inner{(\dualvar I_d - D)^{-1} }{\covsa} & \text{if }\dualvar > \lambda_{\max}(D), \\[1ex]
	        \displaystyle \liminf_{\bar \dualvar \downarrow \dualvar} \bar\dualvar^2 \inner{(\bar \dualvar I_d - D)^{-1} }{\covsa} &\text{if } \dualvar = \lambda_{\max}(D), \\
	+\infty &\text{if }\dualvar < \lambda_{\max}(D).
	\end{cases}
	\]
	Moreover, problem~\eqref{lin_obj:Lag} is solved by $\cov\opt = \dualvar^2 (\dualvar I_d - \direc)^{-1} \covsa (\dualvar I_d - \direc)^{-1}$ whenever $\dualvar > \lambda_{\max}(D)$. This solution is unique if $\covsa\succ 0$.
\end{proposition}

\begin{proof}[Proof of Proposition~\ref{prop:lin_obj:Lag}]
    Assume first that $\dualvar > \lambda_{\max}(D)$. Moreover, in order to simplify the proof, assume temporarily that $\covsa \succ 0$. By applying the nonlinear variable transformation $B \leftarrow ( \covsa^\half \, \cov \, \covsa^\half )^\half$, which implies that $\cov = \covsa^{-\half} \, B^2 \, \covsa^{-\half}$, we can reformulate problem~\eqref{lin_obj:Lag} as
	\begin{align*}
		J\opt = &\Sup{B \succeq 0} ~ \inner{D}{ \covsa^{-\half} \, B^2 \, \covsa^{-\half} } - \dualvar \Tr{\covsa^{-\half} \, B^2 \, \covsa^{-\half} - 2 B} \\
		=&\Sup{B \succeq 0} ~ \inner{B^2}{\covsa^{-\half} \, (D - \dualvar I_d) \, \covsa^{-\half}} + 2 \, \dualvar \, \inner{B}{I_d}, 
	\end{align*}
	where the second equality exploits the cyclicity of the trace operator. Introducing the auxiliary parameter $\Delta = \covsa^{-\half} (D - \dualvar I_d) \covsa^{-\half}$, we can then rewrite the last maximization problem over~$B$ more concisely as
	\begin{align}
		\label{eq:sup-problem}
		J\opt = \Sup{B \succeq 0} ~ \inner{B^2}{\Delta} + 2 \, \dualvar \, \inner{B}{I_d}.
	\end{align}
	Note that~\eqref{eq:sup-problem} represents a convex maximization problem because $\dualvar > \lambda_{\max}(D)$ and $\covsa\succ 0$, which imply that~$\Delta \prec 0$. Ignoring the positive semidefiniteness constraint on~$B$, the objective function of~\eqref{eq:sup-problem} is uniquely minimized by the solution $B\opt = -\dualvar \Delta^{-1}$ of the first-order optimality condition $B \Delta + \Delta B + 2 \dualvar I_d = 0$. Uniqueness follows from \cite[Theorem~12.5]{ref:hespanha2009linear}. As it is strictly positive definite, $B\opt$ thus uniquely solves~\eqref{eq:sup-problem}, which in turn implies that $\cov\opt= \covsa^{-\half} (B\opt)^2 \covsa^{-\half}= \dualvar^2 (\dualvar I_d - \direc)^{-1} \covsa (\dualvar I_d - \direc)^{-1}$ uniquely solves~\eqref{lin_obj:Lag}. Substituting $\cov\opt$ back into the objective function of~\eqref{lin_obj:Lag} further shows that $J\opt =\dualvar^2 \inner{(\dualvar I_d - D)^{-1}}{\covsa}$.
	
	Next, we will argue that the analytical formula for $J\opt$ in the regime $\dualvar > \lambda_{\max}(D)$ remains valid even when~$\covsa$ is rank-deficient. To see this, define
	\[
	J\opt(\covsa)= \dualvar^2 \inner{(\dualvar I_d - D)^{-1}}{\covsa} \quad \text{and} \quad \cov\opt(\covsa) =  \dualvar^2 (\dualvar I_d - \direc)^{-1} \covsa (\dualvar I_d - \direc)^{-1}
    \]
    as explicit continuous functions of the parameter $\covsa\in\PSD^d$. Similarly, define the function
    \[
    F(\cov, \covsa) = \inner{D}{\cov} - \dualvar \Tr{\cov - 2 \big( \covsa^\half \cov \covsa^\half \big)^\half},
    \]
    which is jointly continuous in $\cov\in\PSD^d$ and $\covsa\in\PSD^d$. We then have
    \[
	J\opt(\covsa)=\liminf_{\eps\downarrow 0} J\opt(\covsa+\eps I_d)=\liminf_{\eps\downarrow 0} \sup_{\cov \succeq 0} F(\cov,\covsa+\eps I_d) \ge \sup_{\cov \succeq 0} F(\cov,\covsa) \ge F(\cov\opt(\covsa),\covsa)=J\opt(\covsa),
	\]
    where the first equality follows from the continuity of $J\opt(\covsa)$, while the second equality holds because $\covsa+\eps I_d\succ 0$ for every $\eps>0$ and because $J\opt(\covsa')=\sup_{\cov\succ 0}F(\cov,\covsa')$ for every $\covsa'\succ 0$, which was established in the first part of the proof. The first inequality exploits the fact that a pointwise supremum of continuous functions is is lower semicontinuous, and the second inequality holds because $\cov\opt(\covsa+\eps I_d)\succ 0$ for every $\eps>0$. Finally, the last equality follows from elementary algebra. The above arguments imply that $J\opt(\covsa)$ and $\cov\opt(\covsa)$ represent the optimal value and an optimal solution of~\eqref{lin_obj:Lag}, respectively, even if $\covsa\in\PSD^d$ is rank-deficient.

	Assume next that $\dualvar < \lambda_{\max}(D)$, and denote by $\overline v\in\mbb R^d$ a normalized eigenvector of~$D$ corresponding to the eigenvalue~$\lambda_{\max}(D)$. By optimizing only over matrices of the form $\cov=t \,\overline v\,\overline v^\top$ for some $t\ge 0$, we find
    \begin{align*}
        J\opt &\ge\, \sup_{t\ge 0} ~ t\,\inner{D - \dualvar I_d}{\overline v\,\overline v^\top} + 2\sqrt{t}\,\dualvar\Tr{\big( \covsa^\half \overline v\,\overline v^\top\covsa^\half \big)^\half} \\
        & =\, \sup_{t\ge 0} ~t\, (\lambda_{\max}(D) - \dualvar) + 2\sqrt{t}\,\dualvar\Tr{\big( \covsa^\half \overline v\,\overline v^\top\covsa^\half \big)^\half}=\infty.
    \end{align*}
    
    Assume finally that $\dualvar = \lambda_{\max}(D)$. To investigate this limiting case, note that the objective function of~\eqref{lin_obj:Lag} is linear in~$\dualvar$, which implies that the optimal value of~\eqref{lin_obj:Lag} is convex and lower semicontinuous in~$\dualvar$. Given the results for $\dualvar\neq \lambda_{\rm max}(D)$, it is thus clear that for $\dualvar = \lambda_{\rm max}(D)$ the optimal value of~\eqref{lin_obj:Lag} must be given by $J\opt=\liminf_{\bar \dualvar \downarrow \dualvar} \bar\dualvar^2 \inner{(\bar \dualvar I_d - D)^{-1} }{\covsa}$. This observation completes the proof.
\end{proof}

In order to derive search directions for the Frank-Wolfe algorithm developed in Section~\ref{sect:algorithm}, we need to be able to solve constrained nonlinear SDPs of the form
\begin{align} 
	\label{lin_obj:const}
	\begin{array}{cl}
	\Sup{\cov \succeq 0} & \inner{D}{\cov}\\[-1ex]
	\st & \Tr{\cov + \covsa - 2 \big(\covsa^\half \cov \covsa^\half \big)^\half} \leq \rho^2
	\end{array}
\end{align}
parameterized by $D \in \mbb S^d$, $\covsa \in \PSD^d$ and $\rho \in\mbb R_+$. It is known that problem~\eqref{lin_obj:const} admits a unique optimal solution that is available in quasi-closed form~\cite{ref:nguyen2018distributionally}. Below we review the construction of this optimal solution under more general conditions and uncover several previously unknown properties of this solution.

\begin{proposition}[Quasi-closed form solution of~\eqref{lin_obj:const}]
	\label{prop:quadratic}
	The following statements hold. 
	\begin{enumerate}[label = $(\roman*)$, itemsep = 2mm]
	    \item \label{prop:quad:dual}
	    If $D \in \mbb S^d$, $\covsa \in \PSD^d$ and $\rho \in\mbb R_+$, then problem~\eqref{lin_obj:const} is solvable, and its maximum matches that of the univariate convex minimization problem
	    \begin{align}
	        \label{lin_obj:const_refor}
	        \inf\limits_{\substack{ \dualvar \ge 0 \\\dualvar > \lambda_{\max}(D) }} \dualvar \Big(\rho^2 + \inner{\dualvar(\dualvar I_d - D)^{-1} - I_d}{\covsa}\Big).
	    \end{align}
	    
	    \item \label{prop:quad:cov_opt}
	    If $D \neq 0$, $\covsa \succ 0$ and $\rho > 0$, then problem~\eqref{lin_obj:const_refor} has a unique minimizer $\dualvar\opt\in(\lambda_{\max}(D),\infty)$, and problem~\eqref{lin_obj:const} is solved by $\cov\opt = {\dualvar\opt}^2(\dualvar\opt I_d - D)^{-1} \covsa (\dualvar\opt I_d - D)^{-1}$.

	    \item \label{prop:quad:cov_min}
	    If $D \succeq 0$, $D \neq 0$, $\covsa \succ 0$ and $\rho > 0$, then $\dualvar\opt$ is the unique solution of the algebraic equation
	    \begin{align}
	    \label{eq:quadratic:algebraic-equation}
	        \rho^2 - \inner{\covsa}{\big( I_d - \dualvar\opt (\dualvar\opt I_d - D)^{-1} \big)^2} = 0
	    \end{align}
	    in the interval $(\lambda_{\max}(D),\infty)$, the matrix $\cov\opt$ defined in assertion~\ref{prop:quad:cov_min} is the unique maximizer of~\eqref{lin_obj:const}, the Gelbrich distance constraint in~\eqref{lin_obj:const} is binding at $\cov\opt$, and $\cov\opt \succ \lambda_{\min}(\covsa) I_d$. Moreover, we have
	    \[
	    	\underline\dualvar= \lambda_1\left( 1+\sqrt{v_1^\top\covsa v_1}/\rho \right)\leq \dualvar\opt\leq \lambda_1\left( 1+\sqrt{\Tr{\covsa}}/\rho\right) = \overline\dualvar,
	    \]
	    where $\lambda_1$ is the largest eigenvalue of $D$, and $v_1$ an eigenvector corresponding to $\lambda_1$.
	   
	\end{enumerate}
\end{proposition}

\begin{proof}[Proof of Proposition~\ref{prop:quadratic}]
    As for assertion~\ref{prop:quad:dual}, note that the Lagrangian dual of \eqref{lin_obj:const} can be represented as 
    \begin{align}
    \label{eq:quad:lagrangian}
        \inf_{\substack{ \dualvar \ge 0}} ~\sup_{\cov \succeq 0} ~ \inner{D}{\Sigma} - \dualvar\Tr{\cov + \covsa - 2 \big( \covsa^\half \cov \covsa^\half \big)^\half} + \dualvar \rho^2 .
    \end{align}
    Strong duality as well as primal solvability follow from \cite[Proposition~5.5.4]{ref:bertsekas2009convex}, which applies because the primal problem~\eqref{lin_obj:const} has a continuous objective function and---by virtue of Lemma~\ref{lemma:compact:FS} below---a non-empty, compact and convex feasible set. The postulated reformulation~\eqref{lin_obj:const_refor} then follows immediately from replacing the supremum of the inner maximization problem in~\eqref{eq:quad:lagrangian} with the analytical formula derived in Proposition~\ref{prop:lin_obj:Lag}. We emphasize that for $\dualvar=\lambda_{\rm max}(D)$, depending on the problem data, the inner supremum in~\eqref{eq:quad:lagrangian} may evaluate to any non-negative real number or to~$+\infty$. In order to avoid cumbersome case distinctions, we thus exclude the point $\dualvar=\lambda_{\rm max}(D)$ from the feasible set of~\eqref{lin_obj:const_refor} without affecting the problem's infimum.
	As for assertion~\ref{prop:quad:cov_opt}, note that $\Sigma = \covsa$ represents a Slater point for the primal problem~\eqref{lin_obj:const} because $\rho > 0$. Thus, the dual problem~\eqref{eq:quad:lagrangian} is solvable by~\cite[Proposition~5.3.1]{ref:bertsekas2009convex}. To prove that~\eqref{lin_obj:const_refor} is also solvable, it remains to be shown that~\eqref{eq:quad:lagrangian} does not attain its maximum at the boundary point $\dualvar=\lambda_{\max}(D)$, which has been excluded from the feasible set of~\eqref{lin_obj:const_refor}. This is the case, however, because of the assumption that $\covsa \succ 0$ and $D \neq 0$, which ensures that the objective function value of $\dualvar=\lambda_{\max}(D)$ in~\eqref{eq:quad:lagrangian} amounts to~$+\infty$. We may thus conclude that~\eqref{lin_obj:const_refor} admits a minimizer~$\dualvar\opt\in(\lambda_{\max}(D),\infty)$. This mimnimizer is unique because the objective function of~\eqref{lin_obj:const_refor} is strictly convex when~$\covsa\succ 0$. Finally, the Karush-Kuhn-Tucker optimality conditions  \cite[Proposition~5.3.2]{ref:bertsekas2009convex} imply that any solution of the primal problem~\eqref{lin_obj:const} also solves the inner maximization problem in~\eqref{eq:quad:lagrangian} at $\dualvar = \dualvar\opt$. The formula for $\cov\opt$ thus follows from Proposition~\ref{prop:lin_obj:Lag}.

	To prove assertion~\ref{prop:quad:cov_min}, note that the assumptions $D \succeq 0$ and $D \neq 0$ imply that $\dualvar\opt > \lambda_{\max}(D) > 0$. Therefore, none of the constraints in~\eqref{lin_obj:const_refor} are binding at optimality. As the objective function of~\eqref{lin_obj:const_refor} is smooth and strictly convex, $\dualvar\opt$ is thus uniquely determined by the first-order optimality condition~\eqref{eq:quadratic:algebraic-equation}, which forces the gradient of the objective function to vanish. The uniqueness of~$\cov\opt$ follows from the uniqueness of~$\dualvar\opt$ and the uniqueness of the inner maximizer in~\eqref{eq:quad:lagrangian}; see Proposition~\ref{prop:lin_obj:Lag}. Moreover, as $\dualvar\opt>0$, the Gelbrich distance constraint in~\eqref{lin_obj:const_refor} is binding at $\cov\opt_\rho$ due to complementary slackness. Next, we have
	\begin{align*} 
	\frac{1}{\lambda_{\min}(\Sigma\opt)} &=
	\lambda_{\max}\big((\Sigma\opt)^{-1}\big) = \lambda_{\max}\bigg(\Big( {\dualvar\opt}^2(\dualvar\opt I_n - D)^{-1} \covsa (\dualvar\opt I_d - D)^{-1} \Big)^{-1}\bigg) \\
	& = \lambda_{\max}\Big(\big(I_d - D/\dualvar\opt) \covsa^{-1} \big(I_d - D/\dualvar\opt)\Big) \\
	& \le \lambda_{\max}(I_d - D/\dualvar\opt)^2\; \lambda_{\max}(\covsa^{-1}) < \lambda_{\max}(\covsa^{-1}) = \frac{1}{\lambda_{\min}(\covsa)}\,,
	\end{align*}
	where the strict inequality holds because $\dualvar\opt > \lambda_{\max}(D) > 0$. Thus, we conclude that $\lambda_{\min}(\Sigma\opt)>\lambda_{\min}(\covsa)$.

	To in order derive upper and lower bounds on $\dualvar\opt$, we let $D=\sum_{i=1}^d \lambda_i v_iv_i^\top$ be the eigendecomposition of~$D$, where $\lambda_1,\ldots, \lambda_d$ denote the eigenvalues of $D$ indexed in descending order, while $v_1,\ldots, v_d$ represent the corresponding normalized eigenvectors. The left hand side of~\eqref{eq:quadratic:algebraic-equation} can thus be re-expressed~as
	\[
		\rho^2-\sum_{i=1}^d \left( \frac{\lambda_i}{\dualvar-\lambda_i}\right)^2 v_i^\top \covsa v_i
	\]
	This expression is manifestly non-decreasing in $\dualvar\in(\lambda_1,\infty)$. Moreover, the sum admits the simple bounds
	\[
		\left( \frac{\lambda_1}{\dualvar-\lambda_1}\right)^2 v_1^\top \covsa v_1 \leq \sum_{i=1}^d \left( \frac{\lambda_i}{\dualvar-\lambda_i}\right)^2 v_i^\top \covsa v_i
		\leq \left( \frac{\lambda_1}{\dualvar-\lambda_1}\right)^2 \Tr{\covsa}.
	\]
	Equating these lower and upper bounds to $\rho^2$ and solving the resulting equation for $\dualvar$ yields $\underline \dualvar$ and $\overline \dualvar$, respectively. This observation concludes the proof. 
\end{proof}

In Sections~\ref{sect:dual} and~\ref{sect:nash} we repeatedly encounter nonlinear SDPs of the form
	\begin{align} 
	\label{eq:F}
	\begin{array}{cl}
	\Sup{\cov \succeq 0} & \Inf{L \in \mc C} ~ \inner{L^\top L}{\cov} +f(L)\\
	\st & \Tr{\cov + \covsa - 2 \big( \covsa^\half \cov \covsa^\half \big)^\half} \leq \rho^2
	\end{array}
	\end{align}
parameterized by $\covsa\in\PSD^d$ and $\rho \in\mbb R_+$, where $\mc C \subseteq \R^{\ell\times d}$ is a convex set and $f:\mc C \to \R$ a convex continuous function. Problem~\eqref{eq:F} is reminiscent of~\eqref{lin_obj:const} but accommodates a nonlinear convex objective function. We do not attempt to characterize the maximizers of~\eqref{eq:F} for arbitrary choices of $\mc C$ and $f$, but we can prove that there is at least one well-behaved maximizer that is bounded away from~0.

\begin{lemma} [Structural properties of the maximizers of~\eqref{eq:F}]
	\label{lemma:monotone loss}
	Assume that $\covsa\in\PSD^d$ and $\rho \in\mbb R_+$. If $\mc C \subseteq \R^{\ell\times d}$ is a non-empty convex set and $f:\mc C \to \R$ is a convex continuous function, then the nonlinear SDP~\eqref{eq:F} admits a maximizer $\Sigma\opt \succeq \lambda_{\min}(\covsa)I_d$. 
\end{lemma}

\begin{proof}[Proof of Lemma~\ref{lemma:monotone loss}]
	Note that if $\rho = 0$ or $\lambda_{\min}(\covsa) = 0$, then the claim holds trivially. Thus, we may henceforth assume without loss of generality that $\rho > 0$ and $\covsa \succ 0$. Denoting the feasible set of~\eqref{eq:F} by 
	\[
	\mc S = \left\{ \cov \in \PSD^d: \Tr{\cov + \covsa - 2 \big( \covsa^\half \cov \covsa^\half \big)^\half} \leq \rho^2  \right\},
	\]
	we then find
	\begin{align*}
	\Sup{\cov \in \mc S} ~\Inf{L \in \mc C}~ \inner{L^\top L}{\cov} +f(L)
	&=
	\Inf{L \in \mc C} ~\Sup{\cov \in \mc S} ~ \inner{L^\top L}{\cov} + f(L) \\
	&=
	\Inf{L \in \mc C}\Sup{\substack{\cov \in \mc S \\ \cov \succeq \lambda_{\min}(\covsa) I_d}} ~ \inner{L^\top L}{\cov} + f(L) = 
	\Sup{\substack{\cov \in \mc S \\ \cov \succeq \lambda_{\min}(\covsa) I_d}} \Inf{L \in \mc C} ~ \inner{L^\top L}{\cov} + f(L)\,,
	\end{align*}
	where the first and the third equality follow from Sion's minimax theorem~\cite{ref:sion1958minimax}, which applies because~$\inner{L^\top L}{\cov}$ is convex and continuous in $L$ for every fixed $\cov\succeq 0$ and because $\mc S$ is convex and compact by virtue of Lemma~\ref{lemma:compact:FS}. The second equality follows readily from Proposition~\ref{prop:quadratic}\,\ref{prop:quad:cov_min}. The last maximization problem in the above expression has a solution $\Sigma\opt \succeq \lambda_{\min}(\covsa)I_d$ because its feasible set is compact and its objective function is upper semicontinuous. Clearly, $\cov\opt$ also solves~\eqref{eq:F}, and thus the claim follows.
\end{proof}

The proofs of Proposition~\ref{prop:quadratic} and Lemma~\ref{lemma:monotone loss} rely on the following auxiliary lemma, which extends \cite[Lemma~A.6]{ref:shafieezadeh2018wasserstein} to situations where $\covsa$ may be an arbitrary positive semidefinite covariance matrix.

\begin{lemma}[{Compactness of the feasible set}]
   	\label{lemma:compact:FS}
   	For any $\covsa \in \PSD^d$ and $\rho \in \mbb R_+$, the set
   	\[
   	\mc S = \left\{ \cov \in \PSD^d: \Tr{\cov + \covsa - 2 \big( \covsa^\half \cov \covsa^\half \big)^\half} \leq \rho^2  \right\}
   	\]
   	is convex and compact. Moreover, for any $\cov \in \mc S$ we have $\Tr{\cov} \leq (\rho + \mathop{\rm Tr}[\covsa]^{\half})^2$.
\end{lemma}
\begin{proof}[Proof of Lemma~\ref{lemma:compact:FS}]
    The convexity of the feasible set $\mc S$ follows from the convexity of the squared Gelbrich distance proved in Proposition~\ref{prop:gelbrich-convexity}. To prove that $\mc S$ is compact, we recall from \cite[Proposition~2]{ref:malago2018wasserstein} that
    \begin{align*}
        \Tr{\big( \covsa^\half \cov \covsa^\half \big)^\half}& = \max_{C \in \RR^{d \times d}} \left\{\Tr{C}:\begin{bmatrix} \cov & C \\ C^\top &\covsa  \end{bmatrix} \succeq 0\right\}\\
        &\le \max_{C \in \RR^{d \times d}} \left\{\Tr{C}:\DS C_{ij}^2 \leq \cov_{ii} \covsa_{jj} ~ \forall i,j =1, \dots, d\right\} = \sum_{i=1}^d \sqrt{\cov_{ii} \covsa_{ii}} \leq \sqrt{\Tr{\cov} \Tr{\covsa}}\,,
    \end{align*}
    where the first inequality holds because all second principal minors of a positive semidefinite matrix are non-negative, and the second inequality follows from the Cauchy-Schwarz inequality. Thus, any $\cov\in\mc S$ satisfies
    \begin{equation*}
        \rho^2 \geq \Tr{\cov + \covsa - 2 \big( \covsa^\half \cov \covsa^\half \big)^\half)} \geq \left( \sqrt{\Tr{\cov}} - \sqrt{\Tr{\covsa}} \right)^2 ,
    \end{equation*}
    which implies that $\Tr{\cov} \leq (\rho + \Tr{\covsa}^{\half} )^2$. This allows us to conclude that $\mc S$ is bounded. Moreover, $\mc S$ is closed due to the continuity of the matrix square root established in Lemma~\ref{lemma:hoelder}.
\end{proof}

Finally, we derive the second-order Taylor expansion of the objective function
\[
 f(\cov_x, \cov_w) = \Tr{\cov_x - \cov_x H^\top \left( H \cov_x H^\top + \cov_w \right)^{-1} H \cov_x}
\]
of the nonlinear SDP~\eqref{eq:program:dual:concise}, which is needed for the proof of Proposition~\ref{prop:regularity} in the main text.

\begin{lemma}[Gradient and Hessian of~$f$]
	\label{lem:Taylor-f} If $(\cov_x,\cov_w) \in \PD^n\times \PD^m$ and $G= H\cov_x H^\top +\cov_w\in \PD^m$, then
    \begin{align*}
    \begin{array}{l@{\;}l@{\;}l@{\;\;}l}
    \direc_x &= \phantom{-} \nabla_{\Sigma_x} f\big(\Sigma_x, \Sigma_w\big) &=& (I_n - \cov_x H^\top G^{-1}H )^\top (I_n - \cov_x H^\top G^{-1}H)\\
	\direc_w &= \phantom{-} \nabla_{\Sigma_w} f\big(\Sigma_x, \Sigma_w\big) &=& G^{-1} H \cov_x^2 H^\top G^{-1}\\
	\mc H_{xx} & = -\nabla^2_{xx} f(\cov_x, \cov_w) &=& 2 \direc_x \otimes H^\top G^{-1} H \\
	\mc H_{xw} &= -\nabla^2_{xw} f(\cov_x, \cov_w) &=& H^\top G^{-1} \otimes (H^\top \direc_w - \cov_x H^\top G^{-1}) +  (H^\top \direc_w - \cov_x H^\top G^{-1}) \otimes H^\top G^{-1} \\
	\mc H_{ww} & = -\nabla^2_{ww} f(\cov_x, \cov_w) &=& 2 \direc_w \otimes G^{-1},
	\end{array}
	\end{align*}
	where $\nabla_x$ and $\nabla_w$ stand for the nabla operators with respect to $\vect (\cov_x)$ and $\vect (\cov_w)$, respectively.
\end{lemma}	

\begin{proof}[Proof of Lemma~\ref{lem:Taylor-f}]
	We first derive the second-order Taylor expansion of $G^{-1}$. Specifically, if~$\Delta_x \in \mbb S^n$ and~$\Delta_w \in \mbb S^m$ represent symmetric perturbation directions of $\cov_x$ and $\cov_w$, respectively, then we have
	\begin{align*}
	&\left( H \left[\cov_x + t\Delta_x\right] H^\top + \left[\cov_w + t\Delta_w\right] \right)^{-1} \\
	=\,&\left(G + t\left[H\Delta_x H^\top +\Delta_w\right]\right)^{-1}\\
	=\,& G^{-\half} (I_m + t\,G^{-\half} \left[H \Delta_x H^\top + \Delta_w\right] G^{-\half})^{-1} G^{-\half}  \\
	\begin{split}
	=\,& G^{-\half} \left(I_m - t\,G^{-\half} \left[H \Delta_x H^\top + \Delta_w\right] G^{-\half} + t^2 \left(G^{-\half} \left[H \Delta_x H^\top + \Delta_w\right] G^{-\half}\right)^2 + \mc O(|t|^3) \right) G^{-\half}  
	\end{split} \\
	\begin{split}
	=\,& G^{-1} - t\,G^{-1} \left[H \Delta_x H^\top + \Delta_w\right] G^{-1} + t^2\, G^{-1} \left[H \Delta_x H^\top + \Delta_w\right] G^{-1} \left[H \Delta_x H^\top + \Delta_w\right] G^{-1} + \mc O(| t|^3)\,,
	\end{split}
	\end{align*}
	where the third equality follows from a Neumann series expansion~\cite[Proposition~9.4.13]{ref:bernstein2009matrix}. Thanks to \cite[Fact~7.4.9]{ref:bernstein2009matrix}, the second-order Taylor expansion of $f$ can thus be represented as
	\begin{align*}
	&f(\cov_x + t\,\Delta_x, \cov_w + t\,\Delta_w)\\ 
	=\,& \Tr{ \left[\cov_x + t\,\Delta_x\right] - \left[\cov_x + t\,\Delta_x\right] H^\top \left( H \left[\cov_x + t\,\Delta_x\right] H^\top + \left[\cov_w + t\,\Delta_w\right] \right)^{-1}H \left[\cov_x + t\,\Delta_x\right] }\\
	=\,& f(\cov_x, \cov_w) + t\,\inner{\direc_x}{\Delta_x} + t\,\inner{\direc_w}{\Delta_w}  - \frac{t^2}{2} \begin{pmatrix} \vect(\Delta_x) \\ \vect(\Delta_w) \end{pmatrix}^\top \begin{pmatrix}\mc H_{xx} & \mc H_{xw} \\
	\mc H_{xw}^\top & \mc H_{ww}\end{pmatrix} \begin{pmatrix} \vect (\Delta_x) \\ \vect (\Delta_w) \end{pmatrix} + \mc O(|t|^3)\,,
	\end{align*}
	where the matrices $\direc_x$, $\direc_w$, $\mc H_{xx}$, $\mc H_{xw}$ and $\mc H_{ww}$ are defined as in the statement of the lemma.
\end{proof}

\bibliographystyle{siam}
\bibliography{bibliography}

\begin{thebibliography}{10}

\bibitem{ref:aubinfrankowska1990}
{\sc J.-P. Aubin and H.~Frankowska}, {\em Set-Valued Analysis}, Birkh\"auser,
  1990.

\bibitem{ref:aviv2003time}
{\sc Y.~Aviv}, {\em A time-series framework for supply-chain inventory
  management}, Operations Research, 51 (2003), pp.~210--227.

\bibitem{ref:beck2006robust}
{\sc A.~Beck, A.~Ben-Tal, and Y.~C. Eldar}, {\em Robust mean-squared error
  estimation of multiple signals in linear systems affected by model and noise
  uncertainties}, Mathematical Programming, 107 (2006), pp.~155--187.

\bibitem{ref:beck2007regularization}
{\sc A.~Beck and Y.~C. Eldar}, {\em Regularization in regression with bounded
  noise: {A} {C}hebyshev center approach}, SIAM Journal on Matrix Analysis and
  Applications, 29 (2007), pp.~606--625.

\bibitem{ref:beck2007mean}
{\sc A.~Beck, Y.~C. Eldar, and A.~Ben-Tal}, {\em Mean-squared error estimation
  for linear systems with block circulant uncertainty}, SIAM Journal on Matrix
  Analysis and Applications, 29 (2007), pp.~712--730.

\bibitem{ref:berge1963topological}
{\sc C.~Berge}, {\em Topological Spaces: Including a Treatment of Multi-Valued
  Functions, Vector Spaces, and Convexity}, Courier Corporation, 1963.

\bibitem{ref:bernstein2009matrix}
{\sc D.~S. Bernstein}, {\em Matrix Mathematics: Theory, Facts, and Formulas},
  Princeton University Press, 2009.

\bibitem{ref:bertsekas2009convex}
{\sc D.~Bertsekas}, {\em Convex Optimization Theory}, Athena Scientific, 2009.

\bibitem{ref:bhatia2018strong}
{\sc R.~Bhatia, T.~Jain, and Y.~Lim}, {\em Strong convexity of sandwiched
  entropies and related optimization problems}, Reviews in Mathematical
  Physics,  (2018).

\bibitem{ref:boyd2004convex}
{\sc S.~Boyd and L.~Vandenberghe}, {\em Convex Optimization}, Cambridge
  University Press, 2004.

\bibitem{ref:cambanis1981theory}
{\sc S.~Cambanis, S.~Huang, and G.~Simons}, {\em On the theory of elliptically
  contoured distributions}, Journal of Multivariate Analysis, 11 (1981),
  pp.~368--385.

\bibitem{ref:chang2000adaptive}
{\sc S.~G. Chang, B.~Yu, and M.~Vetterli}, {\em Adaptive wavelet thresholding
  for image denoising and compression}, IEEE Transactions on Image Processing,
  9 (2000), pp.~1532--1546.

\bibitem{ref:chatfield2016analysis}
{\sc C.~Chatfield}, {\em The Analysis of Time Series: {A}n Introduction}, CRC
  Press, 2016.

\bibitem{ref:clarkson2010coresets}
{\sc K.~L. Clarkson}, {\em Coresets, sparse greedy approximation, and the
  frank-wolfe algorithm}, ACM Transactions on Algorithms, 6 (2010), p.~63.

\bibitem{ref:cover2006elements}
{\sc T.~M. Cover and J.~A. Thomas}, {\em Elements of Information Theory},
  Wiley-Interscience, 2006.

\bibitem{ref:cuturi2013sinkhorn}
{\sc M.~Cuturi}, {\em Sinkhorn distances: Lightspeed computation of optimal
  transport}, in Advances in Neural Information Processing Systems 26, 2013,
  pp.~2292--2300.

\bibitem{ref:demyanov1970approximate}
{\sc V.~F. Demyanov and A.~M. Rubinov}, {\em Approximate Methods in
  Optimization Problems}, American Elsevier Publishing, 1970.

\bibitem{ref:diggavi2001worst}
{\sc S.~N. Diggavi and T.~M. Cover}, {\em The worst additive noise under a
  covariance constraint}, IEEE Transactions on Information Theory, 47 (2001),
  pp.~3072--3081.

\bibitem{ref:dowson1982frechet}
{\sc D.~C. Dowson and B.~V. Landau}, {\em The {F}r{\'e}chet distance between
  multivariate normal distributions}, Journal of Multivariate Analysis, 12
  (1982), pp.~450--455.

\bibitem{ref:dunn1979rates}
{\sc J.~C. Dunn}, {\em Rates of convergence for conditional gradient algorithms
  near singular and nonsingular extremals}, SIAM Journal on Control and
  Optimization, 17 (1979), pp.~187--211.

\bibitem{ref:dunn1980convergence}
\leavevmode\vrule height 2pt depth -1.6pt width 23pt, {\em Convergence rates
  for conditional gradient sequences generated by implicit step length rules},
  SIAM Journal on Control and Optimization, 18 (1980), pp.~473--487.

\bibitem{ref:dunn1978conditional}
{\sc J.~C. Dunn and S.~Harshbarger}, {\em Conditional gradient algorithms with
  open loop step size rules}, Journal of Mathematical Analysis and
  Applications, 62 (1978), pp.~432--444.

\bibitem{ref:eldar2006robust}
{\sc Y.~C. Eldar}, {\em Robust competitive estimation with signal and noise
  covariance uncertainties}, IEEE Transactions on Information Theory, 52
  (2006), pp.~4532--4547.

\bibitem{ref:eldar2008minimax}
{\sc Y.~C. Eldar, A.~Beck, and M.~Teboulle}, {\em A minimax {C}hebyshev
  estimator for bounded error estimation}, IEEE Transactions on Signal
  Processing, 56 (2008), pp.~1388--1397.

\bibitem{ref:eldar2004linear}
{\sc Y.~C. Eldar, A.~Ben-Tal, and A.~Nemirovski}, {\em Linear minimax regret
  estimation of deterministic parameters with bounded data uncertainties}, IEEE
  Transactions on Signal Processing, 52 (2004), pp.~2177--2188.

\bibitem{ref:eldar2004robust}
\leavevmode\vrule height 2pt depth -1.6pt width 23pt, {\em Robust mean-squared
  error estimation in the presence of model uncertainties}, IEEE Transactions
  on Signal Processing, 53 (2004), pp.~168--181.

\bibitem{ref:eldar2004competitive}
{\sc Y.~C. Eldar and N.~Merhav}, {\em A competitive minimax approach to robust
  estimation of random parameters}, IEEE Transactions on Signal Processing, 52
  (2004), pp.~1931--1946.

\bibitem{ref:fang1990symmetric}
{\sc K.-T. Fang, S.~Kotz, and K.~W. Ng}, {\em Symmetric Multivariate and
  Related Distributions}, Chapman \& Hall, 1990.

\bibitem{ref:frank1956algorithm}
{\sc M.~Frank and P.~Wolfe}, {\em An algorithm for quadratic programming},
  Naval Research Logistics, 3 (1956), pp.~95--110.

\bibitem{ref:freund2016new}
{\sc R.~M. Freund and P.~Grigas}, {\em New analysis and results for the
  {F}rank--{W}olfe method}, Mathematical Programming, 155 (2016), pp.~199--230.

\bibitem{ref:garber2015faster}
{\sc D.~Garber and E.~Hazan}, {\em Faster rates for the {F}rank--{W}olfe method
  over strongly-convex sets}, in International Conference on Machine Learning,
  2015, pp.~541--549.

\bibitem{ref:gelbrich1990formula}
{\sc M.~Gelbrich}, {\em On a formula for the ${L}^2$ {W}asserstein metric
  between measures on {E}uclidean and {H}ilbert spaces}, Mathematische
  Nachrichten, 147 (1990), pp.~185--203.

\bibitem{ref:givens1984class}
{\sc C.~R. Givens and R.~M. Shortt}, {\em A class of {W}asserstein metrics for
  probability distributions}, The Michigan Mathematical Journal, 31 (1984),
  pp.~231--240.

\bibitem{ref:golnaraghi2017automatic}
{\sc F.~Golnaraghi and B.~Kuo}, {\em Automatic Control Systems}, McGraw-Hill
  Education, 2017.

\bibitem{ref:guo2011estimation}
{\sc D.~Guo, Y.~Wu, S.~S. Shitz, and S.~Verdu}, {\em Estimation in {G}aussian
  noise: {P}roperties of the minimum mean-square error}, IEEE Transactions on
  Information Theory, 57 (2011), pp.~2371--2385.

\bibitem{ref:hamilton1994time}
{\sc J.~Hamilton}, {\em Time Series Analysis}, Princeton University Press,
  1994.

\bibitem{ref:hespanha2009linear}
{\sc J.~P. Hespanha}, {\em Linear Systems Theory}, Princeton University Press,
  2009.

\bibitem{ref:hult2002advances}
{\sc H.~Hult and F.~Lindskog}, {\em Multivariate extremes, aggregation and
  dependence in elliptical distributions}, Advances in Applied Probability, 34
  (2002), pp.~587--608.

\bibitem{ref:jaggi2013revisiting}
{\sc M.~Jaggi}, {\em Revisiting {Frank--Wolfe}: Projection-free sparse convex
  optimization}, in Proceedings of the 30th International Conference on Machine
  Learning, 2013, pp.~427--435.

\bibitem{ref:jondeau2007financial}
{\sc E.~Jondeau, S.~Poon, and M.~Rockinger}, {\em Financial Modeling under
  Non-Gaussian Distributions}, Springer, 2007.

\bibitem{ref:journee2010generalized}
{\sc M.~Journ{\'e}e, Y.~Nesterov, P.~Richt{\'a}rik, and R.~Sepulchre}, {\em
  Generalized power method for sparse principal component analysis}, Journal of
  Machine Learning Research, 11 (2010), pp.~517--553.

\bibitem{ref:juditsky2018lectures}
{\sc A.~Juditsky and A.~Nemirovski}, {\em Lectures on statistical inferences
  via convex optimization}, 2018.

\bibitem{ref:juditsky2018nearoptimality}
{\sc A.~Juditsky, A.~Nemirovski, et~al.}, {\em Near-optimality of linear
  recovery in {G}aussian observation scheme under $\vert \cdot
  \vert_{2}^{2}$-loss}, The Annals of Statistics, 46 (2018), pp.~1603--1629.

\bibitem{ref:kay1993fundamentals}
{\sc S.~M. Kay}, {\em Fundamentals of Statistical Signal Processing: Estimation
  Theory}, Prentice Hall, 1993.

\bibitem{ref:Kelker-1970}
{\sc D.~Kelker}, {\em Distribution theory of spherical distributions and a
  location-scale parameter generalization}, Sankhy{\=a}: The Indian Journal of
  Statistics, Series A,  (1970), pp.~419--430.

\bibitem{ref:knott1984optimal}
{\sc M.~Knott and C.~S. Smith}, {\em On the optimal mapping of distributions},
  Journal of Optimization Theory and Applications, 43 (1984), pp.~39--49.

\bibitem{ref:levitin1966constrained}
{\sc E.~S. Levitin and B.~T. Polyak}, {\em Constrained minimization methods},
  USSR Computational Mathematics and Mathematical Physics, 6 (1966), pp.~1--50.

\bibitem{ref:levy2004robust}
{\sc B.~C. Levy and R.~Nikoukhah}, {\em Robust least-squares estimation with a
  relative entropy constraint}, IEEE Transactions on Information Theory, 50
  (2004), pp.~89--104.

\bibitem{ref:levy2012robust}
\leavevmode\vrule height 2pt depth -1.6pt width 23pt, {\em Robust state space
  filtering under incremental model perturbations subject to a relative entropy
  tolerance}, IEEE Transactions on Automatic Control, 58 (2012), pp.~682--695.

\bibitem{ref:lofberg2004YALMIP}
{\sc J.~L\"ofberg}, {\em {YALMIP}: {A} toolbox for modeling and optimization in
  {MATLAB}}, in IEEE International Conference on Robotics and Automation, 2004,
  pp.~284--289.

\bibitem{ref:mackay2003information}
{\sc D.~MacKay}, {\em Information Theory, Inference and Learning Algorithms},
  Cambridge University Press, 2003.

\bibitem{ref:malago2018wasserstein}
{\sc L.~Malag{\`o}, L.~Montrucchio, and G.~Pistone}, {\em {Wasserstein
  Riemannian geometry of Gaussian densities}}, Information Geometry, 1 (2018),
  pp.~137--179.

\bibitem{ref:murphy2012machine}
{\sc K.~Murphy}, {\em Machine Learning: A Probabilistic Perspective}, MIT
  Press, 2012.

\bibitem{ref:nesterov2018complexity}
{\sc Y.~Nesterov}, {\em Complexity bounds for primal-dual methods minimizing
  the model of objective function}, Mathematical Programming, 171 (2018),
  pp.~311--330.

\bibitem{ref:nguyen2018distributionally}
{\sc V.~A. Nguyen, D.~Kuhn, and P.~Mohajerin~Esfahani}, {\em Distributionally
  robust inverse covariance estimation: The {W}asserstein shrinkage estimator},
  arXiv preprint arXiv:1805.07194,  (2018).

\bibitem{ref:ogata2009modern}
{\sc K.~Ogata}, {\em Modern Control Engineering}, Pearson, 2009.

\bibitem{ref:olkin1982distance}
{\sc I.~Olkin and F.~Pukelsheim}, {\em The distance between two random vectors
  with given dispersion matrices}, Linear Algebra and its Applications, 48
  (1982), pp.~257--263.

\bibitem{ref:oppenheim2015signals}
{\sc A.~V. Oppenheim and G.~C. Verghese}, {\em Signals, Systems and Inference},
  Pearson, 2015.

\bibitem{ref:pedregosa2018stepsize}
{\sc F.~Pedregosa, G.~Negiar, A.~Askari, and M.~Jaggi}, {\em Linearly
  convergent {Frank-Wolfe} with backtracking line-search}, in International
  Conference on Artificial Intelligence and Statistics, 2020, pp.~1--10.

\bibitem{ref:peyre2019computational}
{\sc G.~Peyr{\'e} and M.~Cuturi}, {\em Computational optimal transport},
  Foundations and Trends{\textregistered} in Machine Learning, 11 (2019),
  pp.~355--607.

\bibitem{ref:posekany2011biological}
{\sc A.~Posekany, K.~Felsenstein, and P.~Sykacek}, {\em Biological assessment
  of robust noise models in microarray data analysis}, Bioinformatics, 27
  (2011), pp.~807--814.

\bibitem{ref:rockafellar2010variational}
{\sc R.~T. Rockafellar and R.~J.-B. Wets}, {\em Variational Analysis},
  Springer, 2010.

\bibitem{ref:rubel2017robust}
{\sc O.~Rubel and P.~A. Naik}, {\em Robust dynamic estimation}, Marketing
  Science, 36 (2017), pp.~453--467.

\bibitem{ref:ruttimann1998statistical}
{\sc U.~E. Ruttimann, M.~Unser, R.~R. Rawlings, D.~Rio, N.~F. Ramsey, V.~S.
  Mattay, D.~W. Hommer, J.~A. Frank, and D.~R. Weinberger}, {\em Statistical
  analysis of functional {MRI} data in the wavelet domain}, IEEE Transactions
  on Medical Imaging, 17 (1998), pp.~142--154.

\bibitem{ref:schmitt1992perturbation}
{\sc B.~A. Schmitt}, {\em Perturbation bounds for matrix square roots and
  {P}ythagorean sums}, Linear Algebra and its Applications, 174 (1992),
  pp.~215--227.

\bibitem{ref:shafieezadeh2019mass-transportation}
{\sc S.~Shafieezadeh-Abadeh, P.~{Mohajerin Esfahani}, and D.~Kuhn}, {\em
  Regularization via mass transportation}, Journal of Machine Learning
  Research, 20 (2019), pp.~1--68.

\bibitem{ref:shafieezadeh2018wasserstein}
{\sc S.~Shafieezadeh-Abadeh, V.~Nguyen, D.~Kuhn, and P.~{Mohajerin Esfahani}},
  {\em {W}asserstein distributionally robust {K}alman filtering}, in Advances
  in Neural Information Processing Systems 31, 2018, pp.~8483--8492.

\bibitem{ref:sion1958minimax}
{\sc M.~Sion}, {\em On general minimax theorems}, Pacific Journal of
  Mathematics, 8 (1958), pp.~171--176.

\bibitem{ref:solomon2015convolutional}
{\sc J.~Solomon, F.~D. Goes, G.~Peyr{\'e}, M.~Cuturi, A.~Butscher, A.~Nguyen,
  T.~Du, and L.~Guibas}, {\em Convolutional {W}asserstein distances: Efficient
  optimal transportation on geometric domains}, ACM Transactions on Graphics,
  34 (2015), p.~66.

\bibitem{ref:sriram2007optimal}
{\sc S.~Sriram and M.~U. Kalwani}, {\em Optimal advertising and promotion
  budgets in dynamic markets with brand equity as a mediating variable},
  Management Science, 53 (2007), pp.~46--60.

\bibitem{ref:stock2015introduction}
{\sc J.~Stock and M.~Watson}, {\em Introduction to Econometrics}, Prentice
  Hall, 2015.

\bibitem{ref:taskesen2019complexity}
{\sc B.~Ta{\c s}kesen, S.~Shafieezadeh-Abadeh, and D.~Kuhn}, {\em A unified
  approach for solving optimal transport problem}, Working paper,  (2019).

\bibitem{ref:lint2012applications}
{\sc H.~van Lint and T.~Djukic}, {\em Applications of {K}alman filtering in
  traffic management and control}, in New Directions in Informatics,
  Optimization, Logistics, and Production, INFORMS, 2012, ch.~4, pp.~59--91.

\bibitem{ref:wooldridge2010econometric}
{\sc J.~Wooldridge}, {\em Econometric Analysis of Cross Section and Panel
  Data}, MIT Press, 2010.

\bibitem{ref:zhou1995}
{\sc J.~Zhou}, {\em On the existence of equilibrium for abstract economies},
  Journal of Mathematical Analysis and Applications, 193 (1995), pp.~839--858.

\bibitem{ref:zorzi2016robust}
{\sc M.~Zorzi}, {\em Robust {K}alman filtering under model perturbations}, IEEE
  Transactions on Automatic Control, 62 (2016), pp.~2902--2907.

\bibitem{ref:zorzi2017robustness}
\leavevmode\vrule height 2pt depth -1.6pt width 23pt, {\em On the robustness of
  the {B}ayes and {W}iener estimators under model uncertainty}, Automatica, 83
  (2017), pp.~133--140.

\end{thebibliography}
	
\end{document}